\documentclass[12pt]{amsart}
\usepackage{amsmath,amsfonts,amsthm,amscd,textcomp,times, amssymb}
\usepackage[all,cmtip]{xy}
\usepackage[T1]{fontenc}
\usepackage{calligra}
\usepackage{multido}
\usepackage{frcursive}
\usepackage{mathrsfs}
\usepackage{tikz}
\usetikzlibrary{shapes,arrows}
\newtheorem{theorem}{Theorem}
\newtheorem{conjecture}{Conjecture}
\newtheorem{claim}{Claim}
\newtheorem{problem}{Problem}
\newtheorem{proposition}{Proposition}

\newtheorem{lemma}{Lemma}
\newtheorem{definition}{Definition}
\newtheorem{example}{Example}
\newtheorem{corollary}{Corollary}
\newtheorem{remark}{Remark}
\numberwithin{equation}{section}
\numberwithin{remark}{section}
\numberwithin{theorem}{section}
\numberwithin{proposition}{section}
\numberwithin{definition}{section}
\numberwithin{lemma}{section}
\numberwithin{claim}{section}
\numberwithin{corollary}{section}
\numberwithin{conjecture}{section}
\DeclareMathOperator{\rnk}{rnk}
\newcommand{\bull}{\ensuremath{{}\bullet{}}}
 
\newcommand{\cpn}{\ensuremath{\mathbb{P}^{N}}}
\newcommand{\slnc}{\ensuremath{SL(N+1,\mathbb{C})}}
\newcommand{\dlb}{\ensuremath{\overline{\partial}}}
\newcommand{\dl}{\ensuremath{\partial}}
\newcommand{\ra}{\ensuremath{\longrightarrow}}
\newcommand{\om}{\ensuremath{\omega}}
\newcommand{\vp}{\ensuremath{\varphi}}
\newcommand{\vps}{\ensuremath{\varphi_{\sigma}}}

\newcommand{\vplt}{\ensuremath{\varphi_{\lambda(t)}}}

\newcommand{\cn}{\ensuremath{\mathbb{C}^{N+1}}}

\newcommand{\xhyp}{\ensuremath{X\times\mathbb{P}^{n-1}}}

\newcommand{\lambull}{\ensuremath{\lambda_{\bull}} }
\newcommand{\mubull}{\ensuremath{\mu_{\bull}} }
\newcommand{\elam}{\ensuremath{\mathbb{E}_{\lambda_{\bull}}}}
\newcommand{\emu}{\ensuremath{\mathbb{E}_{\mu_{\bull}}}}

\begin{document}
\bibliographystyle{alpha}
\title[CM stability   ]{CM stability of projective varieties}
\author{Sean Timothy Paul}
\email{stpaul@math.wisc.edu}
\address{Mathematics Dept. Univ. of Wisconsin, Madison}
\subjclass[2000]{53C55}
\keywords{Discriminants, Resultants, K-Energy maps, Projective duality, Geometric Invariant Theory, K\"ahler Einstein metric, Stability .}
\date{June 22, 2012}
 \vspace{-5mm}
\begin{abstract}{Let $G$ denote the complex special linear group, or more generally a complex reductive linear algebraic group .  We consider  two a priori unrelated problems: \\
\ \\
 1)  Given a $G$ equivariant map $\pi:G\cdot [v]\ra G\cdot [w]$ between two orbits in  projectivizations of finite dimensional complex rational representations of $G$ provide a necessary and sufficient condition for $\pi$ to extend to a {regular} map between the Zariski closures of those orbits.  Provide a ``polyhedral-combinatorial'' characterization of the existence of an extension.\\
\ \\
 2) Given a smooth complex projective variety $X\ra \cpn$ provide a necessary and sufficient condition in terms of the geometry of the embedding which insures that the Mabuchi energy of $(X ,\om_{FS}|_X )$ is bounded from below on the space of Bergman metrics. Provide a ``polyhedral-combinatorial'' characterization of the existence of a lower bound.    \\
  \ \\
 The second problem is essentially a special and highly non-trivial case of the first problem. The purpose of this article is to study the incredibly rich structures associated with the  first problem and then to deduce several new results related to the second problem: We provide a new definition of the CM-polarization as a pair of invertible sheaves $(\mathscr{L}_R, \mathscr{L}_{\Delta})$  on the appropriate Hilbert scheme. We prove that this pair is globally generated.  We introduce the $(\mathscr{L}_R, \mathscr{L}_{\Delta})$-semistable locus in the Hilbert scheme and prove that this locus is a finite union of locally closed subschemes. We prove that a point $[X]$ in the Hilbert scheme is semistable with respect to $(\mathscr{L}_R, \mathscr{L}_{\Delta})$ if and only if the Mabuchi energy of  $(X , {\om_{FS}}|_X)$ is bounded below on the space $\mathcal{B}$ of Bergman metrics . We also establish an analytic numerical criterion for a lower bound : the Mabuchi energy of  $(X , {\om_{FS}}|_X)$ is bounded below on $\mathcal{B}$ if and only if it is bounded below on all analytic {curves} in $\mathcal{B}$.
We prove that discriminants commute with smooth limit cycle formation.   We provide generalizations of the classical Calabi-Futaki character, the generalized Futaki character, and the Mabuchi energy. We explain how these notions arise naturally in the finite dimensional representation theory of complex linear algebraic groups.}
  
 \end{abstract}
\maketitle
\setcounter{tocdepth}{1}
\tableofcontents  
 \section{Introduction  }

 One of the main problems in complex geometry is to detect the existence of a
canonical  K\"ahler metric in a given K\"ahler class on a compact K\"ahler 
manifold $(X,\om)$ (see \cite{yau78}, \cite{aubin76}) . In particular, an outstanding problem in the field is to find necessary and sufficient conditions for the existence of a K\"ahler Einstein metric on a {Fano manifold}, or more generally a constant scalar curvature metric in the class $c_1(L)$ of any (very)ample divisor $L$ on $X$. This problem seems extremely difficult and has led to a striking series of 
conjectures, which we prefer to call the \emph{standard conjectures}, which relate these special 
metrics  to the algebraic geometry of the associated projective models of the 
manifold $X$.  Yau speculated that the relevant algebraic geometry would be related  
to Mumford's Geometric Invariant Theory. Motivated by Yau's suggestion  Tian introduced the notions of CM/K \footnote{``CM'' stands for Chow-Mumford and ``K'' refers to the K-energy of Mabuchi.}(semi)stability  (see \cite{tian94} , \cite{tian97}).  Tian was led to \footnote{ This strategy is already completely explicit in \cite{tian1990} .} these stability conditions  through the remarkable idea of restricting the Mabuchi energy $\nu_{\om}$  to the group $G:=SL(H^0(X, L)^{\vee})$ or the space $\Delta(G)$ of it's algebraic one parameter subgroups (see \cite{dingtian}). The problem was to establish the following : 
\begin{align}\label{coreproblem}
\begin{split}
& i) \quad {\nu_{\om}}|_G>-\infty \iff \ \mbox{$X\ra \mathbb{P}^{N}$ is CM-semistable.} \\
\ \\
& ii) \quad {\nu_{\om}}|_{\Delta(G)}>-\infty \iff \mbox{$X\ra \mathbb{P}^{N}$ is K-semistable.}
\end{split}
\end{align}
Part of the problem is to \emph{provide the definitions of (semi)stability}. 
The difficulty here is to understand the behavior of the Mabuchi energy restricted to the spaces on the left and then convert this understanding into meaningful algebro-geometric conditions  on the embedding  $X\ra \mathbb{P}^{N}$. 
Due to this difficulty Tian's original definition of K-stability has been modified and extended in various directions by many authors. Currently, the most widely known reformulation is due to Donaldson (see \cite{skdtoric}). Unfortunately neither this reformulation nor its variants  resolve the problem of Mabuchi energy control along algebraic degenerations in the space of Bergman metrics, let alone the behavior of the energy on the entire space. With the exception of Lu (see \cite{lu2004}), this issue seems to have been neglected by most researchers in the field. Recently the author has given complete solutions to problems $i)$ and $ii)$ for a fixed subvariety $X$ in $\cpn$ (see \cite{paul2011}).  The purpose of this article to analyze the relative case and develop the entire theory in the broader context of equivariant embeddings of algebraic homogeneous spaces and the representation theory of reductive algebraic groups. The main contributions in this paper are Theorems \ref{globgen}, \ref{uniform1} , \ref{weakII}, and the table at the end of Section \ref{generaltheory} .  

\subsection{CM (semi)stability}
In order to better appreciate the results of this paper, we first discuss Tian's approach towards CM-stability and the standard conjectures (see \cite{tian97} section 8 for further details).     
To begin we consider a family $\mathbb{X}\overset{\pi}{\ra} Y$ of polarized varieties satisfying the following conditions :  
 \begin{enumerate}
  \item $\mathbb{X}$ and $Y$ are smooth varieties such that  $\mathbb{X}\subset Y\times\cpn $ .\\ 
  
  \item $\pi:={p_1}|_{\mathbb{X}}: \mathbb{X}\rightarrow Y$ is flat of relative dimension $n$, degree $d$ with Hilbert polynomial $P$.\\
  \item Let $L:=p_2^*\mathcal{O}(1)|_{\mathbb{X}}$. Then $L|_{ {X}_y}$ is very ample and the embedding ${ {X}_y}:=\pi^{-1}(y)\overset{L}{ \rightarrow} \cpn$ is given by a complete linear system for $y\in Y$.\\
\item There is an action of $G:=\slnc$ on the data compatible with the projection and the standard action on $\cpn$.
  \end{enumerate}
 We let $K_{\pi}$ denote the relative canonical bundle :
 \begin{align}
 K_{\pi}:=  K_{\mathbb{X}}\otimes \pi^*K^{-1}_{Y}\ , \ K_{\pi}|_{X_y}\cong K_{X_y}\ \mbox{ for all}\ y\in Y\setminus \mathscr{D}\  .
 \end{align}
 $\mathscr{D}$ denotes the discriminant locus of the map $\pi$. 
 
 Let $\mu$ be the subdominant coefficient of $P$ (the Hilbert polynomial of the family) .  Tian introduces the following virtual bundle over $\mathbb{X}$ :
  \begin{align}
  \mathcal{E}_{\pi}:=(n+1)(K_{\pi}-K_{\pi}^{-1})(L-L^{-1})^n-d\mu (L-L^{-1})^{n+1} \ .
  \end{align}
  Tian's definition of CM (semi)stability is formulated in terms of the following linearization on $Y$. 
 \begin{definition} (Tian \cite{tian97}) \emph{The \textbf{\emph{CM polarization}} of the family $\mathbb{X}\rightarrow Y$ is the line bundle
  \begin{align*} 
\mathscr{L}_{cm}:= \mbox{det}R\pi_*(\mathcal{E}_{\pi})^{-1} \in \mbox{Pic}^G(Y)\  .
 \end{align*}}
 \end{definition}
\begin{remark}
\emph{The CM polarization is not necessarily positive. We will discuss this issue in detail below.}
\end{remark}

 Given $y\in Y$ we define  $e_y$ to be any (non-zero) lift of $y$ to ${\mathscr{L}_{cm}^{-1}}|_y$ . 
 \begin{definition}(Tian \cite{tian97}) \emph{ $y \in Y$ is CM semistable if and only if
 \begin{align}\label{cmstab97}
 \overline{G\cdot e_y}\cap \{\mbox{ zero section  of $\mathscr{L}^{-1}_{cm}$}\}=\emptyset \ .
 \end{align}
}
\end{definition}
  
  In the statement of Theorem \ref{tian97} below $c$ is a constant which depends only on the choice of background K\"ahler metrics on $\mathbb{X}$ and $Y$ , $\omega_y:= \omega_{FS} |_{{X}_y}$ denotes the restriction of the Fubini Study form of $\cpn$ to the fiber $ {X}_y$. We let  $\nu_{\omega_{y}}$ denote the Mabuchi energy of $ {X}_y$ , and $\vps$ denotes the Bergman potential corresponding to $\sigma \in G$. For the convenience of the reader these notions are recalled in detail in section \ref{discussionofresults} below. The following result is due to Tian.
\begin{theorem}\label{tian97} (Tian \cite{tian97} section 8, pg. 34 (8.15)) 
\emph{Assume $Y$ is complete.  Let $h$ be a smooth Hermitian metric on $\mathscr{L}_{cm}^{-1}$ .  
 Then there is a continuous function $\Psi :Y \setminus \mathscr{D} \rightarrow (-\infty, \ c)$
 such that for all $y\in Y\setminus \mathscr{D}$ and all $\sigma \in G$ the following identity holds
 \begin{align}\label{singnrm}
\begin{split}
d(n+1)\nu_{\omega|_{{X}_y}}(\vp_\sigma)=   {\Psi (\sigma y)}+ \log\frac{||\sigma\cdot e_y ||_h^{2} }{||e_y ||_h^{2}} \ .
\end{split}
\end{align}
In particular for all $\sigma\in G$ there is an inequality  
\begin{align}\label{cminequality}
d(n+1)\nu_{\omega|_{{X}_y}}(\vp_\sigma)\leq \log\frac{||\sigma\cdot e_y ||_h^{2} }{||e_y ||_h^{2}}+c\ .
\end{align}}   
 \end{theorem}

\begin{corollary} (Tian \cite{tian97}) \emph{Assume that the Mabuchi energy of the fiber $X_y$ is bounded below. Then $y$ is CM semistable. In particular $y\in Y\setminus \mathscr{D}$ is CM-semistable whenever it admits a canonical metric in the class $\omega|_{{X}_y}$ .}
\end{corollary}

 In contrast to K-(semi)stability CM (semi)stability and the linearization $\mathscr{L}_{cm}$ have remained in the background.  There are, in the opinion of the author, three reasons for this.

The first reason is that $\mathscr{L}_{cm}$ seems to have no convincing positivity properties. Positivity is necessary in order to make contact with classical invariant theory. More precisely given a flat $G$ equivariant family $\mathbb{X}\ra Y$ of polarized manifolds and the corresponding linearization $\mathscr{L}_{cm}\in \mbox{Pic}^G(Y)$ we would like to say that $y\in Y^{ss}(\mathscr{L}_{cm})$ if and only if $0\notin\overline {G\cdot w(y)}$, where $w(y)$ is a point in the dual space of a suitable finite dimensional $G$ invariant base point free subspace\footnote{When $Y$ is complete, we can take $\mathbb{W}=H^0(Y\ ,\ \mathscr{L}_{cm})$ .} $\mathbb{W}$ of $H^0(Y\ ,\ \mathscr{L}_{cm})$ .  Unfortunately it seems that in the majority of examples $\mathscr{L}_{cm}$ has no sections at all and this is why semistability with respect to $\mathscr{L}_{cm}$ is defined by (\ref{cmstab97}) . In the same vein one would like a \emph{geometric interpretation} for the lift $e_y$. For example in \cite{tian94} Tian shows that for the universal family of hypersurfaces of fixed degree in a projective space $e_y$ is just a power of the polynomial corresponding to $y$.  Therefore in this special case the lift $e_y$  is essentially the Cayley-Chow form of $X_y$. We remark that in this case $\mathscr{L}_{cm}$ is ample.
   
The second reason is that even in those rare situations where $\mathscr{L}_{cm}$ is actually positive, semistability with respect to $\mathscr{L}_{cm}$ does {not coincide} with Mabuchi energy lower bounds along the Bergman metrics. This is due to the appearance of the so-called singular term $\Psi$ in (\ref{singnrm}) (see \cite{tian97} section 8 Lemma 8.5 and corollary 8.6 for the definition of $\Psi$).

The third reason is that CM stability seems to require the base of the family $\mathbb{X}\ra Y$ to be complete. In the majority of interesting examples completeness fails. 

Our explanation for these defects is simple but rather jarring:  CM stability should not have been formulated in terms of a linearization. In this paper we formulate an improved version of CM semistability in terms of a \emph{pair} of linearizations 
\begin{align*}
(\mathscr{L}_{R}\ , \ \mathscr{L}_{\Delta}) \in \mbox{Pic}^G(Y) \times \mbox{Pic}^G(Y)  \ .
\end{align*}

We will refer to this pair, by abuse of terminology, as the {CM-polarization} in recognition of Tian's early work.  The theory of embeddings of algebraic homogeneous spaces provides us with the correct definition of {semistability with respect to} $(\mathscr{L}_{R}\ , \ \mathscr{L}_{\Delta}) $. We say that our definition of  \emph{semistable pairs}\footnote{See Definitions \ref{dominate} and \ref{pair}. Our theory is not related to Hitchins' theory.} is the correct one because this definition {coincides} with the existence of a lower bound for the  Mabuchi energy restricted to the Bergman metrics. We prove that the pair $(\mathscr{L}_{R}\ , \ \mathscr{L}_{\Delta}) $ is {globally generated}. In this way the positivity sought after in the old theory is available, and perhaps better understood, in our new theory of pairs. Moreover, this positivity provides the indispensable connection with representation theory. Our new definition of CM semistability cannot be developed or even stated without this connection. The ideas required for the proof of global generation first arose in a breathtaking 1848 note of Arthur Cayley (see \cite{cay}) . More modern formulations can be found in ( \cite{detdiv}, \cite{fogarty}, \cite{gkz}, \cite{weyman}) .

 In our previous papers (see \cite{paul2011} , \cite{paul2009} ) we studied the Mabuchi energy, the $X$-\emph{resultant} $R(X)\in \mathbb{P}(\elam)$ , and the $X$-\emph{hyperdiscriminant} $\Delta(X)\in\mathbb{P}(\emu )$ \footnote{Definitions appear in section \ref{discussionofresults} below.}of a {fixed} projective manifold $X\ra\cpn$ . The vector spaces $\elam$ and $\emu$ are certain finite dimensional (irreducible) representations of $G$ with corresponding highest weights $\lambull$ and $\mubull$. These modules are described in detail below.
 
   In this paper we analyze, among other things, the relative situation. To begin consider a flat family $\mathbb{X} \xrightarrow{\pi} Y$ of polarized manifolds. Define maps $\Delta$ and $R$ as follows: 
 \begin{align}\label{rdmaps}
\begin{split}
& \Delta:Y\rightarrow{ }\mathbb{P}(\emu) \ , \ \Delta(y):= \Delta(X_y)  \ , \\
\ \\
& R:Y\rightarrow{ }\mathbb{P}(\elam) \ , \ R(y):= R(X_y)  \ .
\end{split}
\end{align}
  \begin{remark}
  \emph{When the family is $G$-equivariant the maps $R$ and $\Delta$ are also $G$-equivariant.}
  \end{remark} 
 The first main result of this paper is the following.
 \begin{theorem}\label{globgen} (Positivity of the CM polarization)  \emph{Let $\mathbb{X} \xrightarrow{\pi} Y$ be a $G$ equivariant family of polarized manifolds.  
There exist invertible sheaves $\mathscr{L}_{\Delta} , \mathscr{L}_{R} \in \mbox{Pic}^G(Y)$ such that:
\begin{align*}
& i)\ \mathscr{L}_{\Delta}\ \mbox{is globally generated} . \\
 \ \\
 & ii)\ \mbox{There exists a base point free $G$ invariant finite dimensional subspace}\\
 &\ \mathbb{E}\subset \Gamma(Y\ , \ \mathscr{L}_{\Delta}) \ \mbox{satisfying}\ \mathbb{E}\cong \emu^{\vee}\ . \\
 \ \\
& iii)\ \mbox{The associated morphism}\  \phi_{\mathscr{L}_{\Delta} ,\emu}:Y\ra \mathbb{P}(\emu)
\ \mbox{coincides with $\Delta$ . } \\
 \ \\
 & iv)\ \mathscr{L}_{R}\ \mbox{is globally generated} . \\
 \ \\
 & v)\ \mbox{There exists a base point free $G$ invariant finite dimensional subspace}\\
 &\ \mathbb{F}\subset \Gamma(Y\ , \ \mathscr{L}_{R}) \ \mbox{satisfying}\ \mathbb{F}\cong \elam^{\vee}\ . \\
 \ \\
& vi)\ \mbox{The associated morphism}\  \phi_{\mathscr{L}_{R} ,\elam}:Y\ra \mathbb{P}(\elam)
\ \mbox{coincides with $R$ . }\ \Box
 \end{align*}}
\end{theorem}
\begin{remark}
\emph{In most applications the base $Y$ is quasi projective. In particular, $Y$ is usually not complete.} 
\end{remark}

\begin{corollary}\emph{$R$ and $\Delta$ are $G$-equivariant \emph{\textbf{regular}} maps and there are isomorphisms of invertible sheaves on $Y$}
\begin{align*}
\mathscr{L}_{R}\cong R^*\mathcal{O}_{\mathbb{P}(\elam)}(1) \ , \ \mathscr{L}_{\Delta}\cong \Delta^*\mathcal{O}_{\mathbb{P}(\emu)}(1) \ .\ \Box
\end{align*}
\end{corollary}
The corollary implies that the lifts of $y\in Y$ to $\mathscr{L}^{-1}_{R}$ and $\mathscr{L}^{-1}_{\Delta}$ respectively admit {geometric} interpretations
\begin{align*}
{\mathscr{L}^{-1}_{R}}|_y=\mathbb{C}R(X_y) \ ,\ {\mathscr{L}^{-1}_{\Delta}}|_y=\mathbb{C}\Delta(X_y)\ .
\end{align*}

\begin{remark}
\emph{ The map $R:Y\rightarrow{ }\mathbb{P}(\elam)$ is essentially the Hilbert-Chow morphism.  A study of this map has been carried out in \cite{detdiv} and \cite{fogarty}. In particular the reader should be aware that parts $iv)-vi)$ of Theorem \ref{globgen} are known. The proof provided here is new and, we believe, much simpler.}
\end{remark}

 Thanks to Theorem \ref{globgen} given a $G$-equivariant family $\mathbb{X}\xrightarrow{\pi}Y$  we may define the following subset of $Y$
 \begin{align} 
 Y^{(n)ss}(\mathscr{L}_{R}\ , \ \mathscr{L}_{\Delta}):=  \mbox{(numerically) semistable locus in $Y$ .}
 \end{align} 
  See Definition \ref{defnsemistable} below for the definitions of the (numerical) semistable locus.
 
 The connection between these loci and K-energy bounds is brought out in the following theorem, in order to state it we first fix some notation.
 
 For any vector space $\mathbb{V}$ and any $v\in \mathbb{V}\setminus\{0\}$ we  let $[v]\in\mathbb{P}(\mathbb{V})$ denote the line through $v$. If $\mathbb{V}$ and $\mathbb{W}$ are $G$ modules we define the projective orbits :     
\begin{align}
\mathcal{O}_{vw}:=G\cdot [(v,w)]  \subset \mathbb{P}(\mathbb{V}\oplus\mathbb{W}) \ , \ \mathcal{O}_{v}:=G\cdot [v]  \subset \mathbb{P}(\mathbb{V}\oplus\{0\})\ .
\end{align}
We let $\overline{\mathcal{O}}_{vw}\ ,\  \overline{\mathcal{O}}_{v}$ denote the Zariski closures of these orbits.

 Our second new result is as follows.
 \begin{theorem}\label{uniform1} 
 \emph{  Let $\mathbb{X}\xrightarrow{\pi}Y$ be a flat $G$-equivariant family of smooth subvarieties of $\cpn$ with Hilbert polynomial $P$ } 
 \begin{align*}
 P(T)=c_n\binom{T}{n}+c_{n-1}\binom{T}{n-1}+O(T^{n-2})\ .
 \end{align*}
\begin{enumerate}
\item \emph{ The Mabuchi energy of $(X_y,\ {\om_{FS}}|_{X_y})$ is bounded below on $\mathcal{B}$ if and only if \newline $y\in Y^{ss}(\mathscr{L}_{R}\ , \ \mathscr{L}_{\Delta})$ .} \\
 \ \\
 \item \emph{There is a constant $M$ which depends only on $c_n , c_{n-1}$ and $h$  such that for every \newline $y\in Y^{ss}(\mathscr{L}_{R}\ , \ \mathscr{L}_{\Delta})$ the following inequality holds}
  \begin{align}\label{M}
&|(n+1) \inf_{\sigma\in \mathcal{B}}\nu_{{\om_{FS}}|_{X_y}}(\sigma)- \frac{1}{d^2}
\log\tan^2 d _{g} (\overline{\mathcal{O}}_{R(y)  \Delta(y)},\overline{\mathcal{O}}_{R(y)} )| \leq M \ .  
 \end{align}
 \ \\
\item \emph{ The Mabuchi energy of $(X_y,\ {\om_{FS}}|_{X_y})$ is bounded below on all degenerations in $G$  if and only if } $y\in Y^{nss}(\mathscr{L}_{R}\ , \ \mathscr{L}_{\Delta})$ . \\  
\ \\
\item \emph{ The locus $Y^{nss}(\mathscr{L}_{R}\ , \ \mathscr{L}_{\Delta})$ is {constructible}}. 
\end{enumerate}
\end{theorem}
 In the statement of Theorem \ref{uniform1} we let $h$ denote a background positive definite Hermitian form on $\cn$, $\mathcal{B}$ is the set of Bergman metrics associated to $G$ , and  $d_g $ denotes the distance in the Riemannian metric induced by $h$  on the projective space $\mathbb{P}(\elam\oplus\emu)$ . The polynomials $R(X_y)$ and $\Delta(X_y)$ have been normalized to unit length. The reader is referred to section 2 for an explanation of undefined terms.
 \begin{remark}\emph{
 Part $(4)$ of  Theorem \ref{uniform1} says that the locus of points $y\in Y$ whose corresponding fibers have Mabuchi energy lower bounds along all degenerations in $G$ is cut out by a finite number of  polynomial equalities and non-equalities. Moreover the author expects that the semistable and numerically semistable loci coincide (see Conjecture \ref{numericalcriterion}) .}
 \end{remark}

  In his discussion of the definition of (semi)stability Mumford ( see \cite{git} pg.194 ) remarks   that  `` (these notions) \emph{are interesting topological properties to study in any linear representation} (of $G$) ''.  In this article semistability is developed entirely in the context of complex reductive linear algebraic groups and pairs
  \begin{align*}
  (v\in\mathbb{V}\setminus\{0\} \ ; \ w\in \mathbb{W}\setminus\{0\} )
  \end{align*}
  of finite dimensional rational representations of such groups (see Definition \ref{pair}) . A semi-stable pair is a special case of the relation of {dominance} in the theory of equivariant embeddings of an algebraic homogeneous space $\mathcal{O}$ .  We provide a ``polyhedral combinatorial'' necessary condition for semistability which we call \emph{numerical} semistability (see Definitions \ref{numss} and \ref{numerical}) , the author believes this to be a sufficient condition as well (see Property (P) and Conjecture \ref{numericalcriterion} in section 3 ).  
  
  Numerical semistability is closely related to the next result of this paper. This theorem grew out of the author's many conversations with Joel Robbin.

Let $D_{\delta} $ denote the $\delta$ disk in $\mathbb{C}$.
 
    \begin{theorem}\label{weakII} (Analytic Numerical Criterion ) \emph{Let $X^n \hookrightarrow \cpn$ be a smooth, linearly normal, complex projective variety of degree $d \geq 2$. Assume that
\begin{align}
\inf_{\sigma\in \mathcal{B}}\nu_{\om}(\vps)=-\infty \ .
\end{align}
Then there exists a positive number $\delta$, a holomorphic mapping $\sigma :D_{\delta}\ra G $, and an algebraic one parameter subgroup $\lambda$ of $G$ such that
\begin{align}
\lim_{\alpha\ra 0}\nu_{\om}(\varphi_{\gamma(\alpha)})= -\infty \ \qquad \gamma(\alpha) := \lambda(\alpha)\cdot \sigma(\alpha) \qquad \Box \ .
\end{align}}
\end{theorem}
 
\begin{remark}\emph{The limitation of Theorem \ref{weakII} is that we cannot, at the moment, deduce the existence of a \emph{one parameter subgroup} along which the Mabuchi energy degenerates to $-\infty$. A stronger version of this result can be obtained provided we restrict to algebraic tori in $G$. In this situation one can always approach the boundary with a one parameter subgroup. See Lemma \ref{p4tori} below ( see also Theorem E from \cite{paul2011} ) .}
\end{remark}
  
  Our new versions of  CM/K-semistability bring our work into contact with many subjects in mainstream algebraic geometry: compactifications of (locally) symmetric spaces, toric varieties, complete symmetric varieties, wonderful compactifications, spherical varieties, and classical Geometric Invariant Theory. In fact the definition of semistability developed in this paper already appears implicitly in these subjects.

 Despite the terminology ``semistable pair'', Mumford's Geometric Invariant Theory is not, for our purpose, the only model to emulate. The author has also been inspired by the works of DeConcini-Procesi (see \cite{deconcini1983}), Brion-Luna-Vust (see \cite{blv1986}),  and Popov-Vinberg (see \cite{popov-vinberg}). The abstract representation / equivariant embedding theory is then applied to the non-trivial {example} of projective varieties by considering the map (\ref{rdmaps}).

One of the many advantages of our new point of view is the presence of a large number of interesting {examples} of semistable pairs, equivalently, examples of morphisms between equivariant completions of algebraic homogeneous spaces. A particularly accessible illustration of the definition is provided by the class of {two orbit} varieties. We consider several examples in Section \ref{generaltheory}.

We exhibit the present scope of our theory with the following diagram. In particular, the stability required for the K\"ahler Einstein problem is not a special case of Geometric Invariant Theory, rather, both theories are special cases of the theory of pairs. \\
\ \\
 \tikzstyle{block} = [rectangle, draw, fill=blue!20, 
    text width=5em, text centered, rounded corners, minimum height=4em]
 \tikzstyle{line} = [draw, -latex']
\begin{center} 
 \begin{tikzpicture}[node distance = 3.5cm, auto]
\node[block](morphisms){Morphisms between embeddings of $\mathcal{O}$};
\node[block,below of=morphisms](pairs){Semistable Pairs};
\node[block, below of=pairs](two){Quasi-closed orbits in projective representations};
\node[block, right of=two](mumford){Hilbert-Mumford Geometric Invariant Theory};
\node[block, left of=two](bounds){ K-energy lower bounds on Bergman metrics};
 \path[line](morphisms)--(pairs);
 \path [line](pairs)--(bounds);
\path [line](pairs)--(mumford);
\path [line](pairs)--(two);
 \end{tikzpicture}
\end{center}
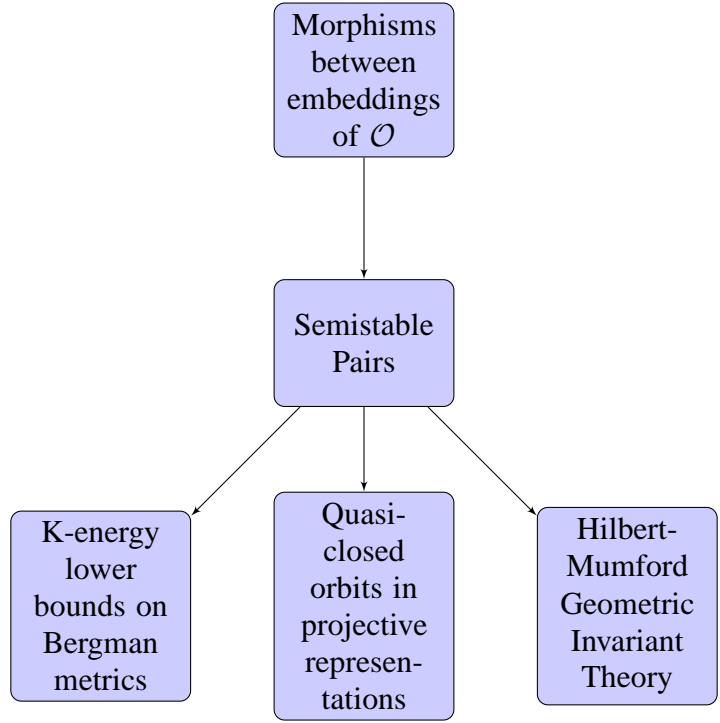
\begin{center}{\small{Figure 1. Special cases of the theory of semistable pairs.}}\end{center}

\subsection{ Organization }
 In order to maximize the readers' understanding we give a synopsis of its contents. In section \ref{discussionofresults}   
we define resultants, discriminants, and hyperdiscriminants of complex projective varieties and define the semistability of projective varieties in terms of these objects (see Definitions \ref{cmss} and \ref{numss}) . We also define the Mabuchi energy,  the space of Bergman metrics, and state several auxiliary results which follow easily from the discussion in Section \ref{generaltheory}. In Section \ref{equivexts}  we define the concept of a semistable pair of points in the context of finite dimensional rational representations of algebraic groups (Definitions \ref{dominate} and \ref{pair} ) . In Section \ref{toricmorphs} we discuss numerical semistability (definition \ref{numerical}). Section \ref{kempfness} is concerned with a Kempf-Ness type ``energy functional'' attached to a pair of representations . At this point the reader will have noticed a strong family resemblance between our theory of semistable pairs and Mumfords' Geometric Invariant Theory. A head to head comparison of the two theories is provided in Table 1 at the end of Section \ref{generaltheory}.   
Section 4 is concerned with determinants of direct images of Cayley-Koszul complexes. This section contains the proof of Theorem \ref{globgen}.

 \section{The Basic Construction and Further Results}\label{discussionofresults}
  
  
  For the convenience of the reader we recall the definitions of  $\Delta(X)$, the hyperdiscriminant, and $R(X)$, the resultant, of a projective variety. We always assume that $X$ is embedded into $\cpn$ as a linearly normal variety. This insures that the resultant and discriminant of $X$ behave as well as possible ( see \cite{tevelev} section 1.4.3).  For further details and background we refer the reader to \cite{gkz} and \cite{paul2011} .
 \subsection{(Hyper)discriminants and Resultants of projective varieties} 
Let $X^n\ra \cpn$ be an irreducible, $n$-dimensional, linearly normal, complex projective variety of degree $d\geq 2$.
Let $\mathbb{G}(k,N)$ denote the Grassmannian of $k$-dimensional projective linear subspaces of $\cpn$. This is the same as $G(k+1, \mathbb{C}^{N+1})$ , the Grassmannian of $k+1$ dimensional subspaces of $\mathbb{C}^{N+1}$.   

\begin{definition} (Cayley  \cite{cayley1860}) \emph{The \textbf{\emph{associated form}} of $X^n\ra \cpn$ is given by}
\begin{align*}
Z_X:=\{L\in \mathbb{G}(N-n-1,N)\ | L\cap X\neq \emptyset \} \ .
\end{align*}
\end{definition}
It is well known that $Z_X$ enjoys the following properties: \\
\ \\
$i)$   $Z_X$ is a \textbf{\emph{divisor}} in $\mathbb{G}(N-n-1,N)$ . \\
 \ \\
$ii)$   $Z_X$ is irreducible . \\
  \ \\
$iii)$  $\deg(Z_X)=d$ (in the Pl\"ucker coordinates) . \\
  \ \\
  
  Therefore there exists $R_X\in H^0(\mathbb{G}(N-n-1,N), \mathcal{O}(d))$ such that
 \begin{align*}
 \{R_X=0\}=Z_X
 \end{align*}
 $R_X$ is the Cayley-{Chow} form of $X$. Following  Gelfand, Kapranov, and Zelevinsky  we call $R_X$ the \textbf{\emph{X-resultant}} .  Observe that   
\begin{align*}
R_X\in \mathbb{C}_{d(n+1)}[M_{(n+1)\times (N+1)}]^{SL(n+1,\mathbb{C})} \ . 
\end{align*}

\subsection{Discriminants} We assume that $X\ra\cpn$ has degree $d\geq 2$. Let $X^{sm}$ denote the smooth points of $X$. For $p\in X^{sm}$ let 
$\mathbb{T}_p(X)$ denote the {embedded} tangent space to $X$ at $p$. Recall that $\mathbb{T}_p(X)$ is an $n$-dimensional {projective} linear subspace of $\cpn$.
\begin{definition}
 \emph{ The \emph{\textbf{dual variety}} of $X$, denoted by $X^{\vee}$, is the Zariski closure of the set of  {tangent hyperplanes} to $X$ at its smooth points }
\begin{align*}
X^{\vee}:=\overline{ \{ f\in {\cpn}^{\vee} \ |  \  \mathbb{T}_p(X)\subset \ker(f) \ , \ p\in X^{sm}\} } \ .
\end{align*}
\end{definition}
 Usually $X^{\vee}$ has codimension one in $ {\cpn}^{\vee}$.  
\begin{definition} \emph{The \emph{\textbf{dual defect}} of $X\ra \cpn$ is the integer}
\begin{align*}
\delta(X):=\mbox{\emph{codim}}(X^{\vee})-1 \ .
\end{align*}
\end{definition}
 When $\delta=0$ there exists an irreducible homogeneous polynomial $\Delta_X\in \mathbb{C}[{\cpn}^{\vee}] $ (  the  \textbf{\emph{X-discriminant}})  such that
\begin{align*}
X^{\vee}=\{\Delta_X=0\} \ .
\end{align*}
 \begin{example}
\emph{There are many varieties with positive dual defect. For an example one may take $X:=G(2,\mathbb{C}^{2n+1})\hookrightarrow \mathbb{P}(\wedge^2\mathbb{C}^{2n+1})$ . The defect for this variety is positive because the space $\wedge^2\mathbb{C}^{2n+1}$ is \emph{prehomogeneous} for the action of $G:=SL(2n+1,\mathbb{C})$, that is, it contains an open dense orbit. Therefore all $G$ invariants are constant. If $X$ had codimension one dual then $\Delta_X$ would be a non-trivial $G$ invariant. In the even case, one may use the determinant as a nontrivial invariant. Recall that the Grassmannian of lines in $\mathbb{P}^3$ is a smooth hypersurface in $\mathbb{P}^5$ so its dual is automatically codimension one by general theory (see \cite{tevelev}).  }
\end{example}   

\subsection{Hyperdiscriminants }  Given $X\ra \cpn$ we consider the \emph{Segre embedding} 
\begin{align*}
\xhyp\ra \mathbb{P}({M_{n\times (N+1)}}^{\vee}) \ .
\end{align*}
Of basic importance for this paper is the next proposition which follows from work of Weyman and Zelevinsky (see \cite{weymanzelevinsky}) and Zak ( see \cite{zak} ) .
\begin{proposition}\label{cayleytrick} \emph{ Let $X\ra \cpn$ be a nonlinear subvariety embedded by a very ample complete linear system. Then $\delta(\xhyp)=0 $ .  }
 \end{proposition}
\begin{remark} \emph{The reader should observe that $X$ is not required to be smooth .}
\end{remark}

It follows from Proposition \ref{cayleytrick} that there exists a nonconstant homogeneous polynomial 
 \begin{align*}
 { \Delta_{\xhyp}\in \mathbb{C}[M_{n\times (N+1)}]}^{SL(n,\mathbb{C})}\ ,
 \end{align*}

which we shall call the \textbf{\emph{X-hyperdiscriminant}}, such that
\begin{align*}
\{\Delta_{\xhyp}=0\}=(\xhyp)^{\vee} \ .
\end{align*}
\begin{remark}
\emph{For further information on the hyperdiscriminant the reader is referred to \cite{paul2011} section 2.2 pg. 270. The two crucial properties are that $\xhyp$ is always dually non-degenerate and that $\Delta_{\xhyp}$ encodes only the Ricci curvature of $X\ra \cpn$.}
\end{remark}

The obvious task is to compute the degree of this polynomial .

\begin{proposition}\label{degreeofhyp} (see \cite{paul2011} Proposition 5.7) \emph{Assume $X$ is {smooth}. Then the degree of the hyperdiscriminant is given as follows}
\begin{align*}
\deg(\Delta_{\xhyp}) =n(n+1)d-d\mu \in \mathbb{Z}_+\ .\ \Box
\end{align*}
\end{proposition} 
In the preceding proposition $\mu$ denotes, as usual, the average of the scalar curvature of ${\om_{FS}}|_X$.  For the algebraic geometer this number is essentially the subdominant coefficient of the Hilbert polynomial of $X$. 

So far, to a nonlinear projective variety $X\ra\cpn$ we have associated two polynomials: $R_X$ and $\Delta_X$.
 Translation invariance of the Mabuchi energy forces us to {normalize the degrees} of  these polynomials. From this point on we are interested in the pair 
\begin{align} 
 R=R(X):= R^{\deg(\Delta_{\xhyp})}_{X} \ ,\ \Delta=\Delta(X):=\Delta ^{\deg(R_X)}_{\xhyp} \ .
  \end{align}
 Below we shall let $r$ denote their \emph{common} degree 
 \begin{align}
 r=\deg(\Delta_{\xhyp})\deg(R_X)=d(n+1)(n(n+1)d-d\mu) \ .
 \end{align}
 We summarize this discussion as follows .\newline
\ \\
 \noindent\textbf{The Basic Construction .}
   Given a partition $\beta_{\bull}$ with $N$ parts let $\mathbb{E}_{\beta_{\bull}}$ denote the corresponding irreducible  $G:=\slnc$ module. 
   Let $X\ra\cpn$ be a linearly normal complex projective variety. We make the associations
\begin{align*}
& X  \rightarrow R=R(X) \in \elam\setminus\{0\} \ , \ (n+1)\lambda_{\bull}= \big(\overbrace{ {r} , {r} ,\dots, {r}}^{n+1},\overbrace{0,\dots,0}^{N-n}\big) \ .\\
\ \\
& X \rightarrow \Delta=\Delta(X) \in \emu\setminus\{0\} \ , \  n\mu_{\bull}= \big(\overbrace{ {r} , {r} ,\dots, {r}}^{n },\overbrace{0,\dots,0}^{N+1-n}\big) \ . 
\end{align*}
 Moreover, the associations are $G$ equivariant:
\begin{align*}
 R(\sigma\cdot X)=\sigma\cdot R(X) \ , \ \Delta(\sigma\cdot X)=\sigma\cdot \Delta(X) \ .
 \end{align*}
The irreducible modules $\elam$ and $\emu$ admit the following descriptions
\begin{align*}
&\elam\cong  H^0(\mathbb{G}(N-n-1,N), \mathcal{O}\Big(\frac{r}{n+1}\Big))\cong \mathbb{C}_{r}[M_{(n+1)\times (N+1)}]^{SL(n+1,\mathbb{C})} \\
\ \\
& \emu \cong H^0(\mathbb{G}(N-n ,N), \mathcal{O}\Big(\frac{r}{n}\Big))\cong \mathbb{C}_{r}[M_{n\times (N+1)}]^{SL(n,\mathbb{C})}\ . \ \Box
\end{align*}
\begin{remark}
\emph{Observe that $r$ is divisible by both $n$ and $n+1$. Therefore $\lambull$ and $\mubull$ are actual partitions.}
\end{remark}

\begin{remark}\emph{ $X\ra \cpn$ need not be smooth in order to carry out the basic construction. In this event $r$ is still the product of the degrees, but there seems to be no straightforward expression for this degree for arbitrarily bad singularities.  }
\end{remark}

For any vector space $\mathbb{V}$ and any $v\in \mathbb{V}\setminus\{0\}$ we  let $[v]\in\mathbb{P}(\mathbb{V})$ denote the line through $v$. If $\mathbb{V}$ is a $G$ module define $\mathcal{O}_{v}:=G\cdot [v]\subset \mathbb{P}(\mathbb{V})$ the projective orbit of $[v]$ .  We let $\overline{\mathcal{O}}_{v}$ denote the Zariski closure of this orbit. Given  $(v\in\mathbb{V}\setminus\{0\} \ ; \ w\in \mathbb{W}\setminus\{0\} )$ we consider the orbits inside the projective space of the direct sum  
\begin{align}
\mathcal{O}_{vw}:=G\cdot [(v,w)]  \subset \mathbb{P}(\mathbb{V}\oplus\mathbb{W}) \ , \ \mathcal{O}_{v} \subset \mathbb{P}(\mathbb{V}\oplus\{0\})\ .
\end{align} 
Now we are prepared to introduce the following replacement for Tian's definition of CM semistability (see \cite{tian97} section 8) .  
 \begin{definition}\label{cmss} 
\emph{Let $X\ra\cpn$ be an irreducible, $n$-dimensional, linearly normal complex projective variety. Then 
$X$ is \emph{\textbf {semistable}}  if and only if }
\begin{align}
\overline{ \mathcal{O}}_{R\Delta}\cap \overline{ \mathcal{O}}_{R}=\emptyset \ .
\end{align}
\end{definition}
 
Recall that for any $G$ module $\mathbb{V}$ and any $v\in \mathbb{V}\setminus\{0\}$ the {weight polytope} $\mathcal{N}(v)$ is the convex hull of the $H$ weights of $v$ where $H$ is a maximal algebraic torus in $G$.  Observe that $\mathcal{N}(v)$ is a compact convex lattice polytope in $M_{\mathbb{R}}(H)$. 

Now we are prepared  to introduce the following replacement for Tian's (respectively  Donaldson's) definition of K semistability (see  \cite{tian97} definition 1.1 and \cite{skdtoric} section 2.1). 

 \begin{definition} \label{numss}  
\emph{Let $X\ra\cpn$ be an irreducible, $n$-dimensional, linearly normal complex projective variety. Then $X$ is \emph{\textbf {numerically semistable}}  if and only if }
\begin{align}
 \mathcal{N}(R) \subset  \mathcal{N}(\Delta) \ \ \mbox{\emph{for every maximal algebraic torus} $H\leq G$} \ .
\end{align}
\end{definition}

\begin{remark}\emph{Clearly Definitions \ref{cmss} and \ref{numss} can be formulated for any nonzero pair of vectors $(v,w)$ in any pair $\mathbb{V},\mathbb{W}$ of finite dimensional $G$-modules.This is the correct level of generality in which to operate and this point of view will be fully developed below (see Definitions \ref{pair} and \ref{numerical} ). The author expects that these definitions are equivalent (see Conjecture \ref{numericalcriterion} below) . }
\end{remark}

 Given any compact be a K\"ahler manifold $(X,\ \om)$ we always let $\mu$ denote the average of the scalar curvature of $\om$ and $V$ denote the volume
\begin{align*}
\mu:=\frac{1}{V}\int_X\mbox{Scal}(\om)\om^n \ , \
V:= \int_X\om^n \ .
\end{align*}

The space of K\"ahler potentials will be denoted by $\mathcal{H}_\om$
\begin{align*}
\mathcal{H}_\om:=\{\vp\in C^{\infty}(X)\ |\ \om_\vp:=\om+\frac{\sqrt{-1}}{2\pi}\dl\dlb\vp>0 \} \ .
\end{align*}

  The \textbf{ \emph{Mabuchi energy}}  is  the functional $\nu_\om:\mathcal{H}_\om\ra \mathbb{R}$  defined by
  
\begin{align}\label{mabenergy}
 \qquad \nu_{\omega}(\varphi):= - \frac{1}{V}\int_{0}^{1}\int_{X}\dot{\varphi_{t}}(\mbox{Scal}(\varphi_{t})-\mu)\omega_{t}^{n}dt.
\end{align}
 $\varphi_{t}$ is a $C^1$ path in $\mathcal{H}_\om$ joining $0$ and $\varphi$. It is well known that  $\nu_{\omega}$ does not depend on the path chosen and that $\vp$ is a critical point of $ \nu_{\omega}$  if and only if $\mbox{Scal}_\om(\vp)\equiv \mu$ (see \cite{mabuchi}) . One of the most important analytic results concerning the Mabuchi energy is the following.
 
 \begin{theorem} (Bando-Mabuchi \cite{bando-mabuchi87})\emph{ Let $(X,\om)$ be a compact K\"ahler manifold satisfying $[\om]=C_1(X)$. If there exists a K\"ahler Einstein metric in the class $[\om]$ then $\nu_{\om}$ is bounded from below on $\mathcal{H}_\om$.   }
 \end{theorem}
 \begin{remark}\emph{This result has been generalized been generalized to an \emph{arbitrary} K\"ahler class by Chen and Tian.}
 \end{remark}
 Now we suppose that $X^n\ra \cpn$ is a smooth complex projective variety . Let $\om_{FS}$ denote the Fubini-Study K\"ahler form on $\cpn$ . We set ${\om:=\om_{FS}}|_X$ .  The space $\mathcal{B}$ of {Bergman  metrics} is defined by
 \begin{align*}
\mathcal{B}:=\{\ \sigma^*\om_{FS}  \ |\ \sigma\in G\ \}\ .
\end{align*}    
  Since $\sigma\in G$ acts by automorphisms of $\cpn$ there is a potential $\vps \in C^{\infty}(X)$ such that
 \begin{align*}
  \sigma^*\om_{FS}= \om_{FS}+\frac{\sqrt{-1}}{2\pi}\dl\dlb\vps >0\ .
  \end{align*}
  Therefore ${\sigma^*\om_{FS}}|_X>0$ and we have an ``inclusion'' map $\mathcal{B}\ra \mathcal{H}_{\om}$. We will often confound $G$ with $\mathcal{B}$. Now we may think of the Mabuchi energy as being a function on $\mathcal{B}$.  For any $\sigma\in \mathcal{B}$ we define 
 \begin{align*}
 \nu_{\om}(\sigma):=\nu_{\om}(\vps) \ .
 \end{align*}
 Given $\lambda:\mathbb{C}^*\ra G$,  an algebraic one parameter subgroup of $G$, the associated potentials $\varphi_{\lambda(t)} \in \mathcal{B}$ , are called \emph{degenerations}.  
 
 A cornerstone of the present article is the following Theorem, recently obtained by the author. The reader should compare this with Tian's result (\ref{singnrm}) . \\
 \ \\
\textbf{Theorem A .} (Paul \cite{paul2011})
 { Let $X^n \hookrightarrow \cpn$ be a smooth, linearly normal, complex projective variety of degree $d \geq 2$ . 
  Then there are continuous norms on $\emu$ and $\elam$ such that the Mabuchi energy restricted to $\mathcal{B}$ is given as follows}
\begin{align*} 
\begin{split}
&d^2(n+1)\nu_{\om}(\varphi_{\sigma})=  \log  \frac{{||\sigma\cdot\Delta ||}^{2}}{{||\Delta ||}^{2}} -   \log\frac{{||\sigma\cdot R||}^{2}}{||R||^2}      \ . \qquad \Box
 \end{split}
\end{align*}
 
 The definitions of stability proposed by the author are completely justified by the next result.
\begin{theorem}(Paul \cite{paul2011})
\begin{align*}
\begin{split}
& i) \quad {\nu_{\om}}|_G>-\infty \iff \  \overline{ \mathcal{O}}_{R\Delta}\cap \overline{ \mathcal{O}}_{R}=\emptyset \ . \\  
\ \\
& ii) \quad {\nu_{\om}}|_{\Delta(G)}>-\infty \iff \mbox{ $\mathcal{N}(R) \subset  \mathcal{N}(\Delta)$  for all maximal algebraic tori $H\leq G$ .}
\end{split}
\end{align*}
\end{theorem}

Next we explain how our new theory of semistability makes contact with Tian's earlier notion of CM-semistability as formulated in \cite{tian97}. Given  $\mathbb{X}\xrightarrow{\pi}Y$  a flat family of polarized manifolds we define a line bundle on $Y$ as follows
\begin{align*}
\mathscr{L}_{\pi}:=\mathscr{L}_{R}^{\vee}\otimes\mathscr{L}_{\Delta}=  R^*\mathcal{O}_{\mathbb{P}(\elam)}(-1)\otimes  \Delta^*\mathcal{O}_{\mathbb{P}(\emu)}(+1) \ .
  \end{align*}
 In the next proposition we compare this sheaf to Tian's CM-polarization. Precisely, we compute the first Chern class of $\mathscr{L}_{\pi}$ in terms of the family $\mathbb{X}\xrightarrow{\pi}Y$. The reader should compare this to 8.3 on pg. 28 of \cite{tian97} . In what follows $\deg(\mathbb{X}/Y)$ denotes the degree of any fiber of the map $\mathbb{X}\xrightarrow{\pi}Y$.
 \begin{proposition}
 \emph{Assume that the Grothendieck Riemann Roch theorem is valid for the map $\mathbb{X}\xrightarrow{\pi}Y$. Then the first Chern class of $\mathscr{L}_{\pi}$ coincides (up to a positive power) with $c_1(\mathscr{L}_{cm})$. Precisely}
\begin{align*}
c_1(\mathscr{L}_{\pi})=\deg(\mathbb{X}/Y)\pi_*\left( (n+1)c_1(K_{\pi})\pi_2^* c_1(L)^n+\mu \pi_2^* c_1(L)^{n+1}\right) \ . \quad \Box
\end{align*}
 \end{proposition}
 \begin{remark}\emph{ All earlier approaches to the positivity of $\mathscr{L}_{cm}$ are based on the study of its first Chern class.  In the present work we produce a large number of algebraic sections. This is a much more substantial measure of positivity.}
 \end{remark} 
 Let $(X,\om)$ be any compact K\"ahler manifold. Let $\eta(X)$ denote the Lie algebra of holomorphic vector fields. Recall that the Calabi-Futaki invariant (see \cite{Futaki}) is the Lie algebra character
\begin{align}\label{futakiinv}
F_{[\om]}:\eta(X)\ra\mathbb{C} \ , \ F_{[\om]}(v):=\int_Xv(h_{\om})\om^n \ .
\end{align}
The potential $h_{\om}$ is defined by
\begin{align*}
Scal(\om)-\mu=\triangle h_{\om} \ .
\end{align*}
Let $\mathbb{X}\xrightarrow{\pi} Y$ be a flat family of polarized manifolds.  For any $y\in Y$ we define  the space of \textbf{\emph{special degenerations}} of $y$, denoted by $\Lambda_y$ ,  as follows 
 \begin{align*}
& \Lambda_y:= \{\lambda:\mathbb{C}^*\ra G\ |\ \lambda(0)\cdot y\in Y \} \ .
\end{align*}
In other words, $\lambda$ is a special degeneration of $y\in Y$  if and only if $\lambda(\alpha)\cdot y$ specializes to a point $y_0:=\lambda(0)\cdot y$ in $Y$ as $\alpha\ra 0$. Next we observe that  $\lambda \in \Lambda_y$  induces a holomorphic vector field $v_{\lambda}$  along the central fiber $X_0:=X_{y_0}$ . Therefore we may consider the Calabi-Futaki invariant of  $(X_0, {\om_{FS}}|_{X_0}, v_{\lambda})$ which we denote by $F_0(v_{\lambda})$. 
 The following definition is a special case of a slightly broader notion introduced by Ding and Tian (see \cite{dingtian})  .
 \begin{definition}\emph{
 Let $\mathbb{X}\xrightarrow{\pi} Y$ be a flat family of polarized manifolds. Let $y\in Y$. The \textbf{\emph{generalized Futaki invariant}} of $X_y$ is the map}
 \begin{align*}
 F:\Lambda_y\ra\mathbb{C} \ , \ F(v_{\lambda}):= F_0(v_{\lambda}) \ .
 \end{align*}
 \end{definition}
  The next corollary (of Theorem \ref{globgen}) gives conditions which insure that the hyperdiscriminant is compatible with limit cycle formation.

  \begin{corollary}\label{limitcycle}
  \emph{Let $y\in Y$ and let $\lambda$ be a special degeneration of $X_y$ . Then}
  \begin{align*}
  \lambda(0)\cdot \Delta(y)=\Delta(\lambda(0)\cdot y) \ .\ \Box
  \end{align*}
  \end{corollary}
 
Theorem A from \cite{paul2011} and corollary \ref{limitcycle} imply the following result, originally due to Ding and Tian. The reader should compare this result to the key claim 3.15 on pg. 328  of \cite{dingtian} (see also Proposition 6.2 on pg. 24 of \cite{tian97} as well as Theorem B pg. 3 from \cite{paul2011} ) .
  \begin{corollary}\label{genfutaki} 
\emph{ Let $\mathbb{X}\xrightarrow{\pi}Y$ be a flat family of polarized manifolds. Let $\nu_{{\om}_y}$ denote the Mabuchi energy of the fiber $(X_y,{\om_{FS}}|_{X_y})$. Let $\lambda$ be a special degeneration of $X_y$. Then there is an asymptotic expansion}
 \begin{align*}
\nu_{{\om}_y}(\vplt)=F_0(v_{\lambda})\log|t|^2+O(1) \ , \ |t|\ra 0 \ 
\end{align*}
\emph{where $F_0(v_{\lambda})$ is Ding and Tian's generalized Futaki invariant.} \emph{In particular}
\begin{align*}
\nu_{{\om}_y}(\vplt)= O(1) \ \mbox{\emph{as}} \ |t|\ra 0 \ \mbox{\emph{ whenever} $F_0(v_{\lambda})=0 $}\quad \Box \ .
\end{align*}
\end{corollary}

  \section{The general theory of semistable pairs }\label{generaltheory}  
 
  \begin{center}
 {\tiny{  
\begin{tabular}{ll}
   $G - $ complex reductive linear algebraic group &   $\lambda_{\bull}, \mu_{\bull} -$  \emph{dominant} integral weights for $G=\slnc$ \\
  $P$, $B$ parabolic (Borel) subgroups of $G$ &   $\lambda_{\bull}=(\lambda_1\geq \lambda_2\geq \dots\geq \lambda_N\geq 0)$ $\lambda_i\in \mathbb{Z}$ \\
$\Lambda_{\mathbb{Z}}(B)$ dominant weights relative to $B$& \\
& \\
   $\mathbb{V} ,\mathbb{W}-$ finite dimensional rational $G$-modules &  $\mathbb{E}_{\lambda_{\bull}}-${irreducible module} corresponding to $\lambda_{\bull}$ \\
 & \\ 
   $v,w-$ \emph{nonzero} vectors in  $\mathbb{V} ,\mathbb{W}$ respectively   &    $m_{\lambull}(\beta):=\dim\elam({\beta})-$ multiplicity of $\beta$ in $\elam$ \\
  $[v] \in  \mathbb{P}(\mathbb{V})$ corresponding line in $\mathbb{V}$ & \\
 &\\
   $H ,T -$ maximal algebraic tori in $G$ &   $\mbox{supp}(\lambull):=\mbox{supp}(\elam)$      \\
  & \\ 
   $W-$\emph{Weyl group} (relative to $B$) &   $\mathcal{N}(\lambull)-$ convexhull$\{s\cdot\lambull \ | \ s\in W\} \subset M_{\mathbb{R}}$   \\
  & \\
   $M_{\mathbb{Z}}=M_{\mathbb{Z}}(H):=$ character lattice of $H$ &    $\mathcal{N}(v)-\mbox{convexhull}\{\beta \in  \mbox{supp}(\mathbb{V})\ |\  v_{\beta}\neq 0\} \subset M_{\mathbb{R}} $ \\
$N_{\mathbb{Z}}:=\mbox{Hom}(M_{\mathbb{Z}},\mathbb{Z})$ (dual lattice) \\
 $M_{\mathbb{R}}:= M_{\mathbb{Z}}\otimes_{\mathbb{Z}}\mathbb{R}$ ,  $N_{\mathbb{R}}:= N_{\mathbb{Z}}\otimes_{\mathbb{Z}}\mathbb{R}$& \\
 &\\
   $\mathbb{V}({\beta})=\{v\in \mathbb{V}\ |\ \tau\cdot v=\beta(\tau)v \ \mbox{for all} \ \tau \in H\} $    &    $\Delta(G)-$algebraic one parameter subgroups of $G$\\
\quad (weight space) $v_{\beta}-$ projection of $v$ into $\mathbb{V}({\beta})$ &\\
 \quad$\beta\in M_{\mathbb{Z}}$&    \\ 
  &\\
  $\mbox{supp}(\mathbb{V}):=\{\beta \in M_{\mathbb{Z}}\ |\  \mathbb{V}({\beta})\neq 0\} $     & $\mathbb{V}^{\vee}=$ dual space to $\mathbb{V}$ \\
   
  \end{tabular}}}
 \end{center}
 \ \\
 \ \\
 \subsection{Equivariant extensions of rational maps}\label{equivexts}
 In this section we loosely follow the first few paragraphs of \cite{deconcini87}, our primary goal is to formulate, and put into context, our  notions of semistability  (Definitions \ref{dominate}, \ref{pair}, and \ref{numerical}) .
 
 To begin, let $G$ be an algebraic group. $H\leq G$ a Zariski closed (possibly finite) subgroup. Let $\mathcal{O}$ denote the algebraic homogeneous space $G/H$. The definition of semistable pair arises immediately upon studying equivariant {completions} of the space $\mathcal{O}$.  
\begin{definition}
\emph{An \textbf{\emph{embedding}} of $\mathcal{O}$ is a $G$ variety $X$ together with a $G$-equivariant embedding $i:\mathcal{O}\ra X$ such that $i(\mathcal{O})$ is an open dense subset of $X$.}
\end{definition}
  $[\sigma]=\sigma H$ denotes the associated $H$ coset for any $\sigma\in G$. Then an embedding of $\mathcal{O}$ has a natural basepoint given by $o:=i([e])$ and we have that
\begin{align*}
i(\mathcal{O})=G\cdot o \ ,\ \overline{G\cdot o}=X\ .
\end{align*}
Let $(X_1,i_1)$ and $(X_2,i_2)$ be two embeddings of $\mathcal{O}$.  We recall the following well established notion.
\begin{definition}
\emph {A \textbf{\emph{morphism}} $\varphi$ from $(X_1,i_1)$ to $(X_2,i_2)$ is a $G$ equivariant regular map \newline $\varphi:X_1\ra X_2$ such that the diagram
\[
 \xymatrix{
  &X_1 \ar[dd]^{\varphi}  \\
 \mathcal{O}\ar[ru]^{i_1}\ar[rd]_{i_2}&\\
 & X_2  }
 \]
commutes. If a morphism  $\varphi$ exists we write $(X_1,i_1)\succsim (X_2,i_2)$ and we say that $(X_1,i_1)$ \textbf{\emph{dominates}} $(X_2,i_2)$. }
\end{definition}
\begin{remark}\emph{Observe that if a morphism  exists it is unique.}\end{remark}

Let $(X_1,i_1)$ and $(X_2,i_2)$ be two embeddings of $\mathcal{O}$ such that $(X_1,i_1)\succsim (X_2,i_2)$. Assume that these embeddings are both projective (hence complete) with very ample linearizations
\begin{align*}
\mathbb{L}_1\in \mbox{Pic}(X_1)^G\ , \ \mathbb{L}_2\in \mbox{Pic}(X_2)^G 
\end{align*}
 satisfying
\begin{align*}
\varphi^*(\mathbb{L}_2)\cong \mathbb{L}_1 \ .
\end{align*}
This is essentially definition 1.2.1 of \cite{alexeev-brion2006}. Observe that the induced map of $G$ modules
\begin{align*}
\varphi^*:H^0(X_2,\mathbb{L}_2) \ra H^0(X_1,\mathbb{L}_1) 
\end{align*}
is injective, hence its dual map
\begin{align*}
(\varphi^*)^t:H^0(X_1,\mathbb{L}_1)^{\vee}\ra H^0(X_2,\mathbb{L}_2)^{\vee}
\end{align*}
is surjective and gives a rational map on the projectivizations of these spaces. The whole set up may be pictured as follows
\[
 \xymatrix{
 & X_1 \ar[dd]^{\varphi}\ar@{^{(}->}[r]^>>>{f_1 }&\  \mathbb{P}(H^0(X_1,\mathbb{L}_1)^{\vee})\ar@{-->}[dd]^{(\varphi^*)^t} \\
 \mathcal{O}\ar[ru]^{i_1}\ar[rd]_{i_2}& &\\
 & X_2 \ar@{^{(}->}[r]^>>>{f_2}& \mathbb{P}(H^0(X_2,\mathbb{L}_2)^{\vee})}
 \]
We isolate some features of this situation.
\begin{enumerate}
\item There are $u_i \in H^0(X_i,\mathbb{L}_i)^{\vee}\setminus\{0\}$ such that $X_i\cong f_i(X_i)=\overline{G\cdot [u_i]}$ $(i=1,2)$ . \\
\ \\
\item $(\varphi^*)^t(u_1)=u_2$ \ . \\
\ \\
\item $\mbox{Span}(G\cdot u_i )=H^0(X_i,\mathbb{L}_i)^{\vee}$ \ . \\
\ \\
\item The map $(\varphi^*)^t:G\cdot [u_1]\ra G\cdot [u_2]$ extends to a regular map on the Zariski closures of these orbits.
\end{enumerate}

We abstract (1)-(4) as follows.
Let $G$ be a complex reductive linear algebraic group.  We consider pairs $(\mathbb{V}; v)$ such that the linear span of $G\cdot v$ coincides with $\mathbb{V}$. Recall from section 2 that for any vector space $\mathbb{V}$ and any $v\in \mathbb{V}\setminus\{0\}$ we  let $[v]\in\mathbb{P}(\mathbb{V})$ denote the line through $v$. If $\mathbb{V}$ is a $G$ module define $\mathcal{O}_{v}:=G\cdot [v]\subset \mathbb{P}(\mathbb{V})$ the projective orbit of $[v]$ . Recall that $\overline{\mathcal{O}}_{v}$ denotes the Zariski closure of this orbit.

All of the author's work on the standard conjectures revolves around the following definition. The only reference to the definition in the context of representation theory known to the author is \cite{smirnov2004} .  

\begin{definition}\label{dominate} \emph{ 
$(\mathbb{V}; v)$  \textbf{\emph{dominates}} $(\mathbb{W}; w)$, in which case we write $(\mathbb{V}; v)\succsim (\mathbb{W}; w)$ ,  
 if and only if there exists $\pi\in Hom(\mathbb{V},\mathbb{W})^G$ such that
$ \pi(v)=w$ and the induced  rational map  
$\pi:\mathbb{P}(\mathbb{V}) \dashrightarrow  \mathbb{P}(\mathbb{W})$
restricts to a regular (finite) map}  
$\pi:\overline{\mathcal{O}}_{v}\ra \overline{\mathcal{O}}_{w} \ $ \emph{between the Zariski closures of the orbits.}
  \end{definition}
The situation may be pictured as follows.
\[
 \xymatrix{
 & \overline{\mathcal{O}}_v\ar[dd]^{\pi}\ar@{^{(}->}[r] & \mathbb{P}(\mathbb{V})\ar@{-->}[dd]^{\pi} \\
 \mathcal{O}\ar[ru]^{i_v}\ar[rd]_{i_w}& &\\
 &  \overline{\mathcal{O}}_w\ar@{^{(}->}[r] & \mathbb{P}(\mathbb{W})}
 \]
Observe that the restriction of the map $\pi$  to $\overline{\mathcal{O}}_v$ is regular if and only if the following holds
  \begin{align}\label{disjoint}
 \qquad  \overline{\mathcal{O}}_v\cap \mathbb{P}(\ker \pi)=\emptyset \ .
\end{align}
The reader should check that whenever $(\mathbb{V}; v)\succsim (\mathbb{W}; w)$ it follows that
\begin{align*} 
 & \pi(\mathbb{V})=\mathbb{W} \ \mbox{and} \ \mathbb{V}=\ker(\pi)\oplus \mathbb{W} \ \mbox{ ($G$-module splitting) } \ .
\end{align*}
Therefore we may identify $\pi$ with projection onto $\mathbb{W}$ and $v$ decomposes as follows
\begin{align*}
v=(v_{\pi},w) \ , \ \ker(\pi)\ni v_{\pi}\neq 0 \ .
\end{align*}

Again the reader can easily check that $(\ref{disjoint})$ is equivalent to
\begin{align}\label{project}
 \qquad \overline{G\cdot[(v_{\pi},w)]}\cap\overline{G\cdot[(v_{\pi},0)]}=\emptyset \quad \mbox{( Zariski closure in  $\mathbb{P}(\ker(\pi)\oplus\mathbb{W}$ ) )} \ .
\end{align}
We reformulate (\ref{project}) as follows. 
Given $(v\in \mathbb{V}\setminus \{0\}\ ;\ w\in \mathbb{W}\setminus \{0\})$ we consider the projective orbits
\begin{align*}
\mathcal{O}_{vw}:=G\cdot[(v,w)]\subset \mathbb{P}(\mathbb{V}\oplus\mathbb{W}) \ , \ \mathcal{O}_{v}:=G\cdot[(v,0)]\subset \mathbb{P}(\mathbb{V}\oplus\{0\})
\end{align*}

 \begin{definition}\label{pair}
\emph{The pair $(v,w)$ is \textbf{\emph{semistable}} if and only if}  $ \overline{\mathcal{O}}_{vw}\cap\overline{\mathcal{O}}_{v}=\emptyset $ .
  \end{definition} 
We stress that semistability and dominance are equivalent
\begin{align*}
  (\mathbb{V}\oplus\mathbb{W}; (v,w))\succsim (\mathbb{W}; w) \Leftrightarrow  \overline{\mathcal{O}}_{vw}\cap\overline{\mathcal{O}}_{v}=\emptyset \ .
  \end{align*}
\begin{example}\label{hmss} \emph{Let $\mathbb{V}=\mathbb{C}, v=1$ (the trivial one dimensional representation). Let $\mathbb{W}$ be any $G$ module.
Then 
\begin{align*}
(\mathbb{C}\oplus\mathbb{W}; (1,w))\succsim (\mathbb{W}; w) \ \mbox{if and only if} \  
0\notin \overline{G\cdot w}\subset \mathbb{W} \quad (\mbox{\emph{affine} closure} ) \ .
\end{align*}
 In other words , the pair $(1,w)$ is semistable if and only if $w$ is semistable in the usual (Hilbert-Mumford) sense.}
\end{example} 
\subsection{Toric morphisms and numerical semistability}\label{toricmorphs}
In this subsection we study the dominance relation $(X_1,i_1)\succsim (X_2,i_2)$ in the special case $G\cong (\mathbb{C}^*)^n$, an algebraic torus.  In this situation our new semistability condition admits a description in terms of certain polytopes generalizing the numerical criterion of Geometric Invariant Theory (see \cite{dolgachev} pg.137 Theorem 9.2). To begin the discussion let's denote our torus by $H$. Let $\chi\in M_{\mathbb{Z}}$ be an $H$ character and $u\in N_{\mathbb{Z}}$ an algebraic one parameter subgroup satisfying $<\chi,u>=1$ . 
\[ 
\xymatrix{
1\ar[r]&T:=\mbox{Ker}(\chi)\ar@{^{(}->}[r]&H\ar[r]^u_{\chi}   & \mathbb{C}^*\ar@/_1pc/[l]\ar[r]&1}
\]
Let $\mathscr{A}, \mathscr{B} \subset M_{\mathbb{Z}}$ be (nonempty) finite subsets satisfying $a(u(\alpha))\equiv 1 , b(u(\alpha))\equiv 1$ for all $a\in \mathscr{A}$ and $b\in \mathscr{B}$ and all $\alpha \in \mathbb{C}^*$ . Define
\begin{align*}
\mathscr{A}_+:=\{\chi\}\cup \{a+\chi\ |\ a\in \mathscr{A}\ \}\ , \ \mathscr{B}_+:=\{\chi\}\cup \{b+\chi\ |\ b\in \mathscr{B}\ \}\ .
\end{align*}
There are two naturally associated $H$ representations $\mathbb{C}^{\mathscr{A}_+}$ and $\mathbb{C}^{\mathscr{B}_+}$ given by
\begin{align*}
&\mathbb{C}^{\mathscr{A}_+}=\mbox{span}\{\mathbf{e}_{\chi}\ ,\ \mathbf{e}_{\chi+a}\ | \ a\in \mathscr{A}\}\ , \  \mathbb{C}^{\mathscr{B}_+}=\mbox{span}\{\mathbf{f}_{\chi}\ ,\ \mathbf{f}_{\chi+b}\ | \ b\in \mathscr{B}\} \\
\ \\
&h\cdot (c_{\chi}\mathbf{e}_{\chi}+\sum_{a\in\mathscr{A}}c_a\mathbf{e}_{\chi+a})=\chi(h)(c_{\chi}\mathbf{e}_{\chi}+\sum_{a\in\mathscr{A}}a(h)c_a\mathbf{e}_{\chi+a})\\
\ \\
&h\cdot (c_{\chi}\mathbf{f}_{\chi}+\sum_{b\in\mathscr{B}}c_b\mathbf{f}_{\chi+b})=\chi(h)(c_{\chi}\mathbf{f}_{\chi}+\sum_{b\in\mathscr{B}}b(h)c_b\mathbf{f}_{\chi+b})\\
\ \\
& \mathbb{C}^{\mathscr{A}_+}\ni v:= \mathbf{e}_{\chi} +\sum_{a\in\mathscr{A}} 
\mathbf{e}_{\chi+a} \ ,\ \mathbb{C}^{\mathscr{B}_+}\ni w:= \mathbf{f}_{\chi} +\sum_{b\in\mathscr{B}} \mathbf{f}_{\chi+b}\ .
\end{align*}
We define $T$ equivariant maps into projective spaces  in the usual manner:
\begin{align*}
&\varphi_{\mathscr{A}}:T\ra \mathbb{P}^{|\mathscr{A}|}:=\mathbb{P}(\mathbb{C}^{\mathscr{A}_+}) \ , \ \varphi_{\mathscr{A}}(t):=\big[ \mathbf{e}_{\chi} +\sum_{a\in\mathscr{A}}a(t)\mathbf{e}_{\chi+a}\big] \\
\ \\
&\varphi_{\mathscr{B}}:T\ra \mathbb{P}^{|\mathscr{B}|}:= \mathbb{P}(\mathbb{C}^{\mathscr{B}_+})\ , \ \varphi_{\mathscr{B}}(t):=\big[ \mathbf{f}_{\chi} +\sum_{b\in\mathscr{B}}b(t)\mathbf{f}_{\chi+b}\big] \\
\end{align*}
Let $\pi\in Hom(\mathbb{C}^{\mathscr{A}_+},\mathbb{C}^{\mathscr{B}_+})^H$ be such that $\pi(v)=w$. Observe that these requirements force that the following conditions are met
\begin{align*}
\mathscr{B}\subseteq\mathscr{A} \ \mbox{and}\ \ker(\pi)=\bigg\{\sum_{a\in \mathscr{A}\setminus\mathscr{B}}c_a\mathbf{e}_{\chi+a}\ |\ c_a\in \mathbb{C}\ \bigg\} \ .
\end{align*}
Then we have exactly the same set up as before
\[
 \xymatrix{
 & X_{\mathscr{A}}:=\overline{\varphi_{\mathscr{A}}(T)}\ar[dd]^{\pi}\ar@{^{(}->}[r] & \mathbb{P}^{|\mathscr{A}|}\ar@{-->}[dd]^{\pi} \\
  {T}\ar[ru]^{\varphi_{\mathscr{A}}}\ar[rd]_{\varphi_{\mathscr{B}}}& &\\
 &  X_{\mathscr{B}}:=\overline{\varphi_{\mathscr{B}}(T)}\ar@{^{(}->}[r] & \mathbb{P}^{|\mathscr{B}|}}
 \]
 
 Recall from our previous discussion that the map $\pi$ extends to $X_{\mathscr{A}}\setminus \varphi_{\mathscr{A}}(T)$ if and only if $(X_{\mathscr{A}}\setminus \varphi_{\mathscr{A}}(T))\cap \mathbb{P}(\mbox{ker}(\pi))=\emptyset $. To test for this latter condition we study 
\begin{align*}
\lambda^u(0)\cdot [v]=\varphi_{\mathscr{A}}(\lambda^u(0))\quad \mbox{for all $u\in N_{\mathbb{Z}}$} \ ,
\end{align*}
where $\lambda^u$ is the one parameter subgroup corresponding to $u$.
\begin{align*}
\lambda^u(t)\cdot v= \sum_{a\in\mathscr{A}\setminus\mathscr{B}}t^{(u,a)}\mathbf{e}_{\chi+a}+\big(\mathbf{e}_{\chi}+\sum_{b\in \mathscr{B}}t^{(u,b)}\mathbf{e}_{\chi+b}\big) \ .
\end{align*}
 It follows at once that $\varphi_{\mathscr{A}}(\lambda^u(0))\in \mathbb{P}(\mbox{ker}(\pi))$ if and only if
 \begin{align*}
 \min_{a\in\mathscr{A}\setminus\mathscr{B}}(u,a) <\ \min\{0,\min_{b\in\mathscr{B}}(u,b) \} \ .
 \end{align*}
Therefore a necessary condition that $\pi$ extend to the closure of the torus orbit is the following
\begin{align*}
\min\{0,\min_{b\in\mathscr{B}}(u,b) \}\leq \min_{a\in\mathscr{A}\setminus\mathscr{B}}(u,a) \ \mbox{for all $u\in N_{\mathbb{Z}}$} \ . \quad (*)
\end{align*}
Observe that $(*)$ is equivalent to the following polyhedral containment 
\begin{align*}
\mathbf{conv}(\mathscr{A}\setminus\mathscr{B})\subseteq \mathbf{conv}(\{0\}\cup \mathscr{B})\ .
\end{align*}
In fact, this necessary condition is also sufficient.
\begin{proposition}
\emph{The map $\pi$ extends to $X_{\mathscr{A}}$ if and only if $\mathbf{conv}(\mathscr{A}\setminus\mathscr{B})\subseteq \mathbf{conv}(\{0\}\cup \mathscr{B})$.}
\end{proposition}
The proposition is a consequence of the following lemma .
\begin{lemma}\label{p4tori}
\emph{Let $[y]\in X_{\mathscr{A}}\setminus \varphi_{\mathscr{A}}(T)$. Then there exists a $u\in N_{\mathbb{Z}}$ such that }
\begin{align*}
\varphi_{\mathscr{A}}(\lambda^u(0))=[y]\ .  
\end{align*}
\end{lemma}

The purpose of this discussion is to motivate the following definition.

\begin{definition}\label{numerical}   \emph{Let $\mathbb{V},\mathbb{W}$ be $G$-modules. The pair $(v\in\mathbb{V}\setminus\{0\} \ , \ w\in \mathbb{W}\setminus\{0\} )$ is \textbf{\emph{numerically semistable}} if and only if  $\mathcal{N}(v)\subset \mathcal{N}(w)$    for all maximal algebraic tori $H\leq G$.
  }
 \end{definition}

 \begin{proposition}\label{conj2} \emph{Semistability implies numerical semistability} :
\begin{align*}
\overline{\mathcal{O}}_{vw}\cap\overline{\mathcal{O}}_{v}=\emptyset \Rightarrow  \mathcal{N}(v)\subset \mathcal{N}(w) \ \mbox{ \emph{for all maximal algebraic tori} $H\leq G$ \ .   } 
\end{align*}
\end{proposition}

We have formulated stability in terms of arbitrary finite dimensional $G=\slnc$-modules $\mathbb{V}$ and $\mathbb{W}$. In this paragraph we assume that the modules are not only both \emph{irreducible} but satisfy a further condition which we will now consider.  The reader should check that this condition is satisfied by the modules that appear in the applications of stability to the Mabuchi enegy. To begin let $\lambda_{\bull}$ be a partition consisting of $N$ parts 
\begin{align}
\lambda_{\bull}=(\lambda_1\geq \lambda_2\geq \dots\geq \lambda_N\geq \lambda_{N+1}=0)\ .
\end{align}
Let $\elam$ denote the corresponding irreducible representation of $G$  with highest weight $\lambda_{\bull}$ (with respect to a choice of Borel $B$) . Let $W$ denote the Weyl group of $G$ with respect to $B$, it is well known the weight polytope of $\elam$ is given by
\begin{align}
\mathcal{N}(\lambda_{\bull})=\mbox{convexhull} \ \big\{W\cdot {\lambda_{\bull}}\big\} \ ,
\end{align}  
 where $W\cdot {\lambda_{\bull}}$ denotes the orbit of the highest weight under the action of the Weyl group.
  Given two partitions $\lambda_{\bull}$  and $\mu_{\bull}$ recall that $\mu_{\bull}$ \emph{dominates} $\lambda_{\bull}$ (written 
 $ \lambda_{\bull} \trianglelefteq \mu_{\bull}$ ) if and only if for all $1\leq i\leq N$ we have
  \begin{align}
\begin{split}
  \sum_{k=1}^i\lambda_k\leq \sum_{k=1}^i\mu_k \ .
\end{split}
\end{align}
 
 Consider two irreducible $G$ modules $\mathbb{V}=\elam$ and $\mathbb{W}=\emu$. 
 Let $(v,w)$ be a pair of non-zero vectors in these modules. We have the following proposition. 
\begin{proposition}  
\emph{If $(v,w)$ is semistable then $\lambda_{\bull} \trianglelefteq \mu_{\bull}$} .
 \end{proposition}
 \begin{proof}{The proposition follows at once from the fact that}
 \begin{align}
\lambda_{\bull} \trianglelefteq \mu_{\bull} \ \mbox{if and only if}\ \mathcal{N}(\lambda_{\bull})\subseteq \mathcal{N}(\mu_{\bull}) \ .
\end{align} 
 \end{proof}
 Dominance is not sufficient to guarantee  that the locus of semistable pairs is non-empty.  
 \begin{example}
\emph{Let $\mathbb{V}_e$ and  $\mathbb{V}_d$ be irreducible $SL(2,\mathbb{C})$ modules with highest weights $e , d \in\mathbb{N}$. These are well known to be spaces of homogeneous polynomials in two variables. Let $f$ and $g$ be two such polynomials in $\mathbb{V}_e\setminus\{0\}$ and $\mathbb{W}_d\setminus\{0\}$
respectively. Then the pair $(f,g)$ is numerically semistable if and only if
\begin{align}\label{d-e/2}
e\leq d \ \mbox{and for all $p\in \mathbb{P}^1$}\ \mbox{ord}_p(g)-\mbox{ord}_p(f)\leq \frac{d-e}{2} \ .
\end{align}
This can be proved by factoring the polynomials.
In particular when $e=0$ and $f=1$ we see that $(1,g)$ is numerically semistable if and only if 
\begin{align}
 \mbox{ord}_p(g) \leq \frac{d}{2} \ \mbox{for all $p\in \mathbb{P}^1$}\ .
\end{align}
 Assume that $e=d-1$ and $(f,g)$ is numerically semistable. Then by (\ref{d-e/2}) we get
\begin{align}\label{d=e}
\mbox{ord}_p(g)\leq \mbox{ord}_p(f) \quad \mbox{for all $p\in \mathbb{P}^1$} \ .
\end{align}
Let $\{p_1,p_2,\dots,p_d\}$ be the zeros of $g$ on $\mathbb{P}^1$ counted with multiplicity. Then by (\ref{d=e}) we have
\begin{align}
d=\sum_{1\leq i\leq d}\mbox{ord}_{p_i}(g)\leq \sum_{1\leq i\leq d}\mbox{ord}_{p_i}(f)\leq d-1 \ .
\end{align}
Therefore there is no numerically semistable pair in $\mathbb{V}_{d-1}\oplus\mathbb{V}_d$.
When $d=e$ the pair $(f,g)$ is numerically semistable if and only if
\begin{align}
\mathbb{C}f=\mathbb{C}g \ . 
\end{align}}
\end{example}
 Inspired by the magisterial work of Sato and Kimura (see \cite{kimura} and references therein) we pose the following problem.
 \begin{problem}
 \emph{Let $G$ be a reductive algebraic group. Classify all $( G ,\mathbb{V}, \mathbb{W} )$ with empty semistable locus .}
 \end{problem}

 
A reasonable conjecture to make is that semistability and numerical semistability are \emph{equivalent}.
 \begin{conjecture}\label{numericalcriterion} \emph{ The pair $(v,w)$  is numerically semistable if and only if $\overline{\mathcal{O}}_{vw}\cap\overline{\mathcal{O}}_{v}=\emptyset$ .}
 \end{conjecture}
The point of the conjecture, as the reader has surely seen, is that we would like to actually \emph{check} whether or not a given $(v,w)$ is semistable, that is, we would like to know if the Zariski closures of the orbits $\mathcal{O}_{vw}$ and $\mathcal{O}_v$ are disjoint. The conjecture, if true, converts this problem into a ``polyhedral-combinatorial'' problem which may be easier to solve.

\begin{remark}
\emph{When $\mathbb{V}$ is the trivial one dimensional representation and $v=1$ the conjecture reduces to the well known Hilbert-Mumford criterion in Geometric Invariant Theory. In particular the conjecture is true in this case.}
\end{remark}
Conjecture \ref{numericalcriterion} is closely related to the following condition, \emph{Property (P)}, 
which may or may not hold for a complex linear algebraic group $G$. Observe that this property is the projective version of  Property (A) from \cite{birkes71}. The reader should also consult Theorem 1.4 of   
 \cite{kempf78}.

Once more we fix notation. Let $\mathbb{W}$ be a finite dimensional rational complex representation of $G$. Let $w\in\mathbb{W}\setminus \{0\}$. Then $\mathcal{O}_w$ denotes the \emph{projective} orbit:
\begin{align}
\mathcal{O}_w:= G\cdot [w] \subset \mathbb{P}(\mathbb{W})\ .
\end{align}
\ \\
\noindent \textbf{Property (P) .}  \emph{ If $\rho:G\ra GL(\mathbb{W})$ is a finite dimensional complex rational representation of $G$, and if $\mathscr{Z}$ is a non-empty $G$ invariant closed subvariety of $\overline{\mathcal{O}}_w\setminus {\mathcal{O}}_w$ , then there exists an element $z_0$ of  $\mathscr{Z}$ and an algebraic one-parameter subgroup $\lambda:\mathbb{C}^*\ra G$ such that}
\begin{align}
\lim_{\alpha\ra 0}\lambda(\alpha)\cdot [w]=z_0 \ .
\end{align}
If Property (P) holds for reductive groups then Conjecture \ref{numericalcriterion} follows at once. 
\begin{remark}
\emph{Observe that when $G=T$ an algebraic torus, Property (P) holds by Lemma \ref{p4tori}.}
\end{remark}
\begin{remark}
\emph{Property (P) does not require that \emph{every} boundary point $z$ is ``accessible'' but that \emph{some} point on $\mathscr{Z}$ can be reached by a one-parameter subgroup.}
\end{remark}
 
 In Lieu of Property (P) we have the following useful result.
 \begin{proposition}\label{valuative}(See Mumford \cite{git}) \emph{Let $G$ be a complex algebraic group. Let $\mathbb{V}$ be a finite dimensional complex rational representation of $G$. Let $v\in \mathbb{V}\setminus \{0\}$. Then for any $z$ satisfying}
 \begin{align}
 z\in \overline{\mathcal{O}}_v\setminus \mathcal{O}_v \subset \mathbb{P}(\mathbb{V}) 
 \end{align}
 \emph{there exists a $\mathbb{C}((t))$ valued point $\gamma(t)$ of $G$ such that}
 \begin{align}
 \lim_{t\ra 0}\gamma(t)\cdot [v]=z \ .
 \end{align}
 \emph{If $G$ is reductive then there exists a one parameter subgroup $\lambda(t)$ and two $\mathbb{C}[[t]]$ valued points $l(t),\sigma(t)$ of $G$ such that}
\begin{align}\label{iwahori}
\gamma(t)= l(t)\lambda(t)\sigma(t)   \ .
\end{align}
 \end{proposition}
\begin{remark}
 \emph{For our purpose we only need $G=SL(N+1,\mathbb{C})$. In which case the decomposition (\ref{iwahori}) follows from elementary row and column operations over $\mathbb{C}[[t]]$ .}
 \end{remark}

Next we provide a simple example of an inaccessible boundary point of an orbit $\mathcal{O}$. 
\begin{example}\emph{Let $G=SL(2,\mathbb{C})$. Let $d\geq 2 \in \mathbb{Z}$. Let $\mathbb{W}:=\mathbb{C}\oplus \mbox{Sym}^{d+1}(\mathbb{C}^2)$. $G$ acts trivially on $\mathbb{C}$ and by the standard action on the second summand. Consider the orbit:
\begin{align}
\mathcal{O}:= G\cdot [(1 , e_1^d\cdot e_2)] \subset \mathbb{P}(\mathbb{W} ) \ .
\end{align}
Define
\begin{align*}
\sigma(t):= \begin{pmatrix}
t& t^{-d}\\
0&t^{-1}
\end{pmatrix}\ .
\end{align*}
Then 
\begin{align}
\lim_{t\ra 0}\sigma(t)\cdot [(1,e_1^d\cdot e_2)]=[(1,e_1^{d+1})]\in \overline{\mathcal{O}}\  .
\end{align}
  However, the stabilizer of this point is 
\begin{align}
\{ \begin{pmatrix} \zeta & z\\
0&\zeta^{-1} 
\end{pmatrix} \ |\ \zeta \in \mathbb{Z}_{d+1} \ , \ z\in \mathbb{C}\ \} \ .
\end{align}
This group contains no one parameter subgroup. Therefore this point cannot be reached from $\mathcal{O}$ by a one parameter subgroup of $G$.  }
\end{example}
\subsection{A Kempf-Ness type functional}\label{kempfness} In this section we study semistability in terms of a Kempf-Ness type functional. In fact it is from this point of view that the author arrived at the definition of semistability. As always $(\mathbb{V}, v)$ and $(\mathbb{W}, w)$ are finite dimensional complex rational representations of $G$ together with a pair of nonzero vectors. We equip $\mathbb{V}$ and $\mathbb{W}$ with Hermitian norms .  We are interested in the function on $G$ which we call the \emph{energy of the pair} $(v,w)$ :
 \begin{align}\label{energy}
 G\ni \sigma\ra \nu_{v,w}(\sigma):=\log||\sigma\cdot w||^2-\log||\sigma\cdot v||^2 \ .
 \end{align}
 Then we have the following fact.
 \begin{proposition}\label{vwlowerbound}
\emph{  $\nu_{vw}$ is bounded from below on $G$ if and only if $(v,w)$ is semistable.}
 \end{proposition}
 The proposition is a consequence of the following observation.
 \begin{lemma}\label{distance}
 \begin{align*}
 \nu_{v,w}(\sigma)=\log\tan^2d_g(\sigma\cdot [(v,w)] , \sigma\cdot [(v,0)]) \ ,
 \end{align*}
 \emph{where $d_g$ denotes the distance in the Fubini-Study metric on} $\mathbb{P}(\mathbb{V}\oplus\mathbb{W})$ .
 \end{lemma}
\begin{proof}
 Let $u,v\in \mathbb{V}$ and let $(\cdot,\cdot)$ be any Hermitian inner product on $\mathbb{V}$ with associated Fubini-Study metric $g$ on $\mathbb{P}(\mathbb{V})$. Recall the distance formula
\begin{align*}
\cos d_g([u],[v])=\frac{|(u,v)|}{||u||||v||}\ .
\end{align*}
With the orthogonal direct sum Hermitian form on $\mathbb{V}\oplus\mathbb{W}$ we have for any $\sigma\in G$
\begin{align*}
\cos d_g(\sigma\cdot [(v,w)] , \sigma\cdot [(v,0)])&= \frac{|(\sigma\cdot [(v,w)] , \sigma\cdot [(v,0)])|}{\sqrt{||\sigma\cdot v||^2+||\sigma\cdot w||^2}||\sigma\cdot v||}\\
 \ \\
 &=\frac{||\sigma\cdot v||}{\sqrt{||\sigma\cdot v||^2+||\sigma\cdot w||^2}} \ .
\end{align*}
Since
\begin{align*}
\sec^2d_g(\sigma\cdot [(v,w)] , \sigma\cdot [(v,0)])&=1+\tan^2d_g(\sigma\cdot [(v,w)] , \sigma\cdot [(v,0)]) \ ,
\end{align*}
we may conclude that
\begin{align*} 
\tan^2d_g(\sigma\cdot [(v,w)] , \sigma\cdot [(v,0)])  = \frac{||\sigma\cdot w||^2}{||\sigma\cdot v||^2}\ .
\end{align*}
Now take the log of both sides.
\end{proof}
Observe that for any $\sigma,\tau \in G$ we have the inequality
\begin{align}\label{sigma-tau}
d_g(\sigma\cdot [(v,w)] , \sigma\cdot [(v,0)])\leq d_g(\sigma\cdot [(v,w)] , \tau\cdot [(v,0)])\ .
\end{align}
 As a corollary of (\ref{sigma-tau}) and lemma \ref{distance} we have the much more refined version of Proposition \ref{vwlowerbound}.
\begin{corollary}\label{infi}\emph{ The infimum of the energy of the pair $(v,w)$ is as follows}
 \begin{align}\label{inf}
 \inf_{\sigma\in G}\nu_{vw}(\sigma)=\log\tan^2d_g(\overline{\mathcal{O}}_{vw},\overline{\mathcal{O}}_{v}) \ .  
 \end{align}  
 \end{corollary}
 Our whole approach to the standard conjectures in K\"ahler Geometry is based on this identity.

 Corollary \ref{infi} and Proposition \ref{valuative} imply the following result.
 \begin{proposition}\label{curveselection}
\emph{Assume that
\begin{align}
\inf_{\sigma\in G}\nu_{vw}(\sigma)=-\infty \ .
\end{align}
Then there exists a positive number $\delta$, a holomorphic mapping $\sigma :D_{\delta}\ra G $, and an algebraic one parameter subgroup $\lambda$ of $G$ such that}
\begin{align}
\lim_{\alpha\ra 0}\nu_{vw}( {\gamma(\alpha)})= -\infty \ \qquad \gamma(\alpha) := \lambda(\alpha)\cdot \sigma(\alpha) \qquad   \ .
\end{align}
\end{proposition}
  $D_{\delta} $ denotes the $\delta$ disk in $\mathbb{C}$. 
\subsection{The Classical Futaki Character} Given the relationship between the Mabuchi energy and the semistability of pairs, we expect to be able to capture the classical Futaki invariant (see \cite{Futaki}) in terms of  representation theory and polyhedral geometry. The purpose of this section is to show that this is indeed the case.
As in the preceding sections  $G$ denotes a reductive complex linear algebraic group. $\mathbb{V}$, $\mathbb{W}$ are finite dimensional complex rational representations of $G$. Let $v\in \mathbb{V}\setminus \{0\}$ and $w\in \mathbb{W}\setminus \{0\}$ . As usual $[v]$ denotes the corresponding point in the projective space $\mathbb{P}(\mathbb{V})$ and $G_{[v]}$ denotes the stabilizer of the line through $v$. Therefore there is a \emph{character}
\begin{align*}
\chi_{v}:G_{[v]}\ra \mathbb{C}^* \ , \ \sigma\cdot v= \chi_{v}(\sigma)\cdot v \ .
\end{align*}
\begin{definition} \emph{Let $v\in \mathbb{V}\setminus \{0\}$ and $w\in \mathbb{W}\setminus \{0\}$ . Then the \textbf{\emph{automorphism group}} of the pair $(v,w)$ is
the algebraic subgroup of $G$ given by }
\begin{align*}
Aut(v,w):=G_{[v]}\cap G_{[w]} \ .
\end{align*}
\end{definition}
Let $\mathfrak{aut}(v,w)$ denote the Lie algebra of  $Aut(v,w)$ .
\begin{definition} \emph {Let $\mathbb{V} , \mathbb{W}$ be finite dimensional complex rational representations of $G$. Let $v,w$ be two nonzero vectors in $\mathbb{V} , \mathbb{W}$ respectively. Then the \textbf{\emph{Futaki character}} of the pair $(v,w)$ is  the algebraic homomorphism
\begin{align}
F:=\chi_{w}\chi_{v}^{-1}: Aut(v,w) \ra \mathbb{C}^* \ 
\end{align}
induced by the one dimensional representation $\mathbb{C}w\otimes (\mathbb{C}v )^{\vee}$ . We set $ {F}_*$ to be the corresponding Lie algebra character
\begin{align*}
{F}_{*}:={d\chi_{w}}-{d\chi_{v}}:\mathfrak{aut}(v,w)\ra \mathbb{C} \ 
\end{align*}
where ${d\chi_{v}}$ denotes the differential of $\chi_{v}$ at the identity.}
\end{definition}
\begin{remark}
\emph{At this point the order is not important. That is, we could equally well consider $\chi_{w}^{-1}\chi_{v}$ .}
\end{remark}
Let $\tau\in G$ and $\sigma\in G_{[v]}\cap G_{[w]}$, then the diagram below is commutative.
\begin{align*}
\xymatrix{
 \mathbb{C}w\otimes (\mathbb{C}v)^{\vee} \ar[d] ^{\chi_{w}\chi_{v}^{-1}(\sigma) } \ar[r]^{\alpha_{\tau}}&\mathbb{C}\tau\cdot w\otimes (\mathbb{C}\tau\cdot v)^{\vee} \ar[d]^{\chi_{\tau\cdot w}\chi_{\tau\cdot v}^{-1}(Ad_{\tau^{-1}}(\sigma))} \\
\mathbb{C}w\otimes (\mathbb{C}v)^{\vee} \ar[r] ^{\alpha_{\tau}}&\mathbb{C}\tau\cdot w\otimes (\mathbb{C}\tau\cdot v)^{\vee} 
 }  
 \end{align*}
 This shows that the Futaki character only depends on the \emph{orbit} of the pair $(v,w)$ .

We can decompose the identity component of $Aut(v,w)$
\begin{align*}
Aut^o(v,w)=S\rtimes U \ ,
\end{align*} 
 $S$ is reductive and $U$ is unipotent . Then we have that
 \ \\
 \begin{center}\emph{ $F$ is completely determined  on $S$}.  \end{center}
 \ \\
 Let $T\leq S$ be any maximal algebraic torus in $S$. Then we have that 
 \ \\
\begin{center}\emph{ $F $ is completely determined  on $T$}.  \end{center}
\ \\
  Let $M_{\mathbb{Z}}(T)$ denote the character group of $T$. Then $F\in M_{\mathbb{Z}}(T)$. Let \newline $N_{\mathbb{Z}}(T):=Hom(M_{\mathbb{Z}}(T), \mathbb{Z})$ denote the dual lattice. It is well known that
 \begin{align*}
 N_{\mathbb{Z}}(T)\cong\mbox{algebraic one parameter subgroups $\lambda$ of $T$ } \ .
 \end{align*}
Finally we have that
\begin{center}\emph{ $F $ is completely determined  on $N_{\mathbb{Z}}(T)$.}\end{center}
\begin{align*}
F : N_{\mathbb{Z}}(T)\ra \mathbb{Z} \ , \ F (\lambda)=<\chi_{w},\lambda>- <\chi_{v},\lambda> \ .
\end{align*}
In many cases one has that $Aut(v,w)$ is {trivial}. In such a situation we  introduce a generalization of $F$. In what follows $H$ denotes a maximal algebraic torus of $G$.
\begin{definition}
\emph{Let $\mathbb{V}$ be a rational representation of $G$. Let $\lambda$ be any degeneration in $H$  . The \textbf{\emph{weight}}  $w_{\lambda}(v)$  of $\lambda$ on $v\in \mathbb{V}\setminus\{0\}$ is the integer}
\begin{align*}
w_{\lambda}(v):= \mbox{\emph{min}}_{ \{ x\in \mathcal{N}(v)\}}\ l_{\lambda}(x)= \mbox{\emph{min}} \{ <\chi,\lambda>| \chi \in \mbox{\emph{supp}}(v)\}\ .
\end{align*}
 \noindent\emph{ Alternatively, $w_{\lambda}(v)$ is the unique integer such that}
\begin{align*}
\lim_{|t|\rightarrow 0}t^{-w_{\lambda}(v)}\lambda(t)v \  \mbox{ {exists in $\mathbb{V}$ and is \textbf{not} zero}}.
\end{align*}
\end{definition}
Let $\Delta(G)$ denote the space of algebraic one parameter subgroups of $G$.   
\begin{definition} \emph{The \textbf{\emph{generalized Futaki character}} of the pair $(v,w)$ is the map }
\begin{align*}
F_{gen} :\Delta (G)\ra \mathbb{Z}\ , \ F_{gen} (\lambda):=w_{\lambda}(w)-w_{\lambda}(v) \ .
\end{align*}
\end{definition}
\begin{proposition}\emph{$F_{gen}(\lambda)\leq 0$ for all $\lambda \in \Delta(G)$ if and only if $(v,w)$ is numerically semistable.  }
\end{proposition}
 
The energy of the pair $(v,w)$ and the generalized character $F_{gen}$ are related as follows (see also \cite{paul2011} (2.30) pg. 269 ) .
 \begin{proposition}\label{fdkenergy}
\emph{ Let $\lambda\in\Delta(G)$ . Then there is an asymptotic expansion as $|t|\ra 0$
 \begin{align*}
 \nu_{vw}(\lambda(t))=F_{gen}(\lambda)\log|t|^2+O(1) \ .  
 \end{align*}}
\end{proposition} 
 In particular for $\sigma\in Aut(v,w)$ we have
 \begin{align*}
 \nu_{vw}(\sigma) =\log|\chi_{w}(\sigma)|^2-\log|\chi_{v}(\sigma)|^2 \ . 
 \end{align*}

We study the relationship between the generalized and classical Futaki invariants.  We have, as in the previous section, the Levi decomposition of $Aut(v,w)^o$
\begin{align}
Aut(v,w)^o=S\ltimes U \ ,
\end{align}
$S$ is reductive and $U$ is the unipotent radical of ${Aut(v,w)^o}$. Let $T \leq S$ be any maximal algebraic torus (possibly trivial). Since
$S\leq G$ there is a maximal algebraic torus $H$ in $G$ containing $T$. Fix any such $H$.   Then we have the short exact sequence of lattices
 \begin{align*}
 0\ra L_{\mathbb{Z}}\xrightarrow{\iota}M_{\mathbb{Z}}(H)\xrightarrow{\pi_T}M_{\mathbb{Z}}(T)\ra 0 \ . \\
   \end{align*}
 
Recall that $\sigma\in Aut(v,w)$ acts on $w$ (resp. $v$)  via a character $\chi_w$ (resp. $\chi_v$ ) . We have the following.
\begin{proposition}\emph{ All the characters in the $H$ support of $w$ (or $v$) coincide upon restriction to $T$
\begin{align*}
\begin{split}
&\chi\in\mbox{supp}(w)\Rightarrow \pi_T(\chi)=\chi_w \\
\ \\
&\eta\in\mbox{supp}(v)\Rightarrow \pi_T(\eta)=\chi_v \ .
\end{split}
\end{align*}
Consequently, the difference of any two characters in supp($w$) (or supp($v$))  lies in $ {L_{\mathbb{Z}}}$ .  }
\end{proposition}
 
Extending scalars to $\mathbb{R}$ gives the sequence
\begin{align*}
0\ra {L_{\mathbb{R}}}\xrightarrow{ {\iota}} M_{\mathbb{R}}(H)&\xrightarrow{ {\pi_T}} M_{\mathbb{R}}(T)\ra 0\ .
\end{align*}

Then $\mathcal{N}(w)$ and $\mathcal{N}(v)$ both lay in affine subspaces of ${\mathbb{R}}^N$ modeled on $ {L_{\mathbb{R}}}$ . 
Now we suppose that 
\begin{align}
\mathcal{N}(v)\subseteq \mathcal{N}(w) \ .
\end{align}

Since ${ {\pi_T}}$ is {linear} we have that
\begin{align}
\{\chi_v\}={ {\pi_T}}\big(\mathcal{N}(v)\big)\subseteq {{\pi_T}}\big(\mathcal{N}(w)\big)=\{\chi_w\} \ .
\end{align}
We conclude 
\begin{align}
\chi_w=\chi_v \ . 
\end{align} 

We summarize the relationship between the character of the pair and numerical semistability.
\begin{proposition}
\emph{ \begin{align*}
 \begin{split}
&a)\ \mbox { $\mathcal{N}(v)$ and $\mathcal{N}(w)$ both lie in parallel affine subspaces of $M_{\mathbb{R}}(H)$ modeled on $L_{\mathbb{R}}$ . }\\
&\mbox{Precisely, for any choice of $\chi\in \mbox{supp}(v)$ and $\eta\in \mbox{supp}(w)$ we have }\\
& \mathcal{N}(v)\subset \mathbb{A}_v:=\chi + L_{\mathbb{R}} \ \mbox{and}\ \mathcal{N}(w)\subset \mathbb{A}_w:=\eta + L_{\mathbb{R}} \ . \\
\ \\
&b) \ F_*\equiv 0 \ \mbox{if and only if} \ \mathbb{A}_v=\mathbb{A}_w \ . \\
\ \\
&c)\ F_*\not\equiv 0 \ \mbox{if and only if}\ \mathbb{A}_v\cap\mathbb{A}_w=\emptyset \ .\\
\ \\
& d)\ \mbox{ $F_*\equiv 0$  whenever the pair $(v,w)$ is numerically semistable .}
 \end{split}
 \end{align*}} 
 \end{proposition} 
Let $X\ra \cpn$ be a smooth subvariety of $\cpn$ satisfying our usual hypotheses. Observe that 
\begin{align}
 {Aut}(X , \mathcal{O}(1)|_X)\cong  {Aut}(R(X),\Delta(X)) \ .
\end{align}
The present discussion is justified by the following result concerning the usual Calabi-Futaki invariant of $X$ (see (\ref{futakiinv}) ).
\begin{theorem} \emph{Let $v\in \mathfrak{aut}(X , \mathcal{O}(1)|_X)$ be a holomorphic vector field. Let $F_*$ denote the Lie algebra character of $(R,\Delta)$. Then the following identity holds
\begin{align}\label{calfut}
\int_{X}v (h_{\om})\om^n=F_*(v ) \ .
\end{align}}
\end{theorem}
\begin{proof} {The proof follows at once from Theorem A and Proposition \ref{fdkenergy} .  }
\end{proof}
\begin{corollary} \emph{The map}
\begin{align}
v \in \mathfrak{aut}(X , \mathcal{O}(1)|_X)\ra \int_{X}v (h_{\om})\om^n \in \mathbb{C}
\end{align}
\emph{is a Lie algebra character and is independent of the choice of K\"ahler metric in the class $[\om]$.} 
\end{corollary} 
\subsection{Highest weight polytopes} Let $\mathbb{E}$ be a finite dimensional complex $G$-module. Let $Y\subset \mathbb{P}(\mathbb{E})$ an irreducible $G$ invariant subvariety. Then $H^0(Y,\mathcal{O}(m))$ is a $G$ module for all $m\in \mathbb{N}$ .  Let $B\leq G$ be a Borel subgroup, and let $T\leq B$ denote a maximal algebraic torus of $B$.  Recall that there is a unique element $w_0$ of the Weyl group satisfying
\begin{align*}
-w_0:\Lambda_{\mathbb{Z}}(B)\ra \Lambda_{\mathbb{Z}}(B)\ .
\end{align*}

We define the space of  (dual) $B$-characters as follows 
\begin{align*}
&H^0(Y,\mathcal{O}(m))^{(B)}:= \\
&\{ -w_0\cdot\chi \in M_{\mathbb{Z}}(T)\ |\ \mbox{there exists}\ f\in H^0(Y,\mathcal{O}(m))\ \mbox{such that}\ b\cdot f=\chi(b)f\ \mbox{for all}\ b\in B \}\ .
\end{align*}
 Following Brion \cite{brion87} we define the \emph{\textbf{highest weight set}} of $Y$ as follows 
\begin{align}\label{brionpoly}
\mathscr{C}(Y):=\bigcup_{m\in \mathbb{N}}\frac{1}{m}H^0(Y,\mathcal{O}(m))^{(B)}\subset M_{\mathbb{Q}}(T)\ .
\end{align}

 Obviously we have that
\begin{align*}
\mathscr{C}(Y)\subset \Lambda_{\mathbb{Q}}(B) \ ,\  \mbox{the set of $B$-{dominant} rational weights  }\ .
\end{align*}
 
  \begin{theorem}(Mumford \cite{ness1984} , Brion \cite{brion87} )
\emph{  $\mathscr{P}(Y):= \overline{\mathscr{C}(Y)}\subset \Lambda_{\mathbb{R}}(B)$ is a compact convex rational polytope.  }
\end{theorem}  
The next result follows at once from Proposition \ref{conj2} and work of Franz (see \cite{franz2002}) and also Guillemin and Sjamaar (see \cite{guillemin&sjamaar2006}) .
 \begin{proposition}
\emph{ Let $\mathbb{V}$ and $\mathbb{W}$ be finite dimensional complex rational $G$-modules. {Let} $v$ and $w$ denote nonzero vectors in $\mathbb{V}$ and $\mathbb{W}$ respectively. If the pair $(v,w)$ is semistable then their  highest weight polytopes coincide}
 \begin{align*}
 \mathscr{P}(\overline{\mathcal{O}}_{vw})=\mathscr{P}(\overline{\mathcal{O}}_{w})\ . 
 \end{align*}
 \end{proposition}
 \subsection{Quasi-closed orbits and semistable pairs} This subsection is devoted to the examination of some simple examples of semistable pairs. Some of these examples are available in the literature but our point of view seems to be new.
  
  Let $\mathbb{E}$ be a finite dimensional \emph{reducible} representation of $G$. Let $u\in \mathbb{E}\setminus\{0\}$. Let $\mathcal{O}\subset \mathbb{P}(\mathbb{E})$ denote the projective orbit $G\cdot [u]$. We always assume that the linear span of $\mathcal{O}$ coincides with $\mathbb{P}(\mathbb{E})$ . Fix a Borel subgroup $B\leq G$ and a maximal algebraic torus $T\leq B$. Let $\Lambda^+$ denote the dominant integral weights relative to $B$. It is well known that  $\overline{\mathcal{O}}$ is a union of orbits at least one of which is {closed} and each closed orbit corresponds to an irreducible $G$ submodule $\emu$ of $\mathbb{E}$. We assume that $\overline{\mathcal{O}}$ consists of \emph{finitely many} orbits. Let $\Lambda^+(\mathcal{O})$ denote the dominant weights corresponding to the closed orbits in $\overline{\mathcal{O}}$. Then we have the decomposition
\begin{align*}
\overline{\mathcal{O}}=\mathcal{O}\cup \bigcup _{\mubull\in \Lambda^+(\mathcal{O})}G\cdot[w_{\mubull}] \cup \mathcal{O}_{1}\cup \dots \cup  \mathcal{O}_{k} \ .
\end{align*}
$w_{\mubull}$ is the corresponding highest weight vector . Now we decompose $\mathbb{E}$ according to the orbit $\mathcal{O}$
\begin{align*}
\mathbb{E}=\bigoplus _{\mubull\in \Lambda^+(\mathcal{O})}\emu \oplus \mathbb{V} \ .
\end{align*}
  We assume that $\mathbb{V}\neq 0$ . Let $\pi_{\mathcal{O}}$ and $\pi_{\mathbb{V}}$ denote the projections onto $\bigoplus _{\mubull\in \Lambda^+(\mathcal{O})}\emu$ and $\mathbb{V}$ respectively.  Then we may decompose $u$ as follows
  \begin{align*}
  u= (v,w):= (\pi_{\mathbb{V}}(u),\pi_{\mathcal{O}}(u)) \ .
  \end{align*}
Then $(v,w)$ is semistable if and only if  for every $1\leq i\leq k$ there exists a $\mubull\in \Lambda^+(\mathcal{O})$  such that $\pi_{\mubull}(x_i)\neq 0$  where $\mathcal{O}_i=G\cdot [x_i]$ and 
$\pi_{\mubull}$ is the projection onto $\emu$.

The simplest case  is when $\overline{\mathcal{O}}$ consists of {two} orbits (one of which is closed)
\begin{align*}
\overline{\mathcal{O}}= \mathcal{O}\cup G\cdot [w_{\mubull}] \ .
\end{align*}
In this case it is automatic that the pair $(\pi_{\mathbb{V}}(u),\pi_{\mathcal{O}}(u))$ is semistable. Therefore the class of two orbit varieties (or, more generally, \emph{quasi-closed orbits}) provides many interesting examples of semistable pairs. Such varieties seem to have been   classified by Stephanie Cupit-Foutou (see \cite{cupit2003}) and Alexander Smirnov (see \cite{smirnov2004}).
 \begin{example}\label{2x2}\emph{Let $G=SL(2,\mathbb{C})$ and consider the $G$ module 
$$\mathbb{E}=\mathbb{C}^2\otimes \mathbb{C}^2\cong\det(\mathbb{C}^2)\oplus S^2(\mathbb{C}^2)$$
Let $u=e_1\otimes e_2$. The orbit closure is the well known quadric surface in $\mathbb{P}^3$
\begin{align*}
\overline{\mathcal{O}}=\overline{G\cdot [u]}=\mathbb{P}^1\times\mathbb{P}^1\rightarrow \mathbb{P}^3=\mathbb{P}(\mathbb{C}^2\otimes \mathbb{C}^2) \ .
\end{align*}
This is a two orbit variety with (unique) closed orbit $G\cdot [e^2_1]=G\cdot [e^2_2]$. We may decompose $u$ as follows
\begin{align*}
u=e_1\wedge e_2+e_1\cdot e_2 \ .
\end{align*}
 $\mathbb{V}$ corresponds to the trivial one dimensional representation with  $\pi_{\mathbb{V}}(u)=e_1\wedge e_2=1$. $\mubull=(2,0)$ and $\pi_{\mathcal{O}}(u)= e_1\cdot e_2 $. So we deduce that the pair $(1,e_1\cdot e_2 )$ is semistable.  Example \ref{hmss} says that this is equivalent to the (classical) semistability of $e_1\cdot e_2 $ under the action of $G$ .}
\end{example}  
\begin{example}\label{blowup}\emph{
 Let $\psi:\mathbb{P}^2\times\mathbb{P}^2\dashrightarrow \mathbb{P}(\wedge^2\mathbb{C}^3)$ be the rational map $\psi([v],[w]):=[v\wedge w]$.
  $\Gamma_{\psi}$ denotes the graph of $\psi$
\begin{align*}
\Gamma_{\psi}:=\{([v],[w],[v\wedge w])\ |\ [v]\neq [w]\}\subset \mathbb{P}^2\times\mathbb{P}^2\times\mathbb{P}(\wedge^2\mathbb{C}^3) \ .
\end{align*}
Recall that the blow up of $\mathbb{P}^2\times\mathbb{P}^2$ along the diagonal $\Delta$  is the Zariski closure of $\Gamma_{\psi}$ inside $\mathbb{P}^2\times\mathbb{P}^2\times\mathbb{P}(\wedge^2\mathbb{C}^3)$ . We will denote the blow up by  $ B_{\Delta}(\mathbb{P}^2\times\mathbb{P}^2)$ and let $E\cong \mathbb{P}(T^{1,0}_{\mathbb{P}^2})$ denote the exceptional divisor. The situation can be pictured as follows
\begin{align*}
\xymatrix{
 \Gamma_{\psi}\subset B_{\Delta}(\mathbb{P}^2\times\mathbb{P}^2)\ar@{^{(}->}[r]^-{\iota} \ar[d]^{p_{12}} & \mathbb{P}^2\times\mathbb{P}^2\times\mathbb{P}(\wedge^2\mathbb{C}^3)   \ar[r]^-{S}\ar[d]^{p_3}& \mathbb{P}(\mathbb{E}_{310}\oplus \mathbb{C}^3\oplus S^2(\wedge^2\mathbb{C}^3)\oplus  \mathbb{C}^3)   \\
\mathbb{P}^2\times\mathbb{P}^2\ar@{-->}^{\psi}[r]&\mathbb{P}(\wedge^2\mathbb{C}^3)  }
\end{align*}
Then we claim that $B= B_{\Delta}(\mathbb{P}^2\times\mathbb{P}^2)$ is a two-orbit $G=SL(3,\mathbb{C})$ variety (for the natural $G$ action) with  orbit decomposition
\begin{align}
B=(B\setminus E) \cup E \ .
\end{align}
Where $(B\setminus E)$ is necessarily the open orbit. There is a $G$ equivariant identification
\begin{align*}
B\setminus E\cong \mathbb{P}^2\times\mathbb{P}^2\setminus \Delta \ .
\end{align*}
Since $G$ acts transitively on planes in $\mathbb{C}^3$ we easily get that $\mathbb{P}^2\times\mathbb{P}^2\setminus \Delta$ is an orbit:
\begin{align*}
G\cdot ([e_1],[e_2]) = \mathbb{P}^2\times\mathbb{P}^2\setminus \Delta \ .
\end{align*}
To see that $E$ is a homogeneous $G$ variety we can proceed as follows. We have the decomposition into  irreducible summands
\begin{align*}
\mathbb{C}^3\otimes \mathbb{C}^3\otimes \wedge^2\mathbb{C}^3\cong \mathbb{E}_{310}\oplus \mathbb{C}^3\oplus S^2(\wedge^2\mathbb{C}^3)\oplus  \mathbb{C}^3 \ .
\end{align*}
The summand $\mathbb{E}_{310}$ appears as follows
\begin{align*}
0\ra \mathbb{E}_{310}\cong\mbox{Ker}(\pi)\ra S^2(\mathbb{C}^3)\otimes \wedge^2(\mathbb{C}^3)\xrightarrow{\pi} \mathbb{C}^3\ra 0 \ ,
\end{align*}
where the map $\pi$ is defined by 
\begin{align*}
\pi(v\cdot w\otimes \alpha)=\alpha(v)w+\alpha(w)v \ .
\end{align*}
Note that
\begin{align*}
e_1^2\otimes (e_1\wedge e_2) \in \mbox{Ker}(\pi) \ .
\end{align*}
Since $e_1^2\otimes (e_1\wedge e_2)$ is a highest weight $(310)$ vector we see that $\mathbb{E}_{310}$ is  a summand of  $\mbox{Ker}(\pi)$.  Since these spaces have the same dimension (which is 15 by the Weyl dimension formula) they coincide. Next we observe that
\begin{align*}
([e_1+te_2], [e_1], [e_1\wedge e_2])\in \Gamma_{\psi} \quad \mbox{for all $t\in\mathbb{C}^*$}\ . 
\end{align*}
As $t\ra 0$ we have
\begin{align*}
([e_1+te_2], [e_1], [e_1\wedge e_2])\ra([e_1],[e_1],[e_1\wedge e_2]) \in E \ .
\end{align*}
Let $S:\mathbb{P}^2\times\mathbb{P}^2\times\mathbb{P}(\wedge^2\mathbb{C}^3)\ra \mathbb{P}(\mathbb{E}_{310}\oplus \mathbb{C}^3\oplus S^2(\wedge^2\mathbb{C}^3)\oplus  \mathbb{C}^3)$ denote the 
Segre map. Then we have that
\begin{align*}
S([e_1],[e_1],[e_1\wedge e_2])=[e_1^2\otimes (e_1\wedge e_2)] \ .
\end{align*}
Therefore 
\begin{align*}
S(E)= G\cdot [e_1^2\otimes (e_1\wedge e_2)]  \ .
\end{align*}
Since $S$ is an embedding $E$ is a closed orbit with stabilizer
\begin{align*}
\begin{pmatrix}
*&*&*\\
0&*&*\\
0&0&*
\end{pmatrix}
\end{align*}
therefore we identify $E$ with ${F}(1,2,\mathbb{C}^3)$ the space of complete flags in $\mathbb{C}^3$. 
The projection
\begin{align*}
{F}(1,2,\mathbb{C}^3)\xrightarrow{p_1}\mathbb{P}^2 \ 
\end{align*}
exhibits ${F}(1,2,\mathbb{C}^3)$ as a projective bundle with fiber
\begin{align*}
p_1^{-1}([v])=\mathbb{P}(\mathbb{C}^3/\mathbb{C}v)  \ .
\end{align*}
Therefore if $\mathcal{Q}$ denotes the quotient bundle over $\mathbb{P}^2$ then we have the $G$ equivariant identifications 
\begin{align*}
{F}(1,2,\mathbb{C}^3)\cong \mathbb{P}(\mathcal{Q})\cong \mathbb{P}(\mathcal{O}(1)\otimes\mathcal{Q}) = \mathbb{P}(T^{1,0}_{\mathbb{P}^2})
\end{align*}
as expected. $S$ maps the point $[(e\otimes f\otimes (e\wedge f)]$ in $X\setminus E$ to
\begin{align*}
e\cdot f\otimes (e\wedge f)+(e\wedge f)^2 \in \mathbb{E}_{310}\oplus S^2(\wedge^2\mathbb{C}^3) \cong \mathbb{E}_{310}\oplus \mathbb{E}_{220}\ .
\end{align*}
We conclude that the pair
\begin{align*}
(v,w):=\big( (e_1\wedge e_2)^2 \ , \ e_1\cdot e_2\otimes (e_1\wedge e_2)\big)\in \mathbb{E}_{220}\oplus \mathbb{E}_{310} 
\end{align*}
is semistable. $\mathbb{E}_{220}$ plays the role of $\mathbb{V}$ and $\mathbb{E}_{310}$ plays the role of $\mathbb{W}$.
 }
\end{example}
 
\subsubsection{\textbf{Parabolic Induction}}
Let $G$ be a complex semisimple Lie group with Lie algebra $\mathfrak{g}$. Fix $\mathfrak{b}$ a Borel subalgebra and $\mathfrak{t}\subset \mathfrak{b}$ a Cartan subalgebra. Let $\Delta, \Delta^+, S\subset  \Delta^+$ denote the set of roots, positive roots, and simple positive roots respectively. Let $\kappa$ denote the Killing form. $W$ denotes the Weyl group. 

Let $\mathbb{E}$ be a finite dimensional complex rational representation of $G$. Let $v\in \mathfrak{N}(\mathbb{E})$ be a (nonzero) point in the {null cone} \footnote{ Recall that the null cone consists of the unstable points for the action.} of $\mathbb{E}$.  Let $\mathcal{O}_v$ denote the orbit of $v$ in $\mathbb{E}$. Given any compact convex region $R$ (for example a convex lattice polytope) in $M_{\mathbb{R}}$ we define its height $ht(R)$ by
\begin{align*}
ht(R):=\min_{x\in R}||x||_{\kappa} \ .
\end{align*}
$||x||_{\kappa}:=\sqrt{\kappa(x,x)}$ is the Killing distance from $x$ to zero . For any $w\in \mathbb{E}$ we define the height of $w$ to be the height of $\mathcal{N}(w)$, the weight polytope of $w$ with respect to $\mathfrak{t}$.
\begin{definition}(Popov-Vinberg \cite{popov-vinberg})\emph{ $v\in \mathfrak{N}(\mathbb{E})$ is a \emph{\textbf{highest point}}\footnote{Popov and Vinberg call such an element \emph{reduced}. See \cite{popov-vinberg} pg. 202.}in its orbit $\mathcal{O}_v$ provided that the following holds}
\begin{align*}
& 1) \ ht(v)=\max_{\sigma\in G}ht(\sigma\cdot v) \\
\ \\
& 2) \ \dim\mathcal{N}(v)=\min_{\{\sigma\in G\ |\ ht(\sigma\cdot v)=ht(  v)\}}\dim\mathcal{N}(\sigma\cdot v) \ .
\end{align*}
\end{definition}
 We assume that $v$ is a highest point in $\mathcal{O}_v$. Since $v$ is unstable its height is positive. Let 
 \begin{align*}
 \chi=\chi_{\min}(v)
 \end{align*}
  be the (unique) point in $\mathcal{N}(v)$ at minimal distance to $O$. In other words $\chi $ realizes the height of $v$.  Observe that $\chi $ is a {rational} point of $\mathcal{N}(v)$.  Next we consider the dual element
\begin{align*}
 h_v:=\frac{2\kappa^{-1}(\chi)}{||\chi||^2} \ .
 \end{align*}
Through the action of $W$ we can arrange that 
\begin{align*}
h_v\in \mathfrak{t}^+_{\mathbb{Q}} \ ,
\end{align*}
in other words that $\alpha(h_v)\geq 0$ for all $\alpha\in S$. Let $\Gamma$ denote any subset of $M_{\mathbb{R}}$. We set
\begin{align*}
\mathbb{E}|_{\Gamma}:= \bigoplus _{\beta\in \Gamma\cap \ \emph{supp}(\mathbb{E})}\mathbb{E}(\beta) \ .
\end{align*}
We give a similar meaning to $v|_{\Gamma}$. Next we introduce some canonical subalgebras associated to $v$.

\begin{align*}
& \mathfrak{p}(h_v):=\mathfrak{t}\oplus \sum_{\Delta_{\geq 0}:=\{\alpha\in \Delta\ |\ \alpha(h_v)\geq 0\}} \mathfrak{g}_{\alpha}\quad\mbox{( parabolic )} \\
\ \\
& \mathfrak{l}(h_v):=\mathfrak{t}\oplus \sum_{\Delta_{= 0}:=\{\alpha\in \Delta\ |\ \alpha(h_v)= 0\}} \mathfrak{g}_{\alpha} \quad \mbox{(reductive )}\\
\ \\
&\mathfrak{u}(h_v):=  \sum_{\Delta_{> 0}:=\{\alpha\in \Delta\ |\ \alpha(h_v)> 0\}} \mathfrak{g}_{\alpha} \quad \mbox{(nilpotent)} \ .
\end{align*}
Let $P(h_v), L(h_v)$ and $R^uP(h_v)$ denote the corresponding closed algebraic subgroups of $G$. Then $P(h_v)$ has the {Levi decomposition}
\begin{align*}
P(h_v)\cong L(h_v) R^uP(h_v) \ .
\end{align*}
We set $\mathbb{C}^*_v:=\{\exp(\tau h_v)\ |\tau \in \mathbb{C}\}$. By definition we have that 
\begin{align*}
\mathbb{C}^*_v\leq L(h_v) \ .
\end{align*}
We define the reduced Levi subgroup ${\mathbf{L}}(h_v)$ to be the quotient $L(h_v)\backslash \mathbb{C}^*_v$ .   Since $h_v$ is a semisimple element  we may define the eigenspace decompositions
\begin{align*}
&\mathbb{E}(q):=\{w\in \mathbb{E}\ |\ h_v\cdot w=qw\} \\
\ \\
&\mathbb{E}(\geq 0):= \bigoplus _{q\in\mathbb{Q}_{+}}\mathbb{E}(q) \ .
\end{align*}
Let $\Gamma_v$ denote the affine hyperplane in $M_{\mathbb{R}}\cong \mathfrak{t}^{\vee}_{\mathbb{R}}$ with equation $\{ \beta(h_v)=2\}$. Observe that $\chi$ lies in $\Gamma_v$ and $\mathcal{N}(v)$ lies on one side of $\Gamma_v$.  Observe that ${\mathbf{L}}(v)$ preserves $\mathbb{E}|_{\Gamma_v}$ . 

Now we make the assumptions that 
\begin{align*}
  \mathcal{N}(v)\subset \Gamma_v \ \mbox{and} \ R^uP(v) \ \mbox{acts trivially on}\ \mathbb{E}|_{\Gamma_v} \ .
\end{align*}
We define
\begin{align*}
& \mathcal{O}_{\Gamma_v}:= {\mathbf{L}}(v)\cdot [v]\subset \mathbb{P}(\mathbb{E}|_{\Gamma_v}) \ .
\end{align*}

The following construction of a \emph{\textbf{parabolically induced orbit}}  plays a basic role in the study of nilpotent orbits in complex semisimple Lie algebras.
\begin{align*}
& { {Ind}^{\ G}_{\ P}}(\overline{\mathcal{O}}_{\Gamma_v}):=G\times_P \overline{\mathcal{O}}_{\Gamma_v} \\
\ \\
& (g,[w])\sim (g\cdot p^{-1}, p\cdot [w]) \ .
\end{align*}
\begin{theorem} (Popov-Vinberg \cite{popov-vinberg} pg. 203 corollary 2.)\emph{ The canonical map
\begin{align*}
 { {Ind}^{\ G}_{\ P}}(\overline{\mathcal{O}}_{\Gamma_v})\ra \overline{G\cdot [v]} 
 \end{align*}
 is birational.}
 \end{theorem}
 \begin{corollary}
 \emph{The number of orbits in $\overline{G\cdot [v]}$ does not exceed the number of orbits in 
 $  {Ind}^{\ G}_{\ P}(\overline{\mathcal{O}}_{\Gamma_v})$ .}
 \end{corollary}
  
Suppose that $\mathbb{E}$ is reducible and coincides with the linear span of $\mathcal{O}_v$.  
Then we may conclude that $\overline{G\cdot [v]}$ consists of two orbits whenever $ { {Ind}^{\ G}_{\ P}}(\overline{\mathcal{O}}_{\Gamma_v})$ consists of two orbits.

In the next example we require a criterion insuring that a given rational semisimple element is a characteristic for a $v \in \mathfrak{N}(\mathbb{E})$.
Popov and Vinberg attribute the following result to L. Ness and F. Kirwan (independently).
\begin{theorem}(see \cite{popov-vinberg} Theorem 5.4 pg. 202) \emph{Let $v\in \mathfrak{N}(\mathbb{E})$. Let $h\in \mathfrak{t}_{\mathbb{Q}}$ be such that
\begin{align*}
v\in \oplus_{\beta\in \Delta(h)_{\geq 0}}\mathbb{E}(\beta) \ .
\end{align*}
Then $v$ is highest in its orbit with characteristic  $h$ if and only if $v|_{\Gamma_{h}}\in \mathbb{E}|_{\Gamma_{h}}$ is semistable for the action of
$\mathbf{L}(h)$.    }
\end{theorem}

\begin{example} \emph{Let $G=SL(3,\mathbb{C})$ and $H= \exp(\mathfrak{t}\oplus \mathbb{C}E_{23}\oplus \mathbb{C}E_{13} ) $ where
  $\mathfrak{t}$ denotes the Cartan subalgebra of $G$ and $E_{ij}$ denote the standard root vectors with respect to $ad(\mathfrak{t})$ . 
 Let $\mathfrak{g}$ denote the Lie algebra of $G$. We define 
 \begin{align*}
 X_{nil}:= \{([A], [\xi])\in \mathbb{P}(\mathfrak{g})\times \mathbb{P}^2 \ |\ A \ \mbox{is nilpotent} \ , \ A\cdot \xi=0\} \ .
 \end{align*}
 There is an obvious exact sequence of $G$ modules 
 \begin{align*}
 0\ra \ker(\Phi)\ra \mathfrak{g}\otimes \mathbb{C}^3\overset{\Phi}{\ra}\mathbb{C}^3\ra 0\qquad \Phi(A\otimes \xi) :=A\cdot \xi\  .
 \end{align*}
 In which case we have the $G$ decomposition 
 \begin{align*}
 \mathfrak{g}\otimes \mathbb{C}^3\cong \ker(\Phi)\oplus \mathbb{C}^3 \ .
 \end{align*}
 Observe that $E_{13}\otimes e_1$ lies in  $\ker(\Phi)$ and has dominant weight $3L_1+L_2=(310)=\mubull$ with respect to $\mathfrak{t}$. Moreover $\xi_{(310)}:= E_{13}\otimes e_1$ is killed by all of the positive root vectors $E_{12},E_{13}, E_{23}$ hence generates an irreducible subrepresentation  
 $\mathbb{E}_{(310)}\subset \ker(\Phi)$ of dimension 15 (by the Weyl dimension formula applied to the weight $\mubull$ ) . Similarly it is easy to check that  $\xi_{(220)}:= E_{23}\otimes e_1-E_{13}\otimes e_2 \in \ker(\Phi)$ is also a highest weight vector with weight $2(L_1+L_2)=(220)=\lambull$ . This generates an irreducible subrepresentation $\mathbb{E}_{(220)}\cong Sym^2(\wedge^2\mathbb{C}^3)$ of dimension 6. Therefore we have the decompositon $\ker(\Phi)\cong \mathbb{E}_{(220)}\oplus \mathbb{E}_{(310)}$ and hence the embedding
 \begin{align*}
 X_{nil}\hookrightarrow \mathbb{P}(\mathbb{E}_{(220)}\oplus \mathbb{E}_{(310)}) \ .
 \end{align*}
Choose the basepoint $o:=[E_{23}\otimes e_1]$. Then we have an embedding of the homogeneous space $G/ H$
\begin{align*}
G/H\cong G\cdot o  \subset X_{nil} \ .
\end{align*}
 In fact we can say more, namely that $\overline{G\cdot o }=X_{nil}$. We will show that $X_{nil}$ consists of two orbits for the $G$ action.  A closed orbit is easily located. Observe that 
 \begin{align*}
 E_{23}\otimes e_1=\frac{1}{2}(\xi_{(220)}+E_{21}\cdot \xi_{(310)}) \ .
 \end{align*}
Let $\sigma:=\exp(E_{12})$. Then it is easy to see that
\begin{align*}
o:= \sigma\cdot (E_{23}\otimes e_1)=\frac{1}{2}(E_{23}\otimes e_1+ 2\xi_{(310)}) \ .
\end{align*}
 Consider the one parameter subgroup 
 \begin{align*}
 \lambda(t):= \begin{pmatrix} t^{-1}&0&0\\ 0&t&0\\ 0&0&1 \end{pmatrix}\ .
 \end{align*}
 We have
 \begin{align*}
  \lambda(t)\cdot o= \frac{1}{2}(E_{23}\otimes e_1+ 2t^{-2}\xi_{(310)}) 
 \end{align*}
therefore $\lambda(t)\cdot o\ra [\xi_{(310)}]$  as $t\ra 0$. Consequently \emph{if } $\overline{G\cdot o }$ consists of two orbits then we must have
\begin{align*}
\overline{G\cdot o }=G\cdot o \cup G\cdot [\xi_{(310)}] \ .
\end{align*}
 In the orbit decomposition of $\mathbb{E}$ we have $\mathbb{V}=\mathbb{E}_{(220)}$, and $\sigma^{-1}o$ decomposes as follows
 \begin{align*}
 u:=\sigma^{-1}o=E_{23}\otimes e_1 = (\xi_{(220)} \ , \ E_{21}\cdot \xi_{(310)}) \ .
 \end{align*}
 Therefore the pair $(\xi_{(220)} \ , \ E_{21}\cdot \xi_{(310)})$  is semistable with respect to $G$.  In order to show that $\overline{G\cdot o }$ consists of exactly two $G$ orbits we observe that the orbit
 \begin{align*}
 G\cdot [ E_{23}\otimes e_1]  
 \end{align*}
 is  Hilbert-Mumford {unstable} and we realize this orbit as a parabolically induced one.
  Obviously $\mathcal{N}( E_{23}\otimes e_1 )=\{(220)\}$ . 
 Therefore we can easily see that
 \begin{align*}
 &\chi_{min}=(220) \ , \ h=\frac{1}{2}\begin{pmatrix}1&0&0\\0&1&0\\0&0&-2\end{pmatrix} \\
 \ \\
 &\Gamma_u=\{aL_1+bL_2\ |\left(\frac{a+b}{2}\right)=2\} \\
 \ \\
 &\Gamma_u\cap \mbox{supp}(\mathbb{E})=\{ 3L_1+L_2,2L_1+2L_2, L_1+3L_2\} \ . \\
 \ \\
 &\mathbb{E}|_{\Gamma_u}=\overbrace{\mathbb{C}\xi_{(310)}}^{(310)}\oplus \overbrace{\mathbb{C}E_{21}\cdot \xi_{(310)}\oplus \mathbb{C}\xi_{(220)}}^{(220)}\oplus \overbrace{\mathbb{C}E^2_{21}\cdot \xi_{(310)}}^{(130)} \\
 \ \\
 &\Delta_{\geq 0}=\Delta^+\cup \{L_2-L_1\} \ ,\ \Delta_{=0}=\{\pm(L_1-L_2)\}\ , \ \Delta_{>0}=\{L_1-L_3, L_2-L_3\} \ .
 \end{align*}
The canonical subalgebras have the shape (see figure 2):
\begin{align*}
\mathfrak{p}(h)=\begin{pmatrix}*&*&*\\ *&*&*\\0&0&*\end{pmatrix}\ , \ \mathfrak{l}(h)=\begin{pmatrix}*&*&0\\ *&*&0\\0&0&*\end{pmatrix}\
, \ \mathfrak{u}(h)=\begin{pmatrix}0&0&*\\ 0&0&*\\0&0&0\end{pmatrix}\ \ .
\end{align*}
 $L(u)$ is easily seen to be isomorphic to $GL(2,\mathbb{C})$ and that $\mathfrak{u}(h)$ acts trivially on $\mathbb{E}|_{\Gamma}$.
\begin{align*}
L(u)= S(GL(2,\mathbb{C})\times \mathbb{C}^*):=\{(g,\alpha )\in GL(2,\mathbb{C})\times \mathbb{C}^*\ | \det(g)\alpha=1\}\cong GL(2,\mathbb{C}) \ .
 \end{align*}
   So $\mathbb{E}|_{\Gamma}$ is a four dimensional representation of  $L(u)$. Next we observe that  $\xi_{(220)}$ is $L(u)$ semiinvariant  
 \begin{align*}
 \sigma \cdot  \xi_{(220)}= \det(g)^2 \xi_{(220)} \ \mbox{for }\ \sigma= (g,\alpha )\in L(u) \ .
  \end{align*} 
  From this we deduce that 
\begin{align*}
\mathbb{E}|_{\Gamma}\cong Sym^2(\wedge ^2\mathbb{C}^2) \oplus Sym^2(\mathbb{C}^2)\otimes  \wedge ^2\mathbb{C}^2  \ \mbox{as $GL(2,\mathbb{C})$ modules .}
\end{align*}
Moreover the basepoint $u$ is identified with $(e_1\wedge e_2,e_1\cdot e_2 )$ . Finally we have that
\begin{align*}
L(u)\cdot [u]=SL(2,\mathbb{C})\cdot [(e_1\wedge e_2,e_1\cdot e_2 )] \ .
\end{align*}
This is the case studied in example \ref{2x2} and we may induce.  Observe that the present example coincides with example \ref{blowup}
\begin{align*}
X_{nil}=\overline{SL(3,\mathbb{C})\cdot[(e,f)]}=B_{\Delta}(\mathbb{P}^2\times\mathbb{P}^2)\hookrightarrow \mathbb{P}(\mathbb{E}_{220}\oplus \mathbb{E}_{310}) \ .
\end{align*}}
\end{example}  
 \begin{center}
 \begin{tikzpicture}[>=latex]
\pgfmathsetmacro\ax{2}
\pgfmathsetmacro\ay{0}
\pgfmathsetmacro\bx{2 * cos(120)}
\pgfmathsetmacro\by{2 * sin(120)}
\pgfmathsetmacro\cx{2 * cos(240)}
\pgfmathsetmacro\cy{2 * sin(240)}
\pgfmathsetmacro\dx{2*cos(60)}
\pgfmathsetmacro\dy{2*sin(60)}
\pgfmathsetmacro\lax{2*\ax/3 + \bx/3}
\pgfmathsetmacro\lay{2*\ay/3 + \by/3}
\pgfmathsetmacro\lbx{\ax/3 + 2*\bx/3}
\pgfmathsetmacro\lby{\ay/3 + 2*\by/3}
 \fill[blue!15] (0,0) -- (0:6 * \lby) -- ( 4,6* \lby) -- cycle;
 \foreach \k in {1,...,6} { \draw[line width=.4mm, magenta!70,-] (0,0) -- +(\k *60 + 0:6) ;}
 \draw[gray,-](3*\ax,3*\ay)--(-3*\cx,-3*\cy);
\draw[gray,-](2*\ax,2*\ay)--(2*\ax+3*\bx,2*\ay+3*\by);
\draw[gray,-](\ax, \ay)--(\ax+3*\bx, \ay+3*\by);
\draw[gray,-](\ax, \ay)--(\ax-2*\bx, \ay-2*\by);
\draw[gray,-](2*\ax, 2*\ay)--(2*\ax-\bx, 2*\ay- \by);
\draw[gray,-](-\ax+2*\bx, -\ay+2*\by)--(-\ax-3*\bx, -\ay- 3*\by);
\draw[gray,-](-2*\ax+\bx, -2*\ay+\by)--(\ax+3*\cx,\ay+3*\cy);
\draw[gray,-](-3*\ax, -3*\ay)--(3*\cx,3*\cy);
\draw[gray,-](3*\bx,3*\by)--(-3*\cx,-3*\cy);
\draw[gray,-](2*\bx-\ax,2*\by-\ay)--(3*\ax+2*\bx,3*\ay+2*\by);
\draw[gray,-](-2*\ax+\bx,-2*\ay+\by)--(3*\ax+\bx,3*\ay+\by);
\draw[gray,-](-3*\ax-\bx,-3*\ay-\by)--(2*\ax-\bx,2*\ay-\by);
\draw[gray,-](3*\cx,3*\cy)--(-3*\bx,-3*\by);
\draw[gray,-](-\ax+2*\cx,-\ay+2*\cy)--(\ax-2*\bx, \ay-2*\by);
\draw[gray,-](3*\ax,3*\ay)--(-3*\bx,-3*\by);
\draw[gray,-](-3*\ax,-3*\ay)--(3*\bx,3*\by);
\draw[gray,-]( \ax+ 3*\bx,\ay+3*\by)--(-3*\ax-\bx, -3*\ay-\by);
\draw[gray,-](-3*\ax+ -2*\bx, -3*\ay+-2*\by)--(2*\ax+ 3*\bx, 2*\ay+3*\by);
\draw[gray,-](3*\ax+ 2*\bx, 3*\ay+2*\by)--(-2*\ax-3*\bx, -2*\ay-3*\by);
\draw[gray,-](3*\ax+ \bx, 3*\ay+\by)--(-\ax+ -3*\bx, -\ay+-3*\by);
\draw[line width=.5mm,blue,-](2*\ax+2*\bx,2*\ay+2*\by)--(2*\ax+2*\cx,2*\ay+2*\cy);
 \draw[line width=.5mm,blue,-](2*\ax+2*\cx,2*\ay+2*\cy)--(2*\bx+2*\cx,2*\by+2*\cy);
\draw[line width=.5mm,blue,-](2*\bx+2*\cx,2*\by+2*\cy)--(2*\ax+2*\bx,2*\ay+2*\by);
\draw[line width=.5mm,red,->] (0,0) -- (\bx-\cx,\by-\cy) node[right] {\(L_2-L_3\)};
\draw[line width=.5mm,red,->] (0,0) -- (\ax-\cx,\ay-\cy) node[right] {\(L_1-L_3\)};
\draw[line width=.5mm,red,->] (0,0) -- (\ax-\bx,\ay-\by) node[right] {\(L_1-L_2\)};
\draw[line width=.5mm,orange,-](3*\ax+\bx,\by)--(3*\ax+\cx,\cy);
\draw[line width=.7mm,red, dashed](3*\ax+\bx,\by)--(\ax+3*\bx,\ay+3*\by);
\draw[line width=.5mm,orange,-](\ax+3*\bx,\ay+3*\by)--( 3*\bx+\cx, 3*\by+\cy);
\draw[line width=.5mm,orange,-]( 3*\bx+\cx, 3*\by+\cy)--(\bx+3*\cx,\by+3*\cy);
\draw[line width=.5mm,orange,-] (\bx+3*\cx,\by+3*\cy) -- (\ax+3*\cx,\ay+3*\cy) ;
\draw[line width=.5mm,orange,-]  (\ax+3*\cx,\ay+3*\cy)--(3*\ax+\cx, 3*\ay+\cy) ;
 \draw[line width=.5mm,black,->](0,0)--(2*\ax+2*\bx,2*\ay+2*\by);
\node at (40:4.3) {\(\Gamma\)};
\node at (54.5:4.4) {\(\quad \lambull=220 \)};
\node at (17:6) {\(\mubull= 310 \)};
\end{tikzpicture}
\end{center}
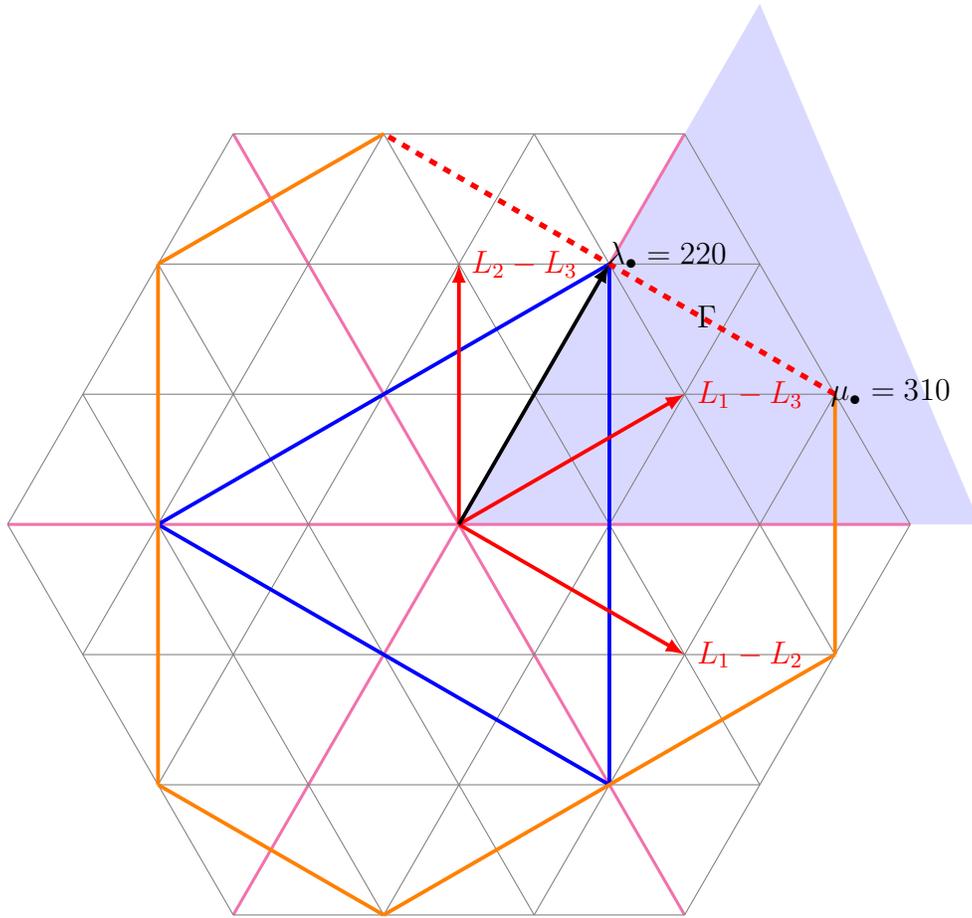
\begin{center}\small{Figure 2. Parabolic induction scheme and polyhedra associated to the modules $\mathbb{E}_{310}$ and $\mathbb{E}_{220}$ .}\end{center}

 \begin{example}[\emph{Classical discriminant and Resultants}]\emph{Consider two polynomials $P$ and $Q$ in one variable of degrees $m$ and $n$ respectively
\begin{align*}
&P(z)=a_mz^m+a_{m-1}z^{m-1}+\dots + a_1z+a_0 \\
\ \\
&Q(z)=b_nz^n+b_{n-1}z^{n-1}+\dots + b_1z+b_0 \ .
\end{align*}
Recall that the classical \emph{resultant} of $P$ and $Q$ is the (quasi)homogeneous polynomial  of the coefficients $(a_0,\dots,a_m;b_0,\dots,b_n)$ defined by
\begin{align*}
R_{m,n}(P,Q)= R_{m,n}(a_0,\dots,a_m;b_0,\dots,b_n):=b_n^m\prod_{\beta_i\in \mbox{zer}(Q)}P(\beta_i)=(-1)^{mn}R_{n,m}(Q,P) \ .
\end{align*}
When $m=n=d\geq 2$ we denote the resultant by $R_d$. Then
\begin{align*}
R_d\in \mathbb{C}_{2d}[M_{2\times (d+1)}] \ .
\end{align*}
$G=SL(d+1,\mathbb{C})$ acts on $R_d$ by the rule
\begin{align*}
\sigma\cdot R_d (A):=R_d(A\cdot \sigma) \quad \sigma\in G \ , \ A\in M_{2\times (d+1)} \ .
\end{align*}
The \emph{discriminant} , $\Delta_d$ , of a polynomial $P$ of degree $d$ is defined by
\begin{align*}
&\Delta_d(a_0,\dots,a_d):=R_{d,d-1}(P,\frac{\dl P}{\dl z}) \\
\ \\
&\Delta_d\in\mathbb{C}_{2d-2}[M_{1\times(d+1)}] \ .
\end{align*}
The action of $G$ is given by
\begin{align*}
\sigma \cdot \Delta_d(a)=\Delta_d(a\cdot \sigma) \ .
\end{align*}
It follows from work of Gelfand, Kapranov, and  Zelevinsky (\cite{newtpoly}) that the pair 
\begin{align*}
(R_d^{\deg(\Delta_d)}, \Delta_d^{\deg(R_d)})
\end{align*} 
is numerically semistable with respect to the {standard} torus, i.e. the torus corresponding to the $d$ fold Veronese embedding of $\mathbb{P}^1$ .
\begin{definition}(\cite{newtpoly})\emph{
Let $d\geq 2 \in \mathbb{N}$. Let $[0,d]:=\{0,1,2.\dots ,d\} .$ Given any subset $S\subseteq [1,d-1]$ the \emph{\textbf{vertex map}} $$V=V(S): [0,d]\ra \mathbb{N}$$ is given as follows.\newline
\noindent \emph{Case} 1. $S=\emptyset$. Set 
 $V(0)=V(d):= d \ ,\ V(i):=0 $ for all $i\in [1,d-1]$. \newline
 \ \\
 \noindent \emph{Case} 2. $S\neq\emptyset$. Say $S=\{1\leq i_1<\dots <i_k\leq d-1\} $. Set $i_0:=0$ and $i_{k+1}:= d$. In this case $V$ is determined by the following rule.
\begin{align*}
&V(i_0):=i_1 \ ,\ V(i_{k+1}):= d-i_k \\
\ \\
&V(i):=0 \ \mbox{for all $i\notin S\cup \{0,d\}$ } \ ,\ V(i_j):=i_{j+1}-i_{j-1} \ .
\end{align*}}
\end{definition}
\begin{theorem}(\cite{newtpoly}) \emph{The vertices of the Chow polytope are given by
\begin{align*}
\{ (V(0),V(1),\dots, V(d)) \ |\ S\subset [1,d-1]\}\ .
\end{align*}
The vertices of the discriminant polytope are given by
\begin{align*}
\{v-(1,0,\dots,0,1)\ | \ \mbox{$v$ is a vertex of the Chow polytope}\}\ .
\end{align*}}
\end{theorem}
\begin{corollary}
\begin{align*}
\deg(\Delta_d)\mathcal{N}(R_d)\subseteq \deg(R_d)\mathcal{N}(\Delta_d) \ .
\end{align*}
\end{corollary}
The proof follows at once from the identity
\begin{align*}
 (2d-2)(v_0,\dots,v_d)=\left(\frac{d-1}{d}\right)2d(v_0-1,\dots,v_d-1)+\left(\frac{1}{d}\right)2d(d-1,0,\dots,d-1)\ .\quad   
 \end{align*}
}\end{example}

 
 \subsection{Parametrized Theory}
 Now we study the dependence of semistability upon algebraic parameters.  To begin let $Y$ be a complex algebraic $G$-variety.  Let $\mathbb{V}$ and $\mathbb{W}$ be finite dimensional complex rational $G$-modules. Given two regular $G$ maps
 \begin{align}\label{parameters}
 \begin{split}
& v:Y\ra\mathbb{P}(\mathbb{V}) \\
 \ \\
 & w: Y\ra \mathbb{P}(\mathbb{W})
 \end{split}
 \end{align}
 we define the $(v,w)$-semistable and numerically semistable loci in $Y$ in the obvious way
 \begin{align*}
 &Y^{ss}(v\ ,\ w):=\{y\in Y \ | \ \overline{\mathcal{O}}_{v(y)w(y)}\cap \overline{\mathcal{O}}_{v(y)}=\emptyset \}\\
 \ \\
 &Y^{nss}(v\ ,\ w):=\{y\in Y \ | \ \mathcal{N}(v(y)) \subset\mathcal{N}(w(y)) \ \mbox{for all $H\leq G$}\ \}\ .
 \end{align*}
 \begin{remark} \emph{ The disjointedness of the orbit closures is independent of the lifts of $v(y)$ and $w(y)$ to $\mathbb{V} , \mathbb{W}$ respectively. }   
\end{remark}
 
 Recall that an invertible sheaf $\mathscr{L}$ on a variety $X$ is \emph{globally generated} or \emph{base point free} if and only if there exists a finite dimensional subspace 
 \begin{align*}
 \mathbb{V}^{\vee} \subset \Gamma (X\ , \ \mathscr{L})
 \end{align*}
 such that for every $p\in X$ there is an $s\in  \mathbb{V}^{\vee}$ such that $s(p)\neq 0$ . We recall that in this situation there is an associated map
 \begin{align*}
 \phi_{\mathscr{L}, \mathbb{V}}:X\ra \mathbb{P}(\mathbb{V})\ 
 \end{align*}
satisfying
\begin{align*}
\phi_{\mathscr{L}, \mathbb{V}}^*\mathcal{O}(1) \cong \mathscr{L} \ .
\end{align*}

 Given $\mathscr{L}\in Pic^G(Y)$ observe that $G$ acts naturally on $\Gamma(Y \ , \ \mathscr{L})$.  When $\mathscr{L}$ is globally generated the finite dimensional subspace $\mathbb{V}^{\vee}$ of  $\Gamma(Y \ , \ \mathscr{L})$ may be chosen to be $G$ {invariant}. In this case the associated map $\phi_{\mathscr{L}, \mathbb{V}}$ is $G$-equivariant.
  
  Next we consider {pairs} $\mathscr{L} , \mathscr{M}$ of globally generated $G$-linearized invertible sheaves on $Y$ together with their associated $G$ maps 
\begin{align*}
  \phi_{\mathscr{L}, \mathbb{V}}\ , \  \phi_{\mathscr{M}, \mathbb{W}}  :Y\ra \mathbb{P}(\mathbb{V}) \ , \ \mathbb{P}(\mathbb{W}) \ .
\end{align*}
 
Finally we can define the $( \mathscr{L}  , \   \mathscr{M})$  {(numerically) semistable loci} in $Y$.
 \begin{definition}\label{defnsemistable} \emph{Let $Y$ be a $G$ variety. Let  $\mathscr{L} , \mathscr{M}\in \mbox{Pic}(Y)^G$ be globally generated. Then the (numerically)semistable loci with respect to the pair $( \mathscr{L} , \mathscr{M})$ are given by
 \begin{align*}
& Y^{ss}( \mathscr{L}  , \   \mathscr{M}):= Y^{ss}(  \phi_{\mathscr{L}, \mathbb{V}}\ , \  \phi_{\mathscr{M}, \mathbb{W}}   ) \\
\ \\
&Y^{nss}( \mathscr{L}  , \   \mathscr{M}):=Y^{nss}(\phi_{\mathscr{L}, \mathbb{V}}\ , \  \phi_{\mathscr{M}, \mathbb{W}} ) \ .
\end{align*} }
\end{definition}
For the convenience of the reader we recall the definition of a \emph{constructible} algebraic set. 
\begin{definition}
\emph{Let $W$ be an algebraic variety. A subset $S\subset W$ is \textbf{\emph{constructible}} if and only if  $S$ is a finite union of locally closed subsets of $W$. That is, there exits a finite collection $W_j,Z_j \ (1\leq j\leq m)$ of subvarieties of $W$ such that}
\begin{align*}
S=\bigcup_{1\leq j\leq m}W_j\setminus Z_j \ .
\end{align*}
\end{definition}
 
The main result in this section is the following.
\begin{proposition}\label{constructible}
\emph{ $Y^{nss}(v,w )$ is a constructible subset of $Y$.}
\end{proposition}
\begin{proof}
For simplicity we assume that $\mathbb{V}=\elam$ and $\mathbb{W}=\emu$ are irreducible $G$ modules. Fix $T\leq G$ a maximal algebraic torus. Given $\beta\in \mbox{supp}(\mu_{\bull})$ let $\emu(\beta)$ denote the corresponding weight space. We will consider this as a trivial $G$-bundle over $Y$ (with trivial action)
\begin{align*}
\emu(\beta)\times Y\xrightarrow{\pi_2} Y \ .
\end{align*}

Let $\tau:G\times Y\ra Y$ denote the action. We consider the $G$ bundle
\begin{align*}
\mathscr{L}_{w}\otimes \emu(\beta) \ .
\end{align*}

For any vector space $\mathbb{V}$ we denote the value at $[v]$ of the tautological section $\iota$ of
\begin{align*}
\mathcal{O}_{\mathbb{P}(\mathbb{V})}(1)\otimes \mathbb{V}
\end{align*}
by $\widehat{v}$.

Let $\{\xi_{ \beta}^{(j)}\}$ ($1\leq j\leq m_{\mubull}(\beta)$ ) denote any unitary weight basis\footnote{One may use the famous \emph{Gelfand-Tsetlin basis}.} of $\emu(\beta)$  .
Then we produce algebraic sections:
\begin{align*}
& Q_{\mubull ; \beta}\in \Gamma(Y\ ,\ \mathscr{L}_{w}\otimes \emu(\beta)) \\
\ \\
&Q_{\mubull ; \beta}(y):=\sum_{1\leq j\leq m_{\mubull}(\beta)}<\widehat{w}(y),\xi_{\beta}^{(j)}>\otimes \xi_{ \beta}^{(j)} \ .
\end{align*}
For any collection of weights $B:=\{\beta_1,\dots,\beta_r\}\subset \mbox{supp}(\mubull)$ we define
\begin{align*}
&Q_{\mubull;B}:=\bigotimes _{1\leq i\leq r}Q_{\mubull ; \beta_i}\in \Gamma(Y, \mathscr{L}_{w}^r\bigotimes _{1\leq i\leq r}\emu(\beta_i)) \ .
\end{align*}
Next observe that there is a finite collection $\{P_1,\dots, P_{k(\mubull)}\}$ of lattice subpolytopes of $\mathcal{N}(\mubull)$ such that for any $v\in\emu\setminus\{0\}$ we have
\begin{align*}
\mathcal{N}(v)\in \{P_1,\dots, P_{k(\mubull)}\}\ .
\end{align*}
 Let $B(i)\subset \mbox{supp}(\mubull)$ denote the \emph{vertex set} of $P_i$ with corresponding section $Q_{\mubull;B(i)}$. 
We define
\begin{align*}
C(i):= \mbox{supp}(\mubull)\cap(\mathcal{N}(\mubull)\setminus P_i) \ ,
\end{align*}
with corresponding section $Q_{\mubull;C(i)}$.
\begin{lemma}
\emph{Let $\sigma\in G$. Then 
\begin{align*}
\sigma\cdot y\in (Q_{\mubull;B(i)}\neq 0)\cap (Q_{\mubull;C(i)}=0) \ \mbox{if and only if}\ \mathcal{N}(\sigma\cdot w(y) )=P_i\ .
\end{align*}}
\end{lemma}  

 Next we consider the weight $\lambull$ and define
\begin{align*}
D(i):=\mbox{supp}(\lambull)\cap (\mathcal{N}(\lambull)\setminus (\mathcal{N}(\lambull)\cap P_i))\ .
\end{align*}
We have the corresponding section
\begin{align*}
Q_{\lambull; D(i)}(y):= \bigotimes _{\alpha_j\in D(i)}Q_{\lambull ; \alpha_j}\in \Gamma(Y, \ \mathscr{L}_v^{\#D(i)}\bigotimes _{\alpha_j\in D(i)}\elam(\alpha_j))\ .
\end{align*}
We now pull back all of the bundles (and the sections $Q$) we have introduced to $G\times Y$ via the action $\tau$.  We define
constructible subsets  $S_i\subset G\times Y$ by
\begin{align*}
S_i:=(\tau^*Q_{\mubull;B(i)}\neq 0)\cap (\tau^*Q_{\mubull;C(i)}=0)\cap(\tau^*Q_{\lambull; D(i)} \neq 0 ) \  .
\end{align*}
Define $S:=\cup _{1\leq i\leq k(\mubull)} S_i $. Let $\pi_2$ denote the projection $\pi_2:G\times Y\ra Y \ $. Observe that $\pi_2(S)$ coincides with the numerically \emph{unstable} locus.
 Since the projection of a constructible set is constructible (as is the compliment of any such set) we are done.
\end{proof}
 \subsection{Hilbert-Mumford Semistability and the Semistability of Pairs}\label{table}
We close this section with a direct comparison of the Hilbert-Mumford Semistability and the Semistability of Pairs. We hope that the table below makes the relationship bewteen the two theories completely transparent.     
\ \\

  \begin{center} \begin{tabular}{l|l}
  \textbf{Hilbert-Mumford Semistability}& \textbf{Dominance}/\textbf{Semistability of Pairs} \\ \\
\hline \\
 $0\notin\overline{G\cdot w}$ & \ $\overline{\mathcal{O}}_{vw}\cap\overline{\mathcal{O}}_{v}=\emptyset$  \\ \\
 \hline \\
 $w_{\lambda}(w)\leq 0$   &\ $w_{\lambda}(w)-w_{\lambda}(v)\leq 0$   \\
 for all 1psg's $\lambda$ of $G$ & \ for all 1psg's $\lambda$ of $G$\\ \\
 \hline \\
 $0\in \mathcal{N}(w)$ all $H\leq G$ &\ $\mathcal{N}(v)\subset \mathcal{N}(w)$ {all $H\leq G$} \\ \\
 \hline \\
 $\exists$ $C\geq 0$ such that &\ $\exists$ $C\geq 0$ such that \\
 $\log||\sigma\cdot w||^2\geq -C$   &\ $\log {||\sigma\cdot w||^2}-\log{||\sigma\cdot v||^2}  \geq -C $\\  
  all $\sigma\in G$ &\ all $\sigma\in G$ \\ \\
  \hline\\
 $\mu^{-1}(0)\cap \overline{G\cdot [w]}\neq \emptyset$ &  $ \mathscr{P}(\overline{\mathcal{O}}_{vw})=\mathscr{P}(\overline{\mathcal{O}}_{w})$\\ \\
 \hline \\
 $Y^{ss}(\mathscr{L}) $ is (Zariski) open &\ $Y^{nss}( \mathscr{L}  \ ,\  \mathscr{M} )$ is constructible
  \end{tabular} 
  \end{center}
\ \\
\begin{center}\small{Table 1.}\end{center}
  
 Observe that in Table 1 the left hand column arises from the right hand column by taking $\mathbb{V}$ to be the trivial one dimensional representation of $G$ and $v$ any (nonzero) constant.  In particular Hilbert-Mumford semistability is a \emph{special case} of the theory of semistable pairs.
  \section{The CM-Polarization Re-examined}
In this section we consider a family parametrized by a smooth quasi-projective base.
\begin{align*}
\xymatrix{\mathbb{X}\ar[rd]_{\pi:={\pi_1}|_{\mathbb{X}}}\ar[r]&Y\times\cpn\ar[d]^{\pi_1}\ar[r]^{\pi_2}&\cpn \\
&Y& }
\end{align*}
We assume that the family satisfies the following hypothesis: 
\begin{enumerate}
\item The map $\mathbb{X}\xrightarrow{\pi} Y$ is flat .  \\
\ \\
\item The fibers ( $X_y ,\  {L}|_{X_y} )$ are polarized manifolds , $ {L}:= \pi_2^*\mathcal{O}_{\cpn}(1)|_{\mathbb{X}}$ . \\
\end{enumerate}
Let $G$ be a linear reductive (for simplicity) algebraic group. We say that the family $\mathbb{X} \xrightarrow{\pi} Y$ is equivariant provided that:
\ \\
\begin{enumerate}
\item There is a rational representation $\rho:G\ra \slnc$ . \\
\ \\
\item $\mathbb{X}$ and $Y$ are $G$-varieties, where $G$ acts naturally on the fibers of $\pi$ and for every $\sigma \in G$  and $y\in Y$ we have $\rho(\sigma)\cdot X_y=X_{\sigma\cdot y}$ . \\
 \end{enumerate}
  \begin{remark} \emph{Usually we have that $G=\slnc$ and  $\rho$ is the identity, however there are important examples where $G$ is smaller than the ambient group.}
  \end{remark}
\subsection{The CM Polarization} Now we begin the construction of the pair $(\mathscr{L}_{R},\mathscr{L}_{\Delta})$ .
Given $q\in \mathbb{Z}$ let $\mathcal{O}_{\mathbb{X}}(q):=\pi_2^*\mathcal{O}_{\cpn}(q)$ .  We define
\begin{align*}
\mathcal{E}_{n+1}:=\overbrace{\bigoplus \mathcal{O}_{\mathbb{X}}(1)}^{n+1} \ .
\end{align*}
Choosing  $m\in \mathbb{Z}_+$  sufficiently large we define an invertible sheaf $\mathscr{L}_{R}(m)$ on $Y$ by
\begin{align*}
\mathscr{L}_{R}(m):=\bigotimes_{0\leq j\leq n+1}\det \pi_{*}\left (\bigwedge^j \mathcal{E}_{n+1}^{\vee}\otimes \mathcal{O}_{\mathbb{X}}(m)\right)^{(-1)^j(n(n+1)d-d\mu)} \ .
\end{align*}

Let $\rho_{\mathbb{X}}$ denote the fiberwise \emph{Gauss map}:
\begin{align*}
\rho_{\mathbb{X}}:\mathbb{X}\ra \mathbb{G}(n,N) \ , \ \rho_{\mathbb{X}}(z):=\mathbb{T}_z(X_{\pi(z)}) \ .
\end{align*}
Let $\mathscr{U}$ denote the universal rank $n+1$ bundle over $\mathbb{G}(n,N)$ . Then the relative bundle of {one jets} of $\mathcal{O}_{\mathbb{X}}(1)$ is 
\begin{align*}
J_{1}(\mathbb{X}/Y):= \rho_{\mathbb{X}}^*(\mathscr{U}^{\vee}) \ .
\end{align*}

Let $\mathbb{X}_{(n-1)}$ denote the new family
\begin{align}\label{newfamily}
\mathbb{X}_{(n-1)}:=\mathbb{X}\times\mathbb{P}^{n-1}\subset Y\times \mathbb{P}(M_{n\times (N+1)}) \ .
\end{align}
Then the map $\pi={\pi_{1}}|_{\mathbb{X}_{(n-1)}}:\mathbb{X}_{(n-1)}\ra Y$ is a flat family of polarized manifolds: 
\begin{align*}
\pi^{-1}(y)= X_{y}\times \mathbb{P}^{n-1} \ra \mathbb{P}(M_{n\times (N+1)}) \ (\mbox{Segre image} ) \ .
\end{align*}

Therefore we may consider $J_{1}(\mathbb{X}_{(n-1)}/Y)$. Letting $m\in \mathbb{Z}_+$ be sufficiently large we define another invertible sheaf $\mathscr{L}_\Delta(m)$ on $Y$ by
\begin{align*}
\mathscr{L}_{\Delta}(m):=\bigotimes _{0\leq j\leq 2n}\det \pi_{*}\left( \bigwedge^iJ_{1}(\mathbb{X}_{(n-1)}/Y)^{\vee}\otimes \mathcal{O}_{\mathbb{X}_{(n-1)}}(m)\right)^{(-1)^id(n+1)} \ .
\end{align*}
\begin{definition}\label{cmpolarization} 
\emph{ Let $\mathbb{X}\xrightarrow{\pi} Y$ be a flat family of polarized manifolds. The \emph{\textbf{CM-polarization}} of the family is the pair $(\mathscr{L}_{R}(m)\ ,\ \mathscr{L}_{\Delta}(m))$ of invertible sheaves on $Y$.}
\end{definition}
\begin{remark}\emph{The term ``polarization'' is strictly speaking inappropriate. This terminology is used for historical reasons only. 
The sheaves evidently depend on a choice of a large integer parameter $m$. It turns out that for $m$ large enough they are \emph{independent} of this choice and we denote the common value by $(\mathscr{L}_{R} \ ,\ \mathscr{L}_{\Delta} )$ .} 
\end{remark}
 
 \subsection{Proof of Theorem \ref{globgen}}
Let $Y$ be any $G$ variety. Recall that a $G$-morphism $w$ from $Y$ to a space of {divisors} on a flag variety 
\begin{align*}
w:Y\ra \mbox{Div}(G/P , \beta):= \mathbb{P}(H^{0}(G/P,\mathbb{L}_{\beta})) \ ( \mbox{$\beta$ dominant} ) \ 
\end{align*}
  is equivalent to the following data:
\begin{align}\label{morphism}
\begin{split}
&\mbox{An invertible sheaf $\mathcal{A}\in Pic(Y)^G$}\ . \\
\ \\
& \mbox{An algebraic $G$-invariant section $S$ of $\pi_1^*\mathcal{A}\otimes\pi_2^*\mathbb{L}_{\beta}$ over $Y\times G/P$ .} \\
\ \\
&t:={ \pi_1}_*S \neq 0 \quad t\in H^0(Y,\mathcal{A}\otimes H^{0}(G/P,\mathbb{L}_{\beta}))^G \ .
\end{split}
\end{align}
The pushed down section ${ \pi_1}_*S$ is defined by the formula
\begin{align*}
{ \pi_1}_*S(y):=S(y\ , \ \cdot )\in  \mathcal{A}_y\otimes H^{0}(G/P\ ,\ \mathbb{L}_{\beta}) \ .
\end{align*}
We define a {relative Cartier divisor} over $Y$
\begin{align*}
\mathcal{Z}:= \mbox{Div}(S)\subset Y\times G/P \ .
\end{align*}

The whole situation is pictured as follows
\begin{align}\label{picture}
\xymatrix{&\pi_1^*\mathcal{A}\otimes \pi_2^*\mathbb{L}_{\beta}\ar[d]\\
\mathcal{Z}\ar@{^{(}->}[r]& Y\times G/P \ar@/^1pc/[u]^{S}\ar[d]^{\pi_1}\\
& Y} \qquad
\xymatrix{\mathcal{A}\otimes H^{0}(G/P,\mathbb{L}_{\beta}) \ar[d] \\
Y\ar@/^1pc/[u]^{t:= {\pi_1}_*S}}
\end{align}

Since $t\neq 0$ , $t$ gives an injection
\begin{align*}
0\ra\mathcal{O}_Y\xrightarrow{\times t} \mathcal{A}\otimes H^{0}(G/P,\mathbb{L}_{\beta}) \ .
\end{align*}
Dualizing and tensoring with $\mathcal{A}$ gives a surjection (hence $\mathcal{A}$ is globally generated)
\begin{align*}
H^{0}(G/P,\mathbb{L}_{\beta})^{\vee}\times Y \ra \mathcal{A}\ra 0 \ ,
\end{align*}
and therefore a map $w:Y\ra  \mbox{Div}(G/P , \beta)$ as required. Moreover 
\begin{align*}
w^*\mathcal{O} (1)\cong \mathcal{A} \ 
\end{align*}
  and the canonical map
  \begin{align*}
  & \Phi:H^{0}(G/P,\mathbb{L}_{\beta})^{\vee}\ra H^0(Y, \mathcal{A}) \\
  \ \\
  & \Phi(\alpha)(y):=\alpha(t(y))\in \mathcal{A}_y \ 
  \end{align*}
is an injection. 
Conversely, given such a map $w$, we define 
 \begin{align*}
 &\mathcal{A}:=w^*\mathcal{O} (1) \\
 \ \\
 &S(y,x):=\widehat{w}(y)(x)\in\mathcal{A}_y\otimes\mathbb{L}_{\beta} \ . 
 \end{align*}
 \subsection{Determinants of Cayley-Koszul Complexes}
In order to proceed with the proof we first review in this subsection the situation (treated, for example in \cite{gkz} , \cite{paul2009}) where the base of the family $\mathbb{X}\ra Y$ consists of a single orbit $\mathcal{O}$, in other words we are really studying a fixed projective variety $X$ (up to the action of $G$). We refer the reader to the references for fuller discussion and complete proofs.

In this case the diagram (\ref{picture}) reduces to
\begin{align}
 \xymatrix{&  \mathbb{L}_{\beta}\ar[d]\\
{Z}\ar@{^{(}->}[r]&  G/P \ar@/^1pc/[u]^{S}
 }
\end{align}

Our question in this situation is the following one: 
\begin{center}{Given ${Z}$ how can we find $S$}? \end{center}

We begin by studying the situation in an affine space.
Let  $Z$ be an irreducible algebraic hypersurface in some $\mathbb{C}^{k}$. We need to assume that the basic set up is in force.\\
\ \\
\textbf{ The Basic Set Up .}
\emph{There exists a complex   variety $X$ and a vector subbundle $\mathcal{S}$ of the trivial bundle $\mathcal{E}:= X\times \mathbb{C}^{k}$ such that the image of the restriction to $I=\mathcal{S}$  of the projection of $\mathcal{E}$ onto  $\mathbb{C}^{k}$ is $Z$.}\\
\ \\
 
 Consider the exact sequence of vector bundles on $X$
\begin{align*}
0\rightarrow \mathcal{S}\rightarrow \mathcal{E}\overset{\pi}\rightarrow \mathcal{Q}\rightarrow 0 \ .
 \end{align*}
There is a section $\xi$ of ${p_1}^*(\mathcal{Q})$ whose base locus is $I$ where $p_1$ is the projection of $\mathcal{E}$ to $X$ and $Z$ equals the image of $I$ under ${p_2}|_I$. This situation is pictured below.
\begin{align*}
\xymatrix{& p_1^*\mathcal{Q}\ar[r]^{\pi_2}\ar[d]^{\pi_1}& \mathcal{Q}\ar[d]\\
I\ar[r]^{\iota}\ar[d]^{{p_2}|_{I}}&X\times \mathbb{C}^{k}\ar@/^1pc/[u]^{\xi}\ar[d]^{p_2}\ar[r]^-{p_1}&X\\
Z\ar[r]^{i}&\mathbb{C}^{k}&}
\end{align*}

Now we attempt to describe the irreducible defining polynomial of $Z$ (denoted by $R_{Z} $ ) through an analysis of the direct image of a Cayley-Koszul complex of sheaves on ${X}\times \mathbb{C}^{ {k}}$. Let $E:=p_1^*\mathcal{Q}$
\begin{align*}
\xymatrix{& K^{\bull}(E) \ar[d]&\\
&X\times \mathbb{C}^{k}\ar[ld]_{p_1}\ar[rd]^{p_2}&\\
X& &\mathbb{C}^{k}}
\end{align*}

We have the free resolution over $\mathcal{O}_{X\times \mathbb{C}^{k}}$
\begin{align}
(K^{\bull}(E), i(\xi) )\rightarrow \iota_{*}\mathcal{O}_{I}\rightarrow 0\ ; \ K^j(E):=\bigwedge^{n+1-j}E^{\vee}\ .
\end{align}
$i(\xi)$ denotes contraction.
Let $\mathcal{V}$ denote any vector bundle on $X$ . Consider the {twisted} 
 complex
\begin{align}
\begin{split}
&(K^{\bull}(E)\otimes p_1^*\mathcal{V} , \ i(\xi))\rightarrow \iota_{*}\mathcal{O}_{I}\otimes  p_1^*\mathcal{V}\rightarrow 0  \ .
\end{split}
\end{align}
 Given $f\in  \mathbb{C}^{k}$ pull the complex back to $X$ via the map
\begin{align}
i_{f}:X\rightarrow X\times \mathbb{C}^{k} \quad i_f(x):= (x,f)
\end{align}

 The crucial point is that $i_f^*(K^{\bull}(E)\otimes p_1^*\mathcal{V} , i(\xi))$ is an 
 exact complex of locally free sheaves on $X$ whenever  $f\in \mathbb{C}^{k}\setminus Z$.

Next we make a {positivity assumption} on $\mathcal{V}$
 \begin{align}
 \mbox{ {Assume that}}\ H^j(X, \ K^i(E)\otimes \mathcal{V})=0 \  \mbox{ {for all $i$ and all $j>0$}}\ .
 \end{align}
Given $X$ we call the package
\begin{align}
\left(\mathcal{E} \ , \mathcal{Q}\ , \mathcal{S}\ , \mathcal{V}\right)
\end{align}
the \emph{basic data} for $Z$.

It follows from the Leray hypercohomology spectral sequence that the complex of {finite dimensional vector spaces} is also exact 
\begin{align}
\begin{split}
&(E^{\bull}(\mathcal{V})\ , \dl_{f}):= (H^0(X, \ K^{\bull}(E)\otimes \mathcal{V}) , \dl^{\bull}_{f})\quad f\in \mathbb{C}^{k}\setminus Z \\
\ \\
&\dl^{\bull}_{f}= i(\xi) \ .
\end{split}
\end{align}
 
Recall that the torsion \footnote{For the definition and basic properties of the torsion (determinant) of an exact complex see the appendix in \cite{gkz} or \cite{paul2009}.}spans the determinant of the complex .
\begin{align}
\begin{split}
 &\mathbb{C}{\mathbf{Tor}}(E^{\bull}(\mathcal{V}),\dl^{\bull}_{f})=\bigotimes_{j=0}^{n+1}\bigwedge^{b_j}H^0(X,\bigwedge^j\mathcal{Q}^{\vee}\otimes \mathcal{V})^{(-1)^{j+1}} \ , \ b_j:= h^0(X, \bigwedge^{j} \mathcal{Q}^{\vee}\otimes\mathcal{V}) \ .
\end{split}
\end{align}
Choosing bases $\{ e^{(\bull)}_j\}$ in each term of $E^{j}(\mathcal{V})$
exhibits the {torsion} of the complex as a nowhere zero {regular} function away from $Z$
\begin{align}\label{regular}
f\in \mathbb{C}^{k}\setminus Z \rightarrow   {\mathbf{Tor}}(E^{\bull}(\mathcal{V}),\dl^{\bull}_{f};\{ e^{(\bull)}_j\})\in \mathbb{C}^* \ .
\end{align}
Therefore this function must be a power of  $R_{Z}(\cdot )$.  
\begin{proposition}\emph{
There is an integer $q$ (the {{Z-adic order}} of the determinant) such that}
\begin{align}\label{ratpwr}
&{\mathbf{Tor}}(E^{\bull}(\mathcal{V}),\dl^{\bull}_{f};\{ e^{(\bull)}_j\})=R_{Z}(f)^q \ .
\end{align}
\end{proposition}
 
 \begin{remark}\emph{ In particular the torsion of the complex $(E^{\bull}(\mathcal{V}),\dl^{\bull}_{f})$ is a \textbf{\emph{polynomial}}, or the reciprocal of a polynomial.} 
\end{remark}

Since the boundary operators $\dl_f$ are  \emph{ {linear}} in $f$ we may deduce the following.
\begin{proposition}\label{wtdeulerchar} \emph{The $Z$-adic order can be computed as follows}
\begin{align} 
q\mbox{\emph{deg}}(R_Z)= \sum_{j=0}^{n+1}(-1)^{j+1}jh^0(X, \bigwedge^{j}\mathcal{Q}^{\vee}\otimes\mathcal{V})   \ .
\end{align}
\end{proposition}
Let $X\ra \cpn$ be an $n$ dimensional complex projective variety. Let $Z=Z_X$ be the associated hypersurface. 
 In this case the basic data for $Z$ may be chosen as follows :
 \begin{align}\label{bdres}
\begin{split}
&\mathcal{E}=X\times M_{n+1, N+1}(\mathbb{C})\\
\ \\
&\mathcal{S}=\{(x,(l_0,l_1,\dots,l_n))|\ l_i(x)=0 \ ; 0\leq i\leq n\} \\
& l_i \ \mbox{denotes a linear form on $\mathbb{C}^{N+1}$}.\\
&\mathcal{Q}\cong  \overbrace{ \mathcal{O}(1)_X\oplus \mathcal{O}(1)_X \oplus \dots \oplus \mathcal{O}(1)_X}^{n+1}  \\
\ \\
&\mathcal{V}= \mathcal{O}_X(m)\quad m>>0 \in \mathbb{Z} \ .
\end{split}
\end{align}
Let $X\ra \cpn$ be a smooth $n$ dimensional complex projective variety. Let $Z=X^{\vee} $ be the dual variety of $X$. 
 Let 
\begin{align*}
\rho:X\ra \mathbb{G}(n,N)
\end{align*}
be the Gauss map of $X$ and $\mathscr{U}$ the rank $n+1$ tautological bundle over $\mathbb{G}(n,N)$. 
In this case  the basic data for $Z$ may be chosen as follows :
\begin{align}\label{bddisc}
\begin{split}
&\mathcal{E}=X\times(\mathbb{C}^{N+1})^{\vee}\\
\ \\
&\mathcal{S}=\{(p, f)\ |\ \rho(p) \subset \mbox{Ker}(f)\} \ , \ f\in (\mathbb{C}^{N+1})^{\vee}\setminus\{0\}  \\
 \ \\
&\mathcal{Q}=  \rho^*(\mathscr{U}^{\vee}) \\
\ \\
&\mathcal{V}=\mathcal{O}_X(m)\ , \  m>>0 \in \mathbb{Z}\ .
\end{split}
\end{align}
Now we may state the following result which is required for the proof of Theorem \ref{globgen}.
 \begin{proposition}\label{ord1}(Paul \cite{paul2009})
\begin{equation}
\ \sum_{j=0}^{n+1}(-1)^{j+1}jh^0(X, \bigwedge^{j}\mathcal{Q}^{\vee}\otimes\mathcal{V})= 
\begin{cases}    \deg(R_X) & \text{$Z$ is the associated hypersurface of $X$,}
\\
\\
 \deg(\Delta_X) &\text{ $Z$ is the dual variety of $X$ .}
\end{cases}
\end{equation}
\end{proposition}
In particular we see that in these cases $q$, the $Z$-adic order of the torsion, is equal to one . \newline
 

\subsection{Families of Resultants}
We consider a flat family $\mathbb{X}\xrightarrow{\pi} Y$ of polarized varieties \footnote{Smoothness is not necessary for resultants .} . We assume, although this can be avoided,  that $Y$ is smooth. The situation we will consider may be pictured as follows.

\begin{align}
\xymatrix{&E\ar[d]\ar[r]& p_1^*\mathcal{O}_{\cpn}(1)\otimes p_2^*\mathcal{Q}\ar[d] \\
 I_{\mathbb{X}}\ar@{^{(}->}[r]^-{\iota} \ar[d] &  \mathbb{X}\times\mathbb{G} \ar@/^1pc/[u]^{q^*(\xi)}\ar[r]^-{q}\ar[d]^{p}&\cpn\times \mathbb{G}  \ar@/^1pc/[u]^{\xi}\\
\mathcal{Z}\ar@{^{(}->}[r]&Y\times \mathbb{G}  \ar[d]^{p_1}&\\
 &Y& }
\end{align}
In the diagram above we have defined
\begin{align}
\begin{split}
& i)\ \xi|_{([v],L)}:\mathbb{C}v\ra \mathbb{C}^{N+1}/L \ , \ \quad \xi(zv)=\pi_{L}(zv) \\
& \quad \pi_L:\mathbb{C}^{N+1}\ra \mathbb{C}^{N+1}/L \ \mbox{denotes the projection} . \\
&\quad \mbox{Observe that}\  \xi|_{([v],L)}=0 \ \mbox{if and only if}\ v\in L \ . \\
 \ \\
&ii) \ q:\mathbb{X}\times\mathbb{G}\ra \cpn\times \mathbb{G} \quad \mbox{is defined by}\ q(x,L):=(\pi_2(x),L) \ . \\
&\quad \mathbb{G}:=\mathbb{G}(N-n-1,N) \ .\\
\ \\
& iii) \ E:=q^*\left(p_1^*\mathcal{O}_{\cpn}(1)\otimes p_2^*\mathcal{Q}\right) \\
&\ \ \quad I_{\mathbb{X}}:= (q^*(\xi)=0)  \ . \\
 \ \\
& iv)\ p:\mathbb{X}\times \mathbb{G}\ra Y\times \mathbb{G} \quad \mbox{is defined by} \ p(x,L):=(\pi(x),L)\\
& \quad \mathcal{Z}:=p( I_{\mathbb{X}})\ .
\end{split}
\end{align}
\begin{remark}
\emph{The reader should observe that $\mathcal{Z}\cap (\{y\}\times \mathbb{G})$ is the associated hypersurface of $X_y$.}
\end{remark}
In exactly the same way as before we have a Cayley-Koszul complex  $K^{\bull}(E)$ (over $\mathbb{X}\times \mathbb{G}$) associated to the structure sheaf of $I_{\mathbb{X}}$ , where $i(q^*(\xi))$ denotes contraction 
\begin{align*}
K^{\bull}(E):= \bigwedge^{n+1}E^{\vee}\xrightarrow{i(q^*(\xi))}\bigwedge^{n}E^{\vee}\xrightarrow{i(q^*(\xi))}  \dots \xrightarrow{i(q^*(\xi))} E^{\vee}\xrightarrow{i(q^*(\xi))} \mathcal{O}_{\mathbb{X}\times \mathbb{G}} \ra {\iota_*\mathcal{O}_{I_{\mathbb{X}}}}  \ .
\end{align*}
We are interested in the direct image of the twisted complex $K^{\bull}(E)(m)$ ( $m>>0$ ) .
\begin{align}
 K^{i}(E)(m)= \left(\bigwedge^iE^{\vee}\otimes \mathcal{O}_{\mathbb{X}}(m)\right) \qquad (0\leq i\leq n+1)\  .
\end{align}
\begin{proposition}\label{mainres}
\emph{
\begin{align*}
& i) \ \mbox{All terms of}\  p_{*}K^{\bull}(E)(m) \ \mbox{are locally free}  \ . \\
\ \\
& ii) \ p_{*}K^{\bull}(E)(m) \ \mbox{is exact away from} \ \mathcal{Z} \ . \\
 \ \\
& iii) \quad \mbox {There is an isomorphism of sheaves on $Y\times \mathbb{G}$} \\
 &\quad  \det  p_{*}K^{\bull}(E)(m) ^{\otimes (n(n+1)d-d\mu)}\cong p_1^*(\mathscr{L}_{R}(m) )\otimes p_2^*\mathcal{O}_{\mathbb{G}}\big(\frac{r}{n+1}\big) \ . \\
 \ \\
& iv) \ \mbox{There is a canonical section} \ S\in H^0(Y\times \mathbb{G}, \det  p_{*}K^{\bull}(E)(m) ) \ \mbox{such that}  \\
&\quad \mbox{ for every} \ y\in Y \ , \ S(y,\cdot) \in \mathbb{P}H^0(\mathbb{G},\mathcal{O}_{\mathbb{G}}(d)) \ \mbox{satisfies} \ Div(S(y,\cdot))=Z_{X_y}\ . \\   
& \quad  \mbox{Therefore $Div(S)=\mathcal{Z}$ ,}\ \det  p_{*}K^{\bull}(E)(m)\cong \mathcal{O}(\mathcal{Z})\  \mbox{and} \ \mathcal{Z}\ra Y \ \mbox{has the}\\
&\quad  \mbox{ structure of a relative Cartier divisor over} \ Y \ .  
\end{align*}}
\end{proposition}
Recall that $\mathcal{O}_{\mathbb{G}}(k):=\det(\mathcal{Q})^{\otimes k}$ for any $k\in\mathbb{Z}$ .
\begin{proof}
$i)$ and $ii)$ follow at once from general theory. 
The proof of $iii)$ is an easy computation. To begin we observe that
\begin{align}
\begin{split}
&\bigwedge^j\mathcal{E}^{\vee}_{n+1}\cong \mathcal{O}_{\mathbb{X}}(-j)\otimes\bigwedge^j\mathbb{C}^{n+1}  \\
\ \\
& \bigwedge^j\mathcal{E}^{\vee}_{n+1}\otimes \mathcal{O}_{\mathbb{X}}(m)\cong \mathcal{O}_{\mathbb{X}}(m-j)\otimes\bigwedge^j\mathbb{C}^{n+1} \\
\ \\
&\pi_*(\mathcal{O}_{\mathbb{X}}(m-j)\otimes\bigwedge^j\mathbb{C}^{n+1})\cong \pi_*(\mathcal{O}_{\mathbb{X}}(m-j))\otimes\bigwedge^j\mathbb{C}^{n+1} \\
\ \\
&\det(\pi_*(\mathcal{O}_{\mathbb{X}}(m-j))\otimes\bigwedge^j\mathbb{C}^{n+1})\cong \det(\pi_*(\mathcal{O}_{\mathbb{X}}(m-j))^{\binom{n+1}{j}} \ .
\end{split}
\end{align}
Therefore
\begin{align}
\left(\bigotimes_{0\leq j\leq n+1}\det(\pi_*(\mathcal{O}_{\mathbb{X}}(m-j))^{(-1)^j\binom{n+1}{j}}\right)^{\otimes (n(n+1)d-d\mu)}\cong \mathscr{L}_{R}  \ .
\end{align}
Next we see that
\begin{align}
K^{i}(E)(m)\cong p_1^*\mathcal{O}_{\mathbb{X}}(m-i)\otimes p_2^*\bigwedge^i\mathcal{Q}^{\vee} \ .
\end{align}
Therefore
\begin{align*}
& \det p_*(K^{i}(E)(m)) \cong p_1^*(\det \pi_*\mathcal{O}_{\mathbb{X}}(m-i))^{\binom{n+1}{i}}\otimes p_2^*\det (\bigwedge ^i \mathcal{Q}^{\vee})^{P(m-i)}\\
\ \\
& P(m-i)= h^0(X_y , \mathcal{O}_{\mathbb{X}}(m-i)|_{X_y}) \ .
\end{align*}
By the splitting principal and the fact that $\mbox{Pic}(\mathbb{G}(k,n))\cong \mathbb{Z}$ we have the following isomorphism for $1\leq i \leq n+1$
\begin{align}
\det (\bigwedge ^i \mathcal{Q}^{\vee})\cong \det (\mathcal{Q}^{\vee})^{\binom{n}{i-1}} \ .
\end{align}
Therefore
\begin{align}
\begin{split}
&\left(\bigotimes_{0\leq i\leq n+1}\det(p_*K^{i}(E)(m))^{(-1)^i}\right)^{\otimes(dn(n+1)-d\mu)}\cong p_1^*(\mathscr{L}_{R} )\otimes p_2^*\det(\mathcal{Q})^{l} \ , \\
\ \\
&l:=(dn(n+1)-d\mu)\sum_{i=0}^{n+1}(-1)^{i+1}\binom{n}{i-1}P(m-i) \ .
\end{split}
\end{align}
Part $iii)$ now follows from the elementary identity
\begin{align}
\sum_{i=0}^{n+1}(-1)^{i+1}\binom{n}{i-1}P(m-i)=d \ . 
\end{align}
 The proof of $iv)$ depends on the preceding discussion together with a straightforward extension of that discussion. More precisely we need the following well known result which can be traced back to work of Arthur Cayley.
\begin{lemma}\label{torsion} 
\emph{Let $(\mathcal{E}^{\bull}, \dl_{\bull})$ be a bounded complex of locally free sheaves on a variety $\mathscr{X}$. If this complex is exact then the torsion provides a global nonvanishing section of the determinant line bundle
\begin{align}
\mbox{det}(\mathcal{E}^{\bull}, \dl_{\bull}):= \bigotimes_{i}\bigwedge^{r_i}(\mathcal{E}^{i})^{(-1)^{i+1}} \overset{\mathbf{Tor}}{\cong } \mathcal{O}_{\mathscr{X}}\ ,\ r_i:=\rnk(\mathcal{E}^{i})\  .
\end{align}}
\end{lemma}
We apply  Lemma \ref{torsion} in the present situation by choosing
\begin{align}
\begin{split}
& \mathcal{E}^{i}=p_*K^i(E)(m) \ , \\
\ \\
& \mathscr{X}= Y\times \mathbb{G}\setminus \mathcal{Z} \ .
\end{split}
\end{align}
Then we have
\begin{align}
\begin{split}
&\mathbf{Tor}(p_*K^{\bull}(E)(m)\ , \ \dl_{\bull})^{\otimes\mu(n)}\in H^0\big(Y\times \mathbb{G}\setminus\mathcal{Z}\ , \ p_1^*(\mathscr{L}_{R}(m) )\otimes p_2^*\mathcal{O}_{\mathbb{G}}(d\mu(n)) \big) \\
\ \\
&\mu(n):=n(n+1)d-d\mu \ .
 \end{split}
 \end{align}
\begin{lemma}
\emph{ $\mathbf{Tor}(p_*K^{\bull}(E)(m)\ , \ \dl_{\bull})$ {vanishes identically} on $\mathcal{Z}$, the $\mu(n)^{th}$ power extends to a global section $S$ of $p_1^*(\mathscr{L}_{R}(m) )\otimes p_2^*\mathcal{O}_{\mathbb{G}}(d\mu(n))$ satisfying condition (\ref{morphism}) , and the induced morphism
\begin{align}
Y\ra \mathbb{P} H^0(\mathbb{G}, \mathcal{O}(d\mu(n)))\cong \mathbb{P}(\elam)
\end{align}
coincides with $R$ . Therefore $R$ is a regular map of quasi projective varieties.} 
\end{lemma}
\begin{proof} Observe that
\begin{align}
(\{y\}\times \mathbb{G})\setminus (\mathcal{Z}\cap (\{y\}\times \mathbb{G})) = \mathbb{G}\setminus Z_{X_y} \ .
\end{align}
Under the dominant (rational) map
\begin{align}
\pi:M_{(n+1)\times (N+1)}\setminus \overline{\pi^{-1}(Z_{X_y})}\dashrightarrow \mathbb{G}\setminus Z_{X_y}
\end{align}
we have that 
\begin{align}
\pi^*\big(\mathbf{Tor}(p_*K^{\bull}(E)(m)\ , \ \dl_{\bull})|_{\mathbb{G}\setminus Z_{X_y}}\big) 
\end{align}
coincides with the torsion of the complex constructed from the basic data (\ref{bdres}). By Propositions \ref{wtdeulerchar} and \ref{ord1} this is the Cayley Chow form of $X_y$. Therefore $\mathbf{Tor}(p_*K^{\bull}(E)(m)\ , \ \dl_{\bull})$ vanishes on $\mathcal{Z}\cap (\{y\}\times \mathbb{G})$ for any $y\in Y$. Therefore it vanishes identically on $\mathcal{Z}$ and hence extends to a global section . The remaining properties are immediate.
\end{proof}
\begin{corollary} \emph{Let $\mathcal{O}(\mathcal{Z})$ denote the line bundle corresponding to $\mathcal{Z}$. Then there is an isomorphism
\begin{align}
\mathcal{O}(\mathcal{Z})\cong p_1^*(\mathscr{L}_{R}(m) )\otimes p_2^*\mathcal{O}_{\mathbb{G}}(d\mu(n)) \ .
\end{align}
Consequently the sheaf $\mathscr{L}_{R}(m)$ is independent of $m$.}
\end{corollary}
This completes the proof of Proposition \ref{mainres} .
\end{proof}

\subsection{Families of Discriminants}
In this section we consider a flat family $\mathbb{X}\xrightarrow{\pi} Y$ of polarized manifolds satisfying the non-degeneracy condition
\begin{align*}
\pi_*(c_n(J_1(\mathbb{X}/Y))\equiv d^{\vee} \in \mathbb{Z}_+ \ .
\end{align*}
That is, all of the fibers $X_y$ are dually non-degenerate and the discriminants of the fibers have (common) degree $d^{\vee}>0$ . This follows from a result of \cite{bfs} . The situation we will consider may be pictured as follows.
\begin{align*}
\xymatrix{&F\ar[d]\ar[r]& p_1^*\mathscr{U}^{\vee}\otimes p_2^*\mathcal{O}_{{\cpn}^{\vee}}(1)\ar[d] \\
 J_{\mathbb{X}}\ar@{^{(}->}[r]^-{\iota} \ar[d] &  \mathbb{X}\times{\cpn}^{\vee} \ar@/^1pc/[u]^{q^*(\zeta)}\ar[r]^-{q}\ar[d]^{p}&\mathbb{G}(n,N)\times {\cpn}^{\vee}  \ar@/^1pc/[u]^{\zeta}\\
\mathcal{Z}\ar@{^{(}->}[r]&Y\times {\cpn}^{\vee}  \ar[d]^{p_1}&\\
 &Y& }
\end{align*}
In the diagram above we have defined
\begin{align*}
& i)\ \zeta|_{(L,[f])}:L\ra \mathbb{C}^{N+1}/\ker(f) \ , \ \quad \zeta(u)=\pi_{\ker(f)}(u) \\
& \quad \pi_{\ker(f)}:\mathbb{C}^{N+1}\ra \mathbb{C}^{N+1}/{\ker(f)} \ \mbox{denotes the projection} . \\
&\quad \mbox{Observe that}\  \zeta|_{(L,[f])}=0 \ \mbox{if and only if}\ L\subset \ker(f)\ . \\
 \ \\
&ii) \ q:\mathbb{X}\times{\cpn}^{\vee}\ra \mathbb{G}(n,N)\times {\cpn}^{\vee}  \ \mbox{is defined by}\ q(x,[f]):=(\rho_{\mathbb{X}}(x),[f]) \ . \\
&\quad \mathbb{G}:=\mathbb{G}(N-n-1,N) \ . \\
\ \\
& iii) \ F:=q^*\left( p_1^*\mathscr{U}^{\vee}\otimes p_2^*\mathcal{O}_{{\cpn}^{\vee}}(1)     \right) \\
&\ \ \quad J_{\mathbb{X}}:= (q^*(\zeta)=0) \ . \\
 \ \\
& iv)\ p:\mathbb{X}\times{\cpn}^{\vee}\ra Y\times {\cpn}^{\vee}  \ \mbox{is defined by} \ p(x,[f]):=(\pi(x),[f])\\
& \quad \mathcal{Z}:=p( I_{\mathbb{X}})\ .
\end{align*}
\begin{remark}
\emph{The reader should observe that $\mathcal{Z}\cap (\{y\}\times {\cpn}^{\vee})$ is the dual variety of the fiber $X_y$.}
\end{remark}
As in the resultant case we study the Cayley-Koszul complex associated to the structure sheaf of $J_{\mathbb{X}}$ :
\begin{align*}
0\ra \bigwedge^{n+1}F^{\vee}\xrightarrow{i(q^*(\zeta))}\bigwedge^{n}F^{\vee}\xrightarrow{i(q^*(\zeta))}  \dots \xrightarrow{i(q^*(\zeta))} F^{\vee}\xrightarrow{i(q^*(\zeta))} \mathcal{O}_{\mathbb{X}\times{\cpn}^{\vee}} \xrightarrow{i(q^*(\zeta))} {\iota_*\mathcal{O}_{J_{\mathbb{X}}}}\ra 0 \ .
\end{align*}
We are interested in the twisted complex $K^{\bull}(F)(m)$ of sheaves on $\mathbb{X}\times {\cpn}^{\vee}$ for $m>>0$ 
\begin{align*}
K^{i}(F)(m):=\bigwedge^iF^{\vee}\otimes \mathcal{O}_{\mathbb{X}}(m) \ (0\leq i\leq n+1)\  .
\end{align*}

\begin{proposition}\label{discrfamily}
\emph{ 
\begin{align*}
& i) \ \mbox{All terms of}\  p_{*}K^{\bull}(F)(m) \ \mbox{are locally free}  \ . \\
\ \\
& ii)  \ p_{*}K^{\bull}(F)(m) \ \mbox{is exact away from} \ \mathcal{Z} \ . \\
\ \\
& iii)\ \mbox{We define an invertible sheaf $\mathcal{A}$ on $Y$ by} \\
 &\mathcal{A}:= \bigotimes_{0\leq i\leq n+1}\det\pi_*(\bigwedge^iJ_1(\mathbb{X}/Y)^{\vee}\otimes \mathcal{O}_{\mathbb{X}}(m))^{(-1)^i}\ . \\
 &\mbox{Then $\mathcal{A}$ is independent  of $m$ and there is an isomorphism of sheaves on}\  Y\times {\cpn}^{\vee} \\
 &\det p_{*}K^{\bull}(F)(m)\cong p_1^*(\mathcal{A})\otimes p_2^*\mathcal{O}_{{\cpn}^{\vee}}(d^{\vee}) \ . 
  \end{align*}
 \begin{align*}
& iv)   \ \mbox{There is a canonical section} \ \Delta \in H^0(Y\times {\cpn}^{\vee}, \det  p_{*}K^{\bull}(F)(m) ) \ \mbox{such that }   \\
&   \mbox{for every $y\in Y$ $\Delta(y,\cdot) \in \mathbb{P}H^0({\cpn}^{\vee},\mathcal{O}_{{\cpn}^{\vee}}(d^{\vee})) $}\ \mbox{satisfies} \ Div(\Delta(y,\cdot))={X_y}^{\vee}\ . \\
&\mbox{Therefore $Div(\Delta)=\mathcal{Z}$, $\det p_{*}K^{\bull}(F)(m)\cong\mathcal{O}(\mathcal{Z})$ and $\mathcal{Z}\ra Y$ has the}\\
& \mbox{structure of a relative Cartier divisor over $Y$} \ .
   \end{align*}}
\end{proposition}

\begin{proof} The proof is identical to the one for resultants. We will just give an indication of $iii)$. 
\begin{align*}
&p_*K^{i}(F)(m)\cong p_1^*\ \pi_*\left(\bigwedge^iJ_1(\mathbb{X}/Y)^{\vee}\otimes\mathcal{O}_{\mathbb{X}}(m)\right)\otimes\mathcal{O}_{{\cpn}^{\vee}}(-i) \ . \\
\ \\
&\mbox{Therefore , we have the following }\\
& \det(p_*K^{i}(F)(m))\cong p_1^*\det \pi_*\left(\bigwedge^iJ_1(\mathbb{X}/Y)^{\vee}\otimes\mathcal{O}_{\mathbb{X}}(m)\right)\otimes\mathcal{O}_{{\cpn}^{\vee}}(-ih^0(i;m)) \ , \\
& h^0(i;m):=h^0(X_y , \bigwedge^i(J_1(\mathcal{O}_{\cpn}(1)|_{X_y})(m)) \ .
\end{align*}
Therefore
\begin{align*}
\bigotimes_{0\leq i\leq n+1}\det(p_*K^{i}(F)(m))^{(-1)^{i+1}}\cong p_1^*(\mathcal{A})\otimes p_2^*\mathcal{O}_{{\cpn}^{\vee}}(\sum_{0\leq i\leq n+1}(-1)^{i+1}ih^0(i;m)) \ .
\end{align*}
$iii)$ the follows from the next claim. For the proof the reader should see \cite{paul2009} pgs. 1357-1362.
\begin{claim}\emph{The function
\begin{align*}  
Y\ni y\ra \sum_{0\leq i\leq n+1}(-1)^{i+1}i  h^0(X_y , \bigwedge^i(J_1(\mathcal{O}_{\cpn}(1)|_{X_y})(m))   
\end{align*}
is constant with value  $\pi_*(c_n(J_1(\mathbb{X}/Y))$, the degree of the dual of any fiber of the family $\mathbb{X}\ra Y$.  }
\end{claim}
\end{proof}
To complete the proof of Theorem \ref{globgen} one twists the given family $\mathbb{X}\ra Y$ fiberwise by $\mathbb{P}^{n-1}$ (see (\ref{newfamily}) ) then applies Proposition \ref{discrfamily} to this new family. For the computation of the degree of the hyperdiscriminant recall Proposition \ref{degreeofhyp} . Further details are left to the reader.  
 
Part $(2)$ of Theorem \ref{uniform1} requires Theorem A together with corollary \ref{infi}. The dependence of the constant $M$ on line (\ref{M}) comes from the definition of the weights $\lambda_{\bull}$ and $\mu_{\bull}$ and the discussion found after remark 5 on pg. 276 of \cite{paul2011} . Part $(4)$ follows from Proposition \ref{constructible} .  

Theorem \ref{weakII} follows from Theorem A and Proposition \ref{curveselection} .   
  
\section{ Asymptotics as $k\ra \infty$ }

Let $k\in \mathbb{Z}_+$ and let  $(X,L)$ be a polarized manifold. Fix a Hermitian metric $h$ on $L$ with curvature form $\om_h=-\sqrt{-1}\dl\dlb\log h$. Then the Bergman space associated to the Kodaira embedding
\begin{align*}
X\overset{L^k}{\ra}\mathbb{P}^{N_k}
\end{align*}
is given by 
\begin{align*}
\mathcal{B}_k:=\bigg\{ \frac{1}{k}\vps \ |\ \sigma\in SL(N_k+1,\mathbb{C})\bigg\} \ .
\end{align*}
The importance of the spaces $\mathcal{B}_k$ is brought out in the following theorem, due to Tian.
\begin{theorem}\label{density}(\cite{tian1990})\ 
\emph{The union $\bigcup_{k\geq 1}\mathcal{B}_k$  is dense in $\mathcal{H}_{\om}$ in the $C^2$ topology.} 
\end{theorem}

Let $R(k), \Delta(k), M(k)$ be the $k$ dependent data associated to the embedding
\begin{align*}
 X \xrightarrow{{L}^k}\mathbb{P}^{N_k} 
\end{align*}
An immediate consequence of Theorems \ref{weakII} and \ref{density} is the following.
\begin{theorem}\label{corethm} \emph
 {{Assume} the sequence $\{M(k)\}_{k\geq1} $ is bounded by a constant $M$. Then the infimum of the Mabuchi energy over the entire space of K\"ahler metrics in the class $[\om]$ is given by}
\begin{align*}
 |(n+1)\inf_{\vp\in\mathcal{H}_{\om}}\nu_{\om}(\vp) - \inf_{k\in\mathbb{Z}_+}\frac{1}{k^{2n}}\log\tan^2 d _{g} (\overline{\mathcal{O}}_{R(k)  \Delta(k)},\overline{\mathcal{O}}_{R(k)} )|\leq M   \ . 
 \end{align*}
\end{theorem}
 
Theorems \ref{density} and \ref{corethm} motivate the following definition.
\begin{definition}\label{asympss}\emph{Let $(X, {L})$ be a polarized manifold. Then $(X, {L})$ is \textbf{\emph{asymptotically semistable}} provided }
\begin{align}
|\inf_{k\in\mathbb{Z}_+}\frac{1}{k^{2n}}\log\tan^2 d _{g} (\overline{\mathcal{O}}_{R(k)  \Delta(k)},\overline{\mathcal{O}}_{R(k)} )|<\infty \ .
\end{align}
\end{definition}
 
 The current status of the author's approach to the standard conjectures is exhibited in the following flow chart.   
 \ \\
 
\tikzstyle{decision} = [rectangle, draw, fill=blue!20, 
    text width=4.5em, text badly centered, node distance=3cm, inner sep=0pt]
\tikzstyle{block} = [rectangle, draw, fill=blue!20, 
    text width=5em, text centered, rounded corners, minimum height=4em]
\tikzstyle{line} = [draw, -latex']
\tikzstyle{cloud} = [draw, ellipse,fill=red!20, node distance=3cm,
    minimum height=2em]
\begin{center}    
\begin{tikzpicture}[node distance = 3.5cm, auto]
   \node [block] (canon) {Existence of a Canonical metric in the class $[\om]=c_1(L)$};
   \node [block, right of=canon] (bound) { Mabuchi energy lower bound on $\mathcal{H}_{\om}$};
    \node [block, right of= bound] (berg) {\emph{\textbf{Uniform}} lower Mabuchi energy bound on $\mathcal{B}_{k}$ (all $k$)};
    \node [block, right of=berg] (berg2) { Lower Mabuchi energy bound on $\mathcal{B}_{k}$ (fixed $k$)};
    \node [block, below of=berg2] (stability) {Semistable embedding $X \xrightarrow{ {L}^k}\mathbb{P}^{N_k}$ (fixed $k$)  };
    \node [block, below of=bound] (asympstability) {$(X , L)$ \emph{\textbf{asymptotically}} semistable};
     \path [line] (canon) --node{(1)} (bound);
    \path [line,dashed] (bound) -- (asympstability);
    \path [line] (bound) -- node{(2)}(berg);
    \path [line] (berg2) --node{(4)}(stability);
    \path [line] (asympstability) --node{(5)}(stability);
   \path [line,dashed] (asympstability) --node{(6)}(bound);
    \path [line] (berg) -- (bound);
   \path [line] (berg) -- node{(3)}(berg2);
   \path [line] (stability) -- (berg2);
  \end{tikzpicture}
 \end{center}
\begin{center}{\small{Figure 3.}}\end{center}

\begin{enumerate}
\item This is due to Bando and Mabuchi when $L=-K_X$ \cite{bando-mabuchi87}. See Chen and Tian for the general case \cite{chentian08}. See also Donaldson \cite{skd2005} and Li \cite{chili2009} .
 \ \\
\item This is Tian's density Theorem \cite{tian1990}.  
 \ \\
\item Obvious.
 \ \\
\item This equivalence is due to the author \cite{paul2011}. Partial results can be found in \cite{tian97}.
 \ \\
 \item Follows from the definition of asymptotic semistability .
\ \\
\item This equivalence holds provided the sequence $\{M(k)\}_{k\geq1} $ is bounded.  
\end{enumerate}
\ \\
 
\begin{conjecture}
 \emph{The sequence $\{M(k)\}_{k\geq1} $ is bounded.}
 \end{conjecture}
  \begin{center}{\textbf{Acknowledgements}}\end{center}
  The author owes an enormous debt of gratitude to Professor Joel Robbin who listened carefully over the past year to the author's lectures on various versions of this paper, his many questions helped sharpen many (if not all) of the statements. In particular Theorem \ref{weakII} was suggested by him. The author also thanks Jeff Viaclovsky and Alina Marian for their advice on how to better structure the paper and Gabriele LaNave and Alberto Della Vedova for early conversations on this work.

 \bibliography{ref}
  \end{document}